%% file: PPS-Koiter-2023-R1.tex
\documentclass{amsart}
\usepackage{moreverb,url}
\usepackage{amssymb,amsmath,amsfonts}
\usepackage{bm}
\usepackage[font=scriptsize, labelfont=bf]{caption}
\usepackage{float}
\usepackage{varwidth}
\usepackage{booktabs}
\usepackage{afterpage}
\usepackage{mathrsfs}
\usepackage{subcaption}
\usepackage[squaren]{SIunits}
\usepackage{verbatim}
\usepackage[square,numbers]{natbib}
\usepackage{stmaryrd}
\usepackage{fancyhdr}
\usepackage{syntax,etoolbox}
\usepackage{graphicx}
\usepackage{grffile}
\usepackage{hyperref}
\hypersetup{colorlinks=true,allcolors=blue}
\usepackage[square,numbers]{natbib}
\usepackage[left=1in, right=1in, bottom=1in,top=1in]{geometry}
\usepackage{lineno}
\usepackage{syntax,etoolbox}
\usepackage{algorithmicx}
\usepackage{algorithm}% http://ctan.org/pkg/algorithms
\usepackage{algpseudocode}% http://ctan.org/pkg/algorithmicx
\usepackage{tikz}
\newtheorem{theorem}{Theorem}[section]
\newtheorem{lemma}{Lemma}[section]

\newtheorem{innercustomgeneric}{\customgenericname}
\providecommand{\customgenericname}{}
\newcommand{\newcustomproblem}[2]{%
	\newenvironment{#1}[1]
	{%
		\renewcommand\customgenericname{#2}%
		\renewcommand\theinnercustomgeneric{##1}%
		\innercustomgeneric
	}
	{\endinnercustomgeneric}
}

\newcustomproblem{customprob}{Problem}

\newcommand*{\bqed}{\hfill\ensuremath{\blacksquare}}%

\newcommand{\vertiii}[1]{{\left\vert\kern-0.25ex\left\vert\kern-0.25ex\left\vert #1 
		\right\vert\kern-0.25ex\right\vert\kern-0.25ex\right\vert}}

\def\dd{\, \mathrm{d}}

\algblock{Input}{EndInput}

\algnotext{EndInput}

\algblock{Output}{EndOutput}

\algnotext{EndOutput}

\algblock{Iterate}{EndIterate}

\algnotext{EndIterate}

\allowdisplaybreaks

\begin{document}
	
	%%
	%% The title of the paper goes here.  Edit to your title.
	%%
	
	\title[Numerical methods for Koiter's shells]{Numerical approximation of the solution of Koiter's model for an elliptic membrane shell in absence of friction subjected to an obstacle via the penalty method}
	
	%%
	%% Now edit the following to give your name and address:
	%%

	\author[Xin Peng]{Xin Peng}
	\address{Department of Applied Mathematics, School of Sciences, Xi'an University of Technology, P.O.Box 1243,Yanxiang Road NO.58 XI'AN, 710054, Shaanxi Province, China}
	\email{1615113433@qq.com}
	
	\author[Paolo Piersanti]{Paolo Piersanti}
	\address{Department of Mathematics and Institute for Scientific Computing and Applied Mathematics, Indiana University Bloomington, 729 East Third Street, Bloomington, Indiana, USA}
	\email[Corresponding author]{ppiersan@iu.edu}
	
	\author[Xiaoqin Shen]{Xiaoqin Shen}
	\address{Department of Applied Mathematics, School of Sciences, Xi'an University of Technology, P.O.Box 1243,Yanxiang Road NO.58 XI'AN, 710054, Shaanxi Province, China}
	\email{xqshen@xaut.edu.cn}

\today

\begin{abstract}
This paper is devoted to the analysis of a numerical scheme based on the Finite Element Method for approximating the solution of Koiter's model for a linearly elastic elliptic membrane shell subjected to remaining confined in a prescribed half-space.

First, we show that the solution of the obstacle problem under consideration is uniquely determined and satisfies a set of variational inequalities which are governed by a fourth order elliptic operator, and which are posed over a non-empty, closed, and convex subset of a suitable space.

Second, we show that the solution of the obstacle problem under consideration can be approximated by means of the penalty method.

Third, we show that the solution of the corresponding penalised problem is more regular up to the boundary.

Fourth, we write down the mixed variational formulation corresponding to the penalised problem under consideration, and we show that the solution of the mixed variational formulation is more regular up to the boundary as well.
In view of this result concerning the augmentation of the regularity of the solution of the mixed penalised problem, we are able to approximate the solution of the one such problem by means of a Finite Element scheme.

Finally, we present numerical experiments corroborating the validity of the mathematical results we obtained.

\smallskip


\end{abstract}

\maketitle

\input{paper.tex}

\bibliographystyle{abbrvnat} %unsrt	
\bibliography{references}	

\end{document}

%% file: paper.tex
\tableofcontents

\section{Introduction}
\label{sec0}
This paper is devoted to the numerical analysis of Koiter's model for a linearly elastic elliptic membrane shell subjected to remaining confined in a prescribed half space. The problem under consideration is formulated in terms of a set of variational inequalities posed over a non-empty, closed, and convex subset of a suitable Sobolev space. The differential operator associated with this formulation is a fourth order elliptic operator.
The constraint according to which the shell must remain confined in a prescribed half space renders the problem under consideration an obstacle problem.

In a seminal paper published in~1981, Scholz~\cite{Scholz1987} established the convergence of a Finite Element Method for approximating the solution of a biharmonic model subjected to a double obstacle. The convergence of the numerical scheme was proved for the mixed formulation of the corresponding penalised problem, by resorting to the augmentation of regularity argument for scalar fourth order problems (cf., Lemma~3 of~\cite{Scholz1987}).

The numerical analysis of obstacle problems governed by fourth order elliptic operators has also been studied by Brenner and her collaborators (cf., e.g., the paper~\cite{Brenner2013} and the references therein) in the case where the objective function is scalar-valued. In the previously cited paper, Brenner and her collaborators resorted to Enriching Operators in order to overcome the difficulties related to the fact that the solution of the fourth-order problems they considered was of class $H^2_0$.

The results obtained by Brenner and her collaborators were later on generalised to the vector-valued case in the paper~\cite{PS}, where a linearly elastic shallow shell model subjected to an obstacle was considered. For this particular model, the solution was proved to be a Kirchhoff-Love field~\cite{Leger2008,Leger2010}. The latter geometrical property allows us to \emph{separate} the behaviour of the transverse component of the solution from the behaviour of the tangential components of the solution in the constitutive equations, and to formulate the constraint in terms of the sole transverse component of the vector field under consideration. By so doing, it was thus possible to implement the technique based on Enriching Operators for approximating the solution of the obstacle problem for linearly elastic shallow shells studied in~\cite{PS} via a Finite Element Method.
Critical to establishing the convergence of the numerical scheme is the augmentation-of-regularity result proved in~\cite{Pie2020-1}.

The numerical analysis of an obstacle problem governed by a set of fourth order variational inequalities whose solution is a vector field, and for which the constraint bears at once on all the components of one such vector field is more difficult to study, and the method based on Enriching Operators is not applicable. To-date, the contrivance of a Finite Element scheme for approximating the solution of a fourth order variational problem whose solution is a vector field appears to be virgin research territory. The purpose of this paper is to make progress in this specific direction.

This paper is divided into ten sections (including this one). In section~\ref{sec1} we lay out the analytical and geometrical notation.

In section~\ref{sec2} we formulate a three-dimensional obstacle problem for a ``general'' linearly elastic shell of Koiter's type.

In section~\ref{sec3} we specialise the type of linearly elastic shells under consideration, by restricting ourselves to the case where a boundary condition of place is imposed on the entire lateral contour, and the middle surface is elliptic. In this case, the linearly elastic shell under consideration becomes a linearly elastic elliptic membrane shell. We then recall that the solution for this  problem satisfies a set of variational inequalities associated with a fourth order elliptic operator.

In section~\ref{sec:penalty}, we penalise the formulation devised in the previous section, as it is well-known that a penalised model is easier to address numerically.

In section~\ref{sec:aug-interior}, we establish the augmentation of regularity up to the boundary of the penalised problem. The obtained higher regularity will turn out to be critical for devising a convergent Finite Element Method for the problem under consideration.

In section~\ref{approx:original} we present an alternative formulation of the variational inequalities considered in section~\ref{sec3}.
The alternative formulation we are here presenting is given in terms of the Lagrangian functional associated with the problem stated in section~\ref{sec3}. We show that, if the Lagrangian functional we consider admits at least one saddle point, the first component of such saddle points is uniquely determined and coincides with the solution of the variational inequalities stated in section~\ref{sec3}. We point out that this section is ancillary,and the final results concerning the convergence of the Finite Element scheme analysed in the remainder of the paper are independent of the results established in section~\ref{approx:original}.

In section~\ref{mixed:penalty} we resort to the intrinsic theory devised by Blouza \& Le Dret as to define a new penalised problem which, this time, takes into account two layers of penalisation: one for relaxing the enforcement of the geometrical constraint, and one for relaxing the regularity requirement on the transverse component of the solution.
Thanks to the results established by Blouza and his collaborators in~\cite{Blouza2016}, we note that the new penalised variational problem coincides with the mixed formulation of the penalised problem introduced in section~\ref{sec:penalty}. We note in passing that, differently from~\cite{Scholz1987}, we exploit the augmented regularity in a different way as the vectorial nature of the problem prevents from straightforwardly applying the estimates that led to the conclusion in~\cite{Scholz1987}.

In section~\ref{approx:penalty}, we establish the convergence of the solution of the discrete version of the doubly penalised problem introduced in section~\ref{mixed:penalty} by means of a Finite Element Method.

Finally, in section~\ref{numerics} we present numerical experiments meant to corroborate the theoretical results established in the previous sections.

\section{Background and notation}
\label{sec1}

For a complete overview about the classical notions of differential geometry used in this paper, see, e.g.~\cite{Ciarlet2000} or \cite{Ciarlet2005}.

Greek indices, except $\varepsilon$ and $\kappa$, vary in the set $\{1,2\}$. Latin indices, except when they are used for ordering sequences, vary in the set $\{1,2,3\}$. Unless differently specified, the Einstein summation convention with respect to repeated indices is used in conjunction with the previous two rules.
 
We take a \emph{real three-dimensional affine Euclidean space}, i.e., a set in which a point $O \in\mathbb{R}^3$ has been chosen as the \emph{origin} and with which a \emph{real three-dimensional Euclidean space}, as a model of the three-dimensional ``physical'' space $\mathbb{R}^3$. We denote the three-dimensional affine Euclidean space by $\mathbb{E}^3$. The space $\mathbb{E}^3$ is equipped with an \emph{orthonormal basis} consisting of three vectors $\bm{e}^i$, whose components are denoted by $e^i_j=\delta^i_j$. 

Defining $\mathbb{R}^3$ as an affine Euclidean space means that any point $x \in \mathbb R^3$ corresponds to a uniquely determined vector $\boldsymbol{Ox} \in \mathbb{E}^3$. The origin $O \in \mathbb{R}^3$ and the orthonormal vectors $\bm{e}^i \in \mathbb{E}^3$ form a \emph{Cartesian frame} in $\mathbb{R}^3$.

After choosing a Cartesian frame,we could identify any point $x \in \mathbb{R}^3$ with the vector $\boldsymbol{Ox}=x_i \bm{e}^i \in \mathbb{E}^3$. The three components $x_i$ of the vector $\boldsymbol{Ox}$ expressed in terms of the orthonormal frame $\{\bm{e}^i\}_{i=1}^3$ are the \emph{Cartesian coordinates} of $x \in \mathbb{R}^3$, or the \emph{Cartesian components} of $\boldsymbol{Ox} \in \mathbb{E}^3$. 

In view of the previous association, we are allowed to identify a set in $\mathbb{R}^3$ with a ``physical'' body in the Euclidean space $\mathbb{E}^3$.

The standard inner product and vector product of any two vectors $\bm{u}, \bm{v} \in \mathbb{E}^3$ are respectively denoted by $\bm{u} \cdot \bm{v}$ and $\bm{u} \wedge \bm{v}$; the Euclidean norm of a vector $\bm{u} \in \mathbb{E}^3$ is denoted by $\left|\bm{u} \right|$. The notation $\delta^j_i$ designates the standard Kronecker symbol.

Let $\Omega$ be an open subset of $\mathbb{R}^N$, where $N \ge 1$.
Denote the usual Lebesgue and Sobolev spaces by $L^2(\Omega)$, $L^1_{\textup{loc}}(\Omega)$, $H^1(\Omega)$, $H^1_0 (\Omega)$, $H^1_{\textup{loc}}(\Omega)$. The notation $\mathcal{D}(\Omega)$ indicates the space of all functions that are infinitely differentiable over $\Omega$ and have compact support in $\Omega$. The symbol $\left\| \cdot \right\|_X$ denotes the norm of a normed vector space $X$. Spaces of vector-valued functions are written in boldface.
The Euclidean norm of any point $x \in \Omega$ is denoted by $|x|$. 

The boundary $\Gamma$ of an open subset $\Omega$ in $\mathbb{R}^N$ is said to be Lipschitz-continuous if the following conditions are satisfied (cf., e.g., Section~1.18 of~\cite{PGCLNFAA}): Given an integer $s\ge 1$, there exist constants $\alpha_1>0$ and $L>0$, a finite number of local coordinate systems, with coordinates 
$$
\bm{\phi}'_r=(\phi_1^r, \dots, \phi_{N-1}^r) \in \mathbb{R}^{N-1} \textup{ and } \phi_r=\phi_N^r, 1 \le r \le s,
$$ 
sets
$$
\tilde{\omega}_r:=\{\bm{\phi}_r \in\mathbb{R}^{N-1}; |\bm{\phi}_r|<\alpha_1\},\quad 1 \le r \le s,
$$
and corresponding functions
$$
\tilde{\theta}_r:\tilde{\omega}_r \to\mathbb{R},\quad 1 \le r \le s,
$$
such that
$$
\Gamma=\bigcup_{r=1}^s \{(\bm{\phi}'_r,\phi_r); \bm{\phi}'_r \in \tilde{\omega}_r \textup{ and }\phi_r=\tilde{\theta}_r(\bm{\phi}'_r)\},
$$
and 
$$
|\tilde{\theta}_r(\bm{\phi}'_r)-\tilde{\theta}_r(\bm{\upsilon}'_r)|\le L |\bm{\phi}'_r-\bm{\upsilon}'_r|, \textup{ for all }\bm{\phi}'_r, \bm{\upsilon}'_r \in \tilde{\omega}_r, \textup{ and all }1\le r\le s.
$$

We observe that the second last formula takes into account overlapping local charts, while the last set of inequalities expresses the Lipschitz continuity of the mappings $\tilde{\theta}_r$.

An open set $\Omega$ is said to be locally on the same side of its boundary $\Gamma$ if, in addition, there exists a constant $\alpha_2>0$ such that
\begin{align*}
	\{(\bm{\phi}'_r,\phi_r);\bm{\phi}'_r \in\tilde{\omega}_r \textup{ and }\tilde{\theta}_r(\bm{\phi}'_r) < \phi_r < \tilde{\theta}_r(\bm{\phi}'_r)+\alpha_2\}&\subset \Omega,\quad\textup{ for all } 1\le r\le s,\\
	\{(\bm{\phi}'_r,\phi_r);\bm{\phi}'_r \in\tilde{\omega}_r \textup{ and }\tilde{\theta}_r(\bm{\phi}'_r)-\alpha_2 < \phi_r < \tilde{\theta}_r(\bm{\phi}'_r)\}&\subset \mathbb{R}^N\setminus\overline{\Omega},\quad\textup{ for all } 1\le r\le s.
\end{align*}

A \emph{domain in} $\mathbb{R}^N$ is a bounded and connected open subset $\Omega$ of $\mathbb{R}^N$, whose boundary $\partial \Omega$ is Lipschitz-continuous, the set $\Omega$ being locally on a single side of $\partial \Omega$.

Let $\omega$ be a domain in $\mathbb{R}^2$ with boundary $\gamma:=\partial\omega$,  and let $\omega_1 \subset \omega$. The notation $\omega_1 \subset \subset \omega$ means that $\overline{\omega_1} \subset \omega$ and $\textup{dist}(\gamma,\partial\omega_1):=\min\{|x-y|;x \in \gamma \textup{ and } y \in \partial\omega_1\}>0$.

Let $y = (y_\alpha)$ denote a generic point in $\omega$, and let $\partial_\alpha := \partial / \partial y_\alpha$. An injective mapping $\bm{\theta} \in \mathcal{C}^1(\overline{\omega}; \mathbb{E}^3)$ is called an \emph{immersion} if the two vectors
$$
\bm{a}_\alpha (y) := \partial_\alpha \bm{\theta}(y)
$$
are linearly independent at each point $y \in \overline{\omega}$. Then the set $\bm{\theta} (\overline{\omega})$ is a \emph{surface in} $\mathbb{E}^3$, equipped with $y_1, y_2$ as its \emph{curvilinear coordinates}. Given any point $y\in \overline{\omega}$, the linear combinations of the vectors $\bm{a}_\alpha (y)$ generate the \emph{tangent plane} to the surface $\bm{\theta}(\overline{\omega})$ at the point $\bm{\theta} (y)$. The unit-norm vector
$$
\bm{a}_3 (y) := \frac{\bm{a}_1(y) \wedge \bm{a}_2 (y)}{|\bm{a}_1(y) \wedge \bm{a}_2 (y)|}
$$
is orthogonal to the surface $\bm{\theta} (\overline{\omega})$ at the point $\bm{\theta}(y)$. The three vectors $\{\bm{a}_i(y)\}_{i=1}^3$ constitute the \emph{covariant} basis at the point $\bm{\theta}(y)$, while the three vectors $\{\bm{a}^j(y)\}_{j=1}^3$ defined by means of the formulas
$$
\bm{a}^j(y) \cdot \bm{a}_i (y) = \delta^j_i
$$
constitute the \emph{contravariant} basis at $\bm{\theta}(y)$. Observe that the vectors $\bm{a}^\beta(y)$ generate the tangent plane to the surface $\bm{\theta}(\overline{\omega})$ at the point $\bm{\theta}(y)$ and that $\bm{a}^3(y) = \bm{a}_3(y)$.

The \emph{first fundamental form} of the surface $\bm{\theta}(\overline{\omega})$ can be thus expressed via its \emph{covariant components}
$$
a_{\alpha\beta} := \bm{a}_\alpha \cdot \bm{a}_\beta = a_{\beta \alpha} \in \mathcal{C}^0 (\overline{\omega}),
$$
or via its \emph{contravariant components}
$$
a^{\alpha\beta}:= \bm{a}^\alpha \cdot \bm{a}^\beta = a^{\beta \alpha}\in \mathcal{C}^0(\overline{\omega}).
$$

Observe that the symmetric matrix field $(a^{\alpha\beta})$ coincides with the inverse of the positive-definite matrix field $(a_{\alpha\beta})$.
Moreover, we have that $\bm{a}^\beta = a^{\alpha\beta}\bm{a}_\alpha$, that $\bm{a}_\alpha = a_{\alpha\beta} \bm{a}^\beta$, and that the \emph{area element} associated with the surface $\bm{\theta}(\overline{\omega})$ is given at each point $\bm{\theta}(y), y \in \overline{\omega}$, by $\sqrt{a(y)} \dd y$, where
$$
a := \det(a_{\alpha\beta}) \in \mathcal{C}^0(\overline{\omega})
$$
is such that $a_0 \le a(y) \le a_1$, for all $y\in\overline{\omega}$ for some $a_0, a_1>0$.

Let $\bm{\theta} \in \mathcal{C}^2(\overline{\omega};\mathbb{E}^3)$ be an immersion. The \emph{second fundamental form} of the surface $\bm{\theta}(\overline{\omega})$ can be defined via its \emph{covariant components} by means of the following formula:
$$
b_{\alpha\beta}:= \partial_\alpha \bm{a}_\beta \cdot \bm{a}_3 = -\bm{a}_\beta \cdot \partial_\alpha \bm{a}_3 = b_{\beta \alpha} \in \mathcal{C}^0(\overline{\omega}),
$$

An alternative way of defining the second fundamental form of the surface $\bm{\theta}(\overline{\omega})$ is via its \emph{mixed components}:
$$
b^\beta_\alpha := a^{\beta \sigma} b_{\alpha \sigma} \in \mathcal{C}^0(\overline{\omega}).
$$

The \emph{Christoffel symbols} associated with the immersion $\bm{\theta}$ are defined by:
$$
\Gamma^\sigma_{\alpha\beta}:= \partial_\alpha \bm{a}_\beta \cdot \bm{a}^\sigma = \Gamma^\sigma_{\beta \alpha} \in \mathcal{C}^0 (\overline{\omega}).
$$

For each $y \in \overline \omega$, the \emph{Gaussian curvature} (or \emph{total curvature}) at each point $\bm{\theta}(y)$, of the surface $\bm{\theta} (\overline{\omega})$ is defined by:
$$
K(y) := \frac{\det (b_{\alpha\beta} (y))}{\det (a_{\alpha\beta} (y))} = \det \left( b^\beta_\alpha (y)\right).
$$

Observe that the denominator in the definition of Gaussian curvature at a point never vanishes, since the mapping $\bm{\theta}$ is assumed to be an immersion. Additionally, observe that the value of the Gaussian curvature at the point $\bm{\theta}(y)$, $y\in\overline{\omega}$, coincides with the product of the two principal curvatures at this point.

Let us now recall the definition of \emph{elliptic surface}: Let $\omega$ be a domain in $\mathbb{R}^2$. Then a surface $\bm{\theta}(\overline{\omega})$ defined by means of an immersion $\bm{\theta} \in \mathcal{C}^2(\overline{\omega};\mathbb{E}^3)$ is said to be \emph{elliptic} if its Gaussian curvature $K$ is everywhere strictly positive in $\overline{\omega}$, or equivalently, if there exists a constant $K_0$ such that:
$$
0 < K_0 \le K (y), \text{ for all } y \in \overline{\omega}.
$$

Let
$\bm{\theta} \in \mathcal{C}^2(\overline{\omega}; \mathbb{E}^3)$ be an immersion. Given a
vector field $\bm{\eta} = (\eta_i) \in \mathcal{C}^1(\overline{\omega};
\mathbb{R}^3)$, one can regard the vector field
$\tilde{\bm{\eta}} := \eta_i \bm{a}^i$
as a \emph{displacement field of the surface} $\bm{\theta}(\overline{\omega})$ defined by means of its \emph{covariant components} $\eta_i$ over the vectors $\bm{a}^i$ of the contravariant basis along the surface. 

If the norms $\left\| \eta_i \right\|_{\mathcal{C}^1(\overline{\omega})}$ are sufficiently small, the mapping $(\bm{\theta} + \eta_i \bm{a}^i) \in \mathcal{C}^1(\overline{\omega}; \mathbb{E}^3)$ is also an immersion. Therefore, the set $(\bm{\theta} + \eta_i \bm{a}^i) (\overline{\omega})$, called the \emph{deformed surface} corresponding to the displacement field $\tilde{\bm{\eta}} = \eta_i \bm{a}^i$, is also a surface in $\mathbb{E}^3$, equipped with the same curvilinear coordinate system as the one of the surface $\bm{\theta}(\overline{\omega})$. This alternative \emph{intrinsic} formulation, which is due to Blouza \& Le Dret~\cite{BlouzaLeDret1999} considers the \emph{whole} displacement at once, rather than focusing on the, single covariant components.

It is thus possible to define the first fundamental form of the deformed surface in terms of its covariant components
$$
a_{\alpha\beta} (\bm{\eta}) := (\bm{a}_\alpha + \partial_\alpha \tilde{\bm{\eta}}) \cdot (\bm{a}_\beta + \partial_\beta \tilde{\bm{\eta}}),
$$
and it is also possible to define the second fundamental form of the deformed surface via its covariant components
$$
b_{\alpha\beta} (\bm{\eta}) := \partial_\alpha (\bm{a}_\beta+\partial_\beta \tilde{\bm{\eta}}) \cdot \dfrac{(\bm{a}_1 +\partial_1 \tilde{\bm{\eta}}) \wedge (\bm{a}_2 +\partial_2 \tilde{\bm{\eta}})}{|(\bm{a}_1 +\partial_1 \tilde{\bm{\eta}}) \wedge (\bm{a}_2 +\partial_2 \tilde{\bm{\eta}})|}.
$$

The \emph{linear part with  respect to} $\tilde{\bm{\eta}}=\eta_i \bm{a}^i$ in the difference $\frac{1}{2}(a_{\alpha\beta}(\bm{\eta}) - a_{\alpha\beta})$ is referred to as \emph{linearised change of metric tensor} (or \emph{linearised strain tensor}) associated with the displacement vector field $\tilde{\bm{\eta}}=\eta_i \bm{a}^i$.
The covariant components of the linearised change of metric tensor are then given by the following formula (viz. Lemma~2 of~\cite{BlouzaLeDret1999}):
\begin{equation}
\label{cmt}
\gamma_{\alpha\beta}(\bm{\eta}) := \dfrac{1}{2}(\bm{a}_\alpha \cdot \partial_\beta \tilde{\bm{\eta}} + \partial_\alpha \tilde{\bm{\eta}} \cdot \bm{a}_\beta) = \frac12 (\partial_\beta \eta_\alpha + \partial_\alpha \eta_\beta ) - \Gamma^\sigma_{\alpha\beta} \eta_\sigma - b_{\alpha\beta} \eta_3 = \gamma_{\beta \alpha} (\bm{\eta}).
\end{equation}

The \emph{linear part with  respect to} $\tilde{\bm{\eta}}=\eta_i \bm{a}^i$ in the difference $\frac{1}{2}(b_{\alpha\beta}(\bm{\eta}) - b_{\alpha\beta})$ is known as the \emph{linearised change of curvature tensor} associated with the displacement field  $\tilde{\bm{\eta}}=\eta_i \bm{a}^i$. The covariant components of the linearised change of curvature tensor are then given by the following formula (viz. Lemma~2 of~\cite{BlouzaLeDret1999}):
\begin{equation}
\label{cct}
\begin{aligned}
	\rho_{\alpha\beta}(\bm{\eta}) &= (\partial_{\alpha\beta}\tilde{\bm{\eta}}-\Gamma_{\alpha\beta}^\sigma \partial_\sigma \tilde{\bm{\eta}})\cdot \bm{a}_3\\
	&= \partial_{\alpha\beta}\eta_3 -\Gamma_{\alpha\beta}^\sigma \partial_\sigma\eta_3 -b_\alpha^\sigma b_{\sigma\beta}\eta_3
	+b_\alpha^\sigma(\partial_\beta \eta_\sigma-\Gamma_{\beta\sigma}^\tau\eta_\tau)+b_\beta^\tau(\partial_\alpha\eta_\tau-\Gamma_{\alpha\tau}^\sigma\eta_\sigma)\\
	&\quad+(\partial_\alpha b_\beta^\tau+\Gamma_{\alpha\sigma}^\tau b_\beta^\sigma-\Gamma_{\alpha\beta}^\sigma b_\sigma^\tau)\eta_\tau= \rho_{\beta \alpha}(\bm{\eta}).
\end{aligned}
\end{equation}

\section{An obstacle problem for a ``general'' Koiter's shell} 
\label{sec2}

Consider a domain $\omega$ in $\mathbb{R}^2$, and recall that $\gamma:= \partial \omega$. Let $\gamma_0$ be a non-empty relatively open subset of $\gamma$. For each $\varepsilon > 0$, define the sets:
$$
\Omega^\varepsilon = \omega \times \left] - \varepsilon , \varepsilon \right[ \quad\textup{ and }\quad \Gamma^\varepsilon_0 := \gamma_0 \times \left] - \varepsilon , \varepsilon \right[.
$$

By $x^\varepsilon = (x^\varepsilon_i)$, we designate an arbitrary point in the set $\overline{\Omega^\varepsilon}$. Define the symbol $\partial^\varepsilon_i := \partial / \partial x^\varepsilon_i$. It is convenient to write the coordinates of a generic point $x^\varepsilon\in\Omega^\varepsilon$ as $x^\varepsilon_\alpha = y_\alpha$ and $\partial^\varepsilon_\alpha = \partial_\alpha$.

Let $\bm{\theta}\in\mathcal{C}^3(\overline{\omega};\mathbb{E}^3)$ be an immersion, and let $\varepsilon>0$ denote the constant half-thickness of a \emph{linearly elastic shell} with \emph{middle surface} $\bm{\theta}(\overline{\omega})$. The \emph{reference configuration} of the shell corresponds to the set $\bm{\Theta}(\overline{\Omega^\varepsilon})$, where the mapping $\bm{\Theta}: \overline{\Omega^\varepsilon} \to \mathbb{E}^3$ is given by:
$$
\bm{\Theta} (x^\varepsilon) := \bm{\theta} (y) + x^\varepsilon_3 \bm{a}^3(y) \text{ at each point } x^\varepsilon = (y, x^\varepsilon_3) \in \overline{\Omega^\varepsilon}.
$$

If $\varepsilon > 0$ is small enough, an application of the implicit function theorem shows that such a mapping $\bm{\Theta} \in \mathcal{C}^2(\overline{\Omega^\varepsilon}; \mathbb{E}^3)$ is an \emph{immersion}, in the sense that the three vectors
$$
\bm{g}^\varepsilon_i (x^\varepsilon) := \partial^\varepsilon_i \bm{\Theta} (x^\varepsilon)
$$
are linearly independent at each point $x^\varepsilon \in \overline{\Omega^\varepsilon}$ (viz., e.g., Theorem~3.1-1 of~\cite{Ciarlet2000}).
These vectors  then constitute the \emph{covariant basis} at the point $\bm{\Theta}(x^\varepsilon)$, while the three vectors $\bm{g}^{j, \varepsilon} (x^\varepsilon)$ defined by the relations
$$
\bm{g}^{j,\varepsilon}(x^\varepsilon) \cdot \bm{g}^\varepsilon_i(x^\varepsilon)= \delta^j_i
$$
constitute the \emph{contravariant basis} at the same point. It will be implicitly assumed in the sequel that $\varepsilon > 0$ \emph{is small enough so to apply Theorem~3.1-1 of~\cite{Ciarlet2000} and to infer that $\bm{\Theta}:\overline{\Omega^\varepsilon} \to \mathbb{E}^3$} is an \emph{immersion}.

It will also be implicitly assumed that the immersion $\bm{\theta}\in \mathcal{C}^3(\overline{\omega};\mathbb{E}^3)$ is \emph{injective} and that $\varepsilon > 0$ \emph{is small enough so that $\bm{\theta} : \overline{\Omega^\varepsilon} \to \mathbb{E}^3$} is \emph{a $\mathcal{C}^2$-diffeomorphism onto its image}.

We also assume that the shell is made of a \emph{homogeneous} and \emph{isotropic linearly elastic material} and that its reference configuration $\bm{\theta}(\overline{\Omega^\varepsilon})$ is a \emph{natural state}, in the sense that it is stress free. As a result of these assumptions, the elastic behaviour of this elastic material is completely characterised by its two \emph{Lam\'e constants} $\lambda \ge 0$ and $\mu >0$ (see, e.g., Section~3.8 in~\cite{Ciarlet1988}).

We also assume that the linearly elastic shell under consideration is subjected to \emph{applied body forces} whose density per unit volume is defined by means of its \emph{covariant} components $f^{i, \varepsilon} \in L^2(\Omega^\varepsilon)$, i.e., over the vectors $\bm{g}_i^\varepsilon$ of the covariant bases, and to a \emph{homogeneous boundary condition of place} along the portion $\Gamma^\varepsilon_0$ of its lateral face, i.e., the admissible displacement fields vanish on $\Gamma^\varepsilon_0$.

A commonly used \emph{two-dimensional} set of equations for modelling such a shell (``two-dimensional'' in the sense that it is posed over $\overline{\omega}$ instead of $\overline{\Omega^\varepsilon}$) was proposed in 1970 by Koiter~\cite{Koiter1970}. We now describe the modern formulation of this model.

First, define the space
$$
\bm{V}_K (\omega):= \{\bm{\eta}=(\eta_i) \in H^1(\omega)\times H^1(\omega)\times H^2(\omega);\eta_i=\partial_{\nu}\eta_3=0 \textup{ on }\gamma_0\},
$$
where the symbol $\partial_{\nu}$ denotes the \emph{outer unit normal derivative operator along} $\gamma$, and define the norm $\|\cdot\|_{\bm{V}_K(\omega)}$ by
$$
\|\bm{\eta}\|_{\bm{V}_K(\omega)}:=\left\{\sum_{\alpha}\|\eta_\alpha\|_{1,\omega}^2+\|\eta_3\|_{2,\omega}^2\right\}^{1/2}\quad\mbox{ for each }\bm{\eta}=(\eta_i)\in \bm{V}_K(\omega).
$$

Next, define the \emph{fourth-order two-dimensional elasticity tensor of the shell}, viewed here as a two-dimensional linearly elastic body, by means of its contravariant components
$$
a^{\alpha\beta\sigma\tau} := \frac{4\lambda \mu}{\lambda + 2 \mu} a^{\alpha\beta} a^{\sigma\tau} + 2\mu \left(a^{\alpha \sigma} a^{\beta \tau} + a^{\alpha \tau} a^{\beta \sigma}\right).
$$

Finally, define the bilinear forms $B_M(\cdot,\cdot)$ and $B_F(\cdot,\cdot)$ by
\begin{align*}
B_M(\bm{\xi}, \bm{\eta})&:=\int_\omega a^{\alpha\beta\sigma\tau} \gamma_{\sigma\tau}(\bm{\xi}) \gamma_{\alpha\beta}(\bm{\eta}) \sqrt{a} \dd y,\\
B_F(\bm{\xi}, \bm{\eta})&:=\dfrac13 \int_\omega a^{\alpha\beta\sigma\tau} \rho_{\sigma\tau}(\bm{\xi}) \rho_{\alpha\beta}(\bm{\eta}) \sqrt{a} \dd y,
\end{align*}
for each $\bm{\xi}=(\xi_i)\in \bm{V}_K(\omega)$ and each $\bm{\eta}=(\eta_i) \in \bm{V}_K(\omega)$.
and define the linear form $\ell^\varepsilon$ by
$$
\ell^\varepsilon(\bm{\eta}):=\int_\omega p^{i,\varepsilon} \eta_i \sqrt{a} \dd y, \textup{ for each } \bm{\eta}=(\eta_i) \in \bm{V}_K(\omega), 
$$
where $p^{i,\varepsilon}(y):=\int_{-\varepsilon} ^\varepsilon f^{i,\varepsilon}(y,x_3) \dd x_3$ at each $y \in \omega$.

Then the \emph{total energy} of the shell is the \emph{quadratic functional} $J^\varepsilon:\bm{V}_K(\omega) \to \mathbb{R}$ defined by
$$
J^\varepsilon(\bm{\eta}):=\dfrac{\varepsilon}{2}B_M(\bm{\eta},\bm{\eta})+\dfrac{\varepsilon^3}{2}B_F(\bm{\eta},\bm{\eta})-\ell^\varepsilon(\bm{\eta})\quad\textup{ for each }\bm{\eta} \in \bm{V}_K(\omega).
$$
The term $\displaystyle{\dfrac{\varepsilon}{2}B_M(\cdot,\cdot)}$ and $\displaystyle{\dfrac{\varepsilon^3}{2}B_F(\cdot,\cdot)}$ respectively represent the \emph{membrane part} and the \emph{flexural part} of the total energy, as aptly recalled by the subscripts ``$M$'' and ``$F$''.

In Koiter's model, the unknown ``two-dimensional'' displacement field $\zeta^\varepsilon_i \bm{a}^i:\overline{\omega} \to \mathbb{E}^3$ of the middle surface $\bm{\theta}(\overline{\omega})$ of the shell is such that the vector field $\bm{\zeta}^\varepsilon:=(\zeta^\varepsilon_i)$ should be the solution of the following \emph{problem}: \emph{Find a vector field $\bm{\zeta}^\varepsilon:\overline{\omega} \to \mathbb{R}^3$ that satisfies
	$$
	\bm{\zeta}^\varepsilon \in \bm{V}_K(\omega) \quad\textup{ and }\quad J^\varepsilon(\bm{\zeta}^\varepsilon)=\inf_{\bm{\eta} \in \bm{V}_K(\omega)} J^\varepsilon(\bm{\eta}),
	$$
	\emph{or equivalently}, find $\bm{\zeta}^\varepsilon \in \bm{V}_K(\omega)$ that satisfies the following variational equations:
	$$
	\varepsilon B_M(\bm{\zeta}^\varepsilon, \bm{\eta})+\varepsilon^3 B_F(\bm{\zeta}^\varepsilon, \bm{\eta})=\ell^\varepsilon(\bm{\eta}) \quad\textup{ for all }\bm{\eta} \in \bm{V}_K(\omega).
	$$}
As first shown in~\cite{BerCia1976} (see also~\cite{BerCiaMia1994}), this problem has one and only one solution.

Assume now that the shell is subjected to the following \emph{confinement condition}: The unknown displacement field $\bm{\zeta}^\varepsilon_i \bm{a}^i$ of the middle surface of the shell must be such that the corresponding deformed middle surface remains in a given half-space, of the form
$$
\mathbb{H}:=\{x \in \mathbb{E}^3; \boldsymbol{Ox} \cdot \bm{q} \ge 0\},
$$
where $\bm{q} \in \mathbb{E}^3$ is a given unit-norm vector. Equivalently, the deformed middle surface must constantly remain in one of the half-spaces defined by the plane $\{x \in \mathbb{E}^3; \boldsymbol{Ox} \cdot \bm{q} \ge 0\}$, which thus plays the role of the obstacle. Clearly, one will need to assume that the middle surface satisfies this confinement condition when no forces are applied: 
$$
\bm{\theta}(\overline{\omega}) \subset \mathbb{H}.
$$

It is to be emphasised that the above confinement condition \emph{considerably departs} from the usual \emph{Signorini condition} favoured by most authors, who usually require that only the points of the undeformed and deformed lower face $\omega \times \{-\varepsilon\}$ of the reference configuration satisfy the confinement condition (see, e.g., \cite{Leger2008}, \cite{LegMia2018}, \cite{MezChaBen2020}, \cite{Rodri2018}). 

The confinement condition we are considering in this paper is more physically realistic than the Signorini condition, which imposes the constraint only on the lower face of the reference configuration and does not thus prevent, \emph{a priori}, the interior points of the deformed reference configuration from crossing the plane $\{\bm{Ox}\in \mathbb{E}^3; \; \bm{Ox} \cdot \bm{q} = 0\}$ constituting the obstacle. 

The mathematical models characterised by the confinement condition introduced beforehand, confinement condition which was originally considered in the seminar paper by Brezis \& Stampacchia~\cite{BrezisStampacchia1968} and is also considered in the seminal paper~\cite{Leger2008} in a different geometrical framework, do not take any traction forces into account. Indeed, by Classical Mechanics, there could be no traction forces applied to the portion of the three-dimensional shell boundary that engages contact with the obstacle. 

Friction is not considered in the context of this analysis.

Frictionless contact problems arise in mechanical engineering whenever there is need to predict the stress or deformation of the moving parts of machines or vehicles. The first example of frictionless contact problem in elasticity was investigated in~1881 by Hertz. Hertz noticed that when two elastic bodies with smooth surface come into contact and the contact area is much smaller than the radius of each body (as in~\cite{PGCFEM}, by \emph{radius of a body}, we define the radius of the largest ball contained in an elastic body), then the nonlinear non-penetration boundary conditions are confined to a small region of predictable shape and it is possible to simplify the conditions of equilibrium near the contact so that they can be solved analytically. Hertz’s results are still relevant for the design of bearings, gears, and other bodies when two smooth and nonconforming surfaces come into contact, the strains are small and within the elastic limit, the area of contact is much smaller than the characteristic radius of the bodies, and the friction can be neglected. For a more detailed applicability of frictionless contact problems to Industrial and Engineering problems, we refer the reader to Chapter~1 of~\cite{DKSV17}. Examples of frictionless contact problems in virus mechanics were recently considered in the papers~\cite{PWDT2D,PWDT3D}. Frictionless contact problems were also studied in the context of glaciology in the papers~\cite{Diaz2002,JouvBuel2012,PT2023}.

Differently from the \emph{Signorni condition}, the confinement condition here considered is more suitable in the context of multi-scales multi-bodies problems like, for instance, the study of the motion of the human heart valves, conducted by Quarteroni and his associates in~\cite{Quarteroni2021-3,Quarteroni2021-2,Quarteroni2021-1} and the references therein.

Such a confinement condition renders the study of this problem considerably more difficult, however, as the constraint now bears on a vector field, the displacement vector field of the reference configuration, instead of on only a single component of this field.

The variational problem modelling the deformation of a Koiter shell subjected to the confinement condition according to which all the points of the deformed reference configuration have to remain confined in a prescribed half space takes the following form (viz. \cite{CiaPie2018b}).

\begin{customprob}{$\mathcal{P}_K^\varepsilon(\omega)$}
	\label{problemK}
	Find 
	$\bm{\zeta}^\varepsilon \in \bm{U}_K(\omega):=\{\bm{\eta}=(\eta_i)\in \bm{V}_K(\omega);(\bm{\theta}(y)+\eta_i(y)\bm{a}^i(y)) \cdot \bm{q} \ge 0 \textup{ for a.a. }y \in \omega\}$ satisfying the following variational inequalities:
	\begin{equation*}
	\varepsilon\int_\omega a^{\alpha\beta\sigma\tau} \gamma_{\sigma\tau}(\bm{\zeta}^\varepsilon) \gamma_{\alpha\beta} (\bm{\eta} - \bm{\zeta}^\varepsilon) \sqrt{a} \dd y
	+\dfrac{\varepsilon^3}{3} \int_\omega a^{\alpha\beta\sigma\tau} \rho_{\sigma\tau}(\bm{\zeta}^\varepsilon) \rho_{\alpha\beta} (\bm{\eta} - \bm{\zeta}^\varepsilon) \sqrt{a} \dd y \ge \int_\omega p^{i,\varepsilon} (\eta_i - \zeta^\varepsilon_i) \sqrt{a} \dd y,
	\end{equation*}
	for all $\bm{\eta} = (\eta_i) \in \bm{U}_K(\omega)$, where $p^{i,\varepsilon}:=\int_{-\varepsilon}^{\varepsilon} f^{i,\varepsilon} \dd x_3$.
	\bqed
\end{customprob}

This variational problem admits one and only one solution (cf., e.g., Theorem~2.1 of~\cite{CiaPie2018b}). This uniqueness result hinges on the following inequality of Korn's type on a general surface. Recall that the expression of the covariant components of the linearised change of metric tensor and the linearised change of curvature tensor have been recalled in~\eqref{cmt} and~\eqref{cct}.

\begin{theorem}[Theorem~2.6-4 in~\cite{Ciarlet2000}]
\label{kornsurface}
Let $\omega$ be a domain in $\mathbb{R}^2$, let $\bm{\theta}\in \mathcal{C}^3(\overline{\omega};\mathbb{E}^3)$ be an injective mapping such that the two vectors $\bm{a}_\alpha:=\partial_\alpha \bm{\theta}$ are linearly independent at all the points $y\in\overline{\omega}$, let $\gamma_0$ be a $\dd\gamma$-measurable subset of $\gamma:=\partial\omega$ that satisfies $\textup{length }\gamma_0>0$. Consider the space
$$
\bm{V}_K (\omega):= \{\bm{\eta}=(\eta_i) \in H^1(\omega)\times H^1(\omega)\times H^2(\omega);\eta_i=\partial_{\nu}\eta_3=0 \textup{ on }\gamma_0\}.
$$

Then, there exists a constant $c_0=c_0(\omega,\gamma_0,\bm{\theta})>0$ such that
$$
\left\{\sum_\alpha\|\eta_\alpha\|_{H^1(\omega)}^2+\|\eta_3\|_{H^2(\omega)}^2\right\}^{1/2} \le c_0\left\{\sum_{\alpha, \beta}\|\gamma_{\alpha\beta}(\bm{\eta})\|_{L^2(\omega)}^2+\sum_{\alpha, \beta}\|\rho_{\alpha\beta}(\bm{\eta})\|_{L^2(\omega)}^2\right\}^{1/2},
$$
for all $\bm{\eta}=(\eta_i) \in \bm{V}_K(\omega)$.
\qed
\end{theorem}

The first remarkable feature we encounter in the formulation of Problem~\ref{problemK} is that the transverse component of the displacement vector field $\bm{\zeta}^\varepsilon$ is of class $H^2(\omega)$. Since $\omega$ is a domain in $\mathbb{R}^2$, a direct application of the Rellich-Kondra\v{s}ov Theorem (cf., e.g., Theorem~6.6-3 of~\cite{PGCLNFAA}) gives that $\zeta^\varepsilon_3 \in \mathcal{C}^0(\overline{\omega})$.

This kind of setting, however, is not amenable for studying the numerical approximation of the solution of Problem~\ref{problemK}.
Indeed, it is not possible, in general, to reproduce the argument in~\cite{Brenner2013,PS} as, this time, the constraint bears on all the components of the displacement vector field at once.

The formulation based on displacement vector fields of the form $\tilde{\bm{\eta}}=\eta_i\bm{a}^i$ suggested by Bluza \& Le Dret~\cite{BlouzaLeDret1999}, instead, turns out to be more amenable for studying the numerical approximation of the solution of Problem~\ref{problemK}, as the $H^2$-regularity for the transverse component is replaced by a series of equivalent conditions that will put us in the position to consider the mixed formulation associated with Problem~\ref{problemK}, mixed formulation that is solely defined over the space $H^1(\omega)$. The bridging between the classical formulation of Koiter's model~\cite{Koiter1959,Koiter1966,Koiter1970} and the formulation due to Blouza \& Le Dret~\cite{BlouzaLeDret1999} is even \emph{stronger} when a boundary condition of place is imposed over the entire boundary $\gamma$, which is the case we are going to discuss next.

\section{Koiter's model for linearly elastic elliptic membrane shells: Classical formulation} \label{sec3}

In section~\ref{sec2}, we considered an obstacle problem for ``general'' linearly elastic shells. In all what follows, we will restrict ourselves  to considering a special class of shells, according to the following definition, that was originally proposed in~\cite{Ciarlet1996} (see also \cite{Ciarlet2000}). A recent application of Koiter's model was proposed in~\cite{BCGMQ14,MC15}. In this direction, we also cite the recent result~\cite{ALP23}.

Consider a linearly elastic shell subjected to the assumptions made in section~\ref{sec2}. Such a shell is said to be a \emph{linearly elastic elliptic membrane shell} (from now on simply \emph{membrane shell}) if the following two additional assumptions are satisfied: \emph{first}, $\gamma_0 = \gamma$, i.e., the homogeneous boundary condition of place is imposed over the \emph{entire lateral face} $\gamma \times \left] - \varepsilon , \varepsilon \right[$ of the shell, and \emph{second}, the shell middle surface $\bm{\theta}(\overline{\omega})$ is \emph{elliptic}, according to the definition given in section \ref{sec1}. 

Define the mapping $\pi^\varepsilon:\overline{\omega} \times [-1,1] \to \overline{\Omega^\varepsilon}$ by:
$$
\pi^\varepsilon(x):=(x_1,x_2, \varepsilon x_3), \quad \textup{ for all } x=(x_i) \in \overline{\omega} \times [-1,1].
$$

In what follows we assume that the applied body forces $\bm{f}^\varepsilon=(f^{i,\varepsilon})$ satisfy the following scaling with respect to the thickness parameter $\varepsilon$
$$
f^{i,\varepsilon}(x^\varepsilon)=f^i(x),\quad\textup{ for a.a. } x^\varepsilon=\pi^\varepsilon(x) \in \Omega^\varepsilon,
$$
for some $\bm{f}=(f^i) \in \bm{L}^2(\omega\times (-1,1))$.

It turns out that, when an \emph{elliptic surface} is subjected to a displacement field $\tilde{\bm{\eta}}=\eta_i \bm{a}^i$ whose \emph{tangential covariant components $\eta_\alpha$ vanish on the entire boundary of the domain $\omega$}, the following inequality of Koirn's type holds. Note that the components of the displacement fields and linearised change of metric tensors appearing in the next theorem no longer need to be continuously differentiable functions, but they are understood to belong to \emph{ad hoc} Lebesgue or Sobolev spaces.

\begin{theorem} 
	\label{korn}
	Let $\omega$ be a domain in $\mathbb{R}^2$ and let an immersion $\bm{\theta} \in \mathcal{C}^3 (\overline{\omega}; \mathbb{E}^3)$ be given such that the surface $\bm{\theta}(\overline{\omega})$ is elliptic. Define the space
	$$
	\bm{V}_M (\omega) := H^1_0 (\omega) \times H^1_0 (\omega) \times L^2(\omega).
	$$
	
	Then, there exists a constant $c_0=c_0(\omega,\bm{\theta})>0$ such that
	$$
	\left\{\sum_\alpha\left\|\eta_\alpha \right\|^2_{H^1(\omega)} + \left\|\eta_3 \right\|^2_{L^2(\omega)}\right\}^{1/2} \le c_0 \left\{\sum_{\alpha,\beta} \left\| \gamma_{\alpha\beta}(\bm{\eta}) \right\|_{L^2(\omega)}^2\right\}^{1/2}
	$$
	for all $\bm{\eta}= (\eta_i) \in \bm{V}_M (\omega)$.
	\qed
\end{theorem}

The previous inequality, that was originally established by Ciarlet, Lods and Sanchez-Palencia in the papers~\cite{CiaLods1996a} and \cite{CiaSanPan1996} (see also Theorem~2.7-3 of~\cite{Ciarlet2000}), is an example of a Korn's inequality on a surface. This estimate asserts that the three components of an admissible displacement vector field can be bounded above by a suitable rescaling of a norm associated with a \emph{measure of strain}.

In the case where the middle surface is elliptic, the space $\bm{V}_K(\omega)$ reduces to
$$
\bm{V}_K(\omega)=H^1_0(\omega) \times H^1_0(\omega) \times H^2_0(\omega).
$$

It was shown in Theorem~3.2 of~\cite{CiaPie2018b} that the solution of Problem~\ref{problemK} asymptotically behaves as the solution of the following variational problem, which was also recently studied in the papers~\cite{Pie-2022-interior}, as the thickness parameter $\varepsilon$ approaches zero.

\begin{customprob}{$\mathcal{P}_M(\omega)$}
	\label{problem1}
	Find $\bm{\zeta}=(\zeta_i) \in \bm{U}_M(\omega)$ satisfying the following variational inequalities:
	\begin{equation*}
	\int_\omega a^{\alpha\beta\sigma\tau} \gamma_{\sigma\tau}(\bm{\zeta}) \gamma_{\alpha\beta} (\bm{\eta} - \bm{\zeta}) \sqrt{a} \dd y \ge \int_\omega p^i (\eta_i - \zeta_i) \sqrt{a} \dd y,
	\end{equation*}
	for all $\bm{\eta} = (\eta_i) \in \bm{U}_M(\omega)$, where $p^i:=\int_{-1}^{1} f^i \dd x_3$.
	\bqed
\end{customprob}

We observe that, as a result of the limit process as $\varepsilon \to 0$ that makes the solution of Problem~\ref{problemK} converge to the solution of Problem~\ref{problem1}, the stretching mode associated with the linearised change of metric tensor prevails over the bending mode associated with the linearised change of curvature tensor when the shell under consideration is a membrane shell.

It is worth recalling that, by virtue of Korn's inequality (Theorem~\ref{korn}), it results that Problem~\ref{problem1} admits a unique solution. Solving Problem~\ref{problem1} amounts to minimizing the energy functional $J: H^1(\omega) \times H^1(\omega) \times L^2(\omega) \to \mathbb{R}$, which is defined by
\begin{equation*}
\label{Jeps}
J(\bm{\eta}):=\dfrac{1}{2} \int_{\omega} a^{\alpha\beta\sigma\tau} \gamma_{\sigma\tau}(\bm{\eta}) \gamma_{\alpha\beta}(\bm{\eta})\sqrt{a} \dd y-\int_{\omega} p^i \eta_i \sqrt{a} \dd y,
\end{equation*}
along all the test functions $\bm{\eta}=(\eta_i) \in \bm{U}_M(\omega)$.

Critical to establishing the asymptotic behaviour of the solution of Problem~\ref{problemK} as the thickness parameter $\varepsilon$ approaches zero, is the following ``density property'' (viz. \cite{CiaMarPie2018b,CiaMarPie2018}).

\begin{theorem}
	\label{density}
	Let $\bm{\theta} \in \mathcal{C}^2(\overline{\omega}; \mathbb{E}^3)$ be an immersion with the following property: There exists a non-zero vector $\bm{q} \in \mathbb{E}^3$ such that
	\begin{equation*}
	\min_{y \in \overline{\omega}} (\bm{\theta}(y) \cdot \bm{q}) > 0
	\textup{ and }
	\min_{y \in \overline{\omega}} (\bm{a}_3(y) \cdot \bm{q}) > 0.
	\end{equation*}
	
	Define the sets
	\begin{align*}
		\bm{U}_M(\omega) := \{\bm{\eta} = &(\eta_i) \in H^1_0(\omega) \times H^1_0(\omega) \times L^2(\omega); \big(\bm{\theta}(y) + \eta_i(y) \bm{a}^i(y)\big) \cdot \bm{q} \ge 0 \textup{ for a.a. } y \in \omega\}, \\
		\bm{U}_M(\omega) \cap \boldsymbol{\mathcal{D}} (\omega) := \{ \bm{\eta} =& (\eta_i ) \in \mathcal{D}(\omega) \times \mathcal{D} (\omega) \times \mathcal{D} (\omega); \big(\bm{\theta}(y) + \eta_i(y) \bm{a}^i(y)\big) \cdot \bm{q} \ge 0   \textup{ for a.a. } y \in \omega\}.
	\end{align*}
	Then the set $\bm{U}_M(\omega) \cap \boldsymbol{\mathcal{D}}(\omega)$ is dense in the set $\bm{U}_M(\omega)$ with respect to the norm $\left\|\cdot\right\|_{H^1(\omega) \times H^1(\omega) \times L^2(\omega)}$.
	\qed
\end{theorem}

Examples of membrane shells satisfying the ``density property'' thus include those whose middle surface is a portion of an ellipsoid that is strictly contained in one of the open half-spaces that contain two of its main axes, the boundary of the half-space coinciding with the obstacle in this case.

One such ``density property'' turned out to be the keystone for establishing the higher interior regularity of the solution of Problem~\ref{problem1} (viz. \cite{Pie-2022-interior}).

In what follows, we will restrict ourselves to considering elliptic middle surface satisfying the sufficient conditions ensuring the ``density property'' that were laid out in Theorem~\ref{density}.

\section{Approximation of the solution of Problem~$\mathcal{P}_K^\varepsilon(\omega)$ by penalization}
\label{sec:penalty}

Following~\cite{Scholz1984,Scholz1987}, we approximate the solution of Problem~\ref{problemK} by the penalty method. The geometrical constraint characterising the set $\bm{U}_K(\omega)$ is made appear in the energy functional of the model under consideration, and takes the form of a monotone term. The variational formulation corresponding to the \emph{penalised energy} is no longer tested over a non-empty, closed and convex subset of $\bm{V}_K(\omega)$, but over the \emph{whole} space $\bm{V}_K(\omega)$. Moreover, the variational inequalities are replaced by a set of nonlinear equations, where the nonlinearity is a monotone operator.

More precisely, define the operator $\bm{\beta}:\bm{L^2}(\omega) \to\bm{L}^2(\omega)$ in the following fashion
\begin{equation*}
\bm{\beta}(\bm{\xi}):=\left(-\{(\bm{\theta}+\xi_j \bm{a}^j)\cdot\bm{q}\}^{-}\left(\dfrac{\bm{a}^i\cdot\bm{q}}{\sqrt{\sum_{\ell=1}^{3}|\bm{a}^\ell \cdot\bm{q}|^2}}\right)\right)_{i=1}^3,\quad\textup{ for all }\bm{\xi}=(\xi_i) \in \bm{L}^2(\omega),
\end{equation*}
and we notice that this operator is associated with a penalisation proportional to the extent the constraint is broken. Note that the denominator \emph{never} vanishes, and that this fact is \emph{independent} of the assumption $\min_{y \in \overline{\omega}}(\bm{a}^3\cdot\bm{q})>0$ (viz. Theorem~\ref{density}). Besides, since the vectors $\{\bm{a}^i\}_{i=1}^3$ of the contravariant basis are linearly independent at all the points $y \in \overline{\omega}$, we can assume without loss of generality that $|\bm{a}^i(y)|=1$, for all $y\in\overline{\omega}$, for all $1\le i \le 3$.

It can be shown, in the same fashion as~\cite{Pie2023,PT2023}, that the operator $\bm{\beta}$ is monotone, bounded and non-expansive.

\begin{lemma}
\label{lem:beta}
Let $\bm{q} \in \mathbb{E}^3$ be a given unit-norm vector. Assume that $\min_{y \in \overline{\omega}}(\bm{a}^3(y)\cdot\bm{q})>0$.
Assume that the vectors $\{\bm{a}^i\}_{i=1}^3$ of the contravariant basis are such that $|\bm{a}^i(y)|=1$ and for all $y\in\overline{\omega}$, for all $1\le i \le 3$.

Then, the operator $\bm{\beta}:\bm{L^2}(\omega) \to\bm{L}^2(\omega)$ defined by 
\begin{equation*}
\bm{\beta}(\bm{\xi}):=\left(-\{(\bm{\theta}+\xi_j \bm{a}^j)\cdot\bm{q}\}^{-}\left(\dfrac{\bm{a}^i\cdot\bm{q}}{\sqrt{\sum_{\ell=1}^{3}|\bm{a}^\ell \cdot\bm{q}|^2}}\right)\right)_{i=1}^3,\quad\textup{ for all }\bm{\xi}=(\xi_i) \in \bm{L}^2(\omega),
\end{equation*}
is bounded, monotone and Lipschitz continuous with Lipschitz constant $L=1$.
\qed
\end{lemma}
\begin{proof}
	Let $\bm{\xi}$ and $\bm{\eta}$ be arbitrarily given in $\bm{L}^2(\omega)$. Evaluating
	\begin{align*}
		&\int_{\omega} (\bm{\beta}(\bm{\xi})-\bm{\beta}(\bm{\eta}))\cdot(\bm{\xi}-\bm{\eta})\dd y
		=\int_{\omega} \left(\left[-\{(\bm{\theta}+\xi_j\bm{a}^j)\cdot\bm{q}\}^{-}\right] - \left[-\{(\bm{\theta}+\eta_j\bm{a}^j)\cdot\bm{q}\}^{-}\right]\right) \left(\dfrac{(\xi_i-\eta_i)\bm{a}^i\cdot\bm{q}}{\sqrt{\sum_{\ell=1}^{3}|\bm{a}^\ell \cdot\bm{q}|^2}}\right) \dd y\\
		&=\int_{\omega}\dfrac{\left|-\{(\bm{\theta}+\xi_j\bm{a}^j)\cdot\bm{q}\}^{-}\right|^2}{\sqrt{\sum_{\ell=1}^{3}|\bm{a}^\ell \cdot\bm{q}|^2}} \dd y +\int_{\omega}\dfrac{\left|-\{(\bm{\theta}+\eta_j\bm{a}^j)\cdot\bm{q}\}^{-}\right|^2}{\sqrt{\sum_{\ell=1}^{3}|\bm{a}^\ell \cdot\bm{q}|^2}} \dd y\\
		&\quad+\int_{\omega}\dfrac{\left(-\{(\bm{\theta}+\xi_j\bm{a}^j)\cdot\bm{q}\}^{-}\right)}{\sqrt{\sum_{\ell=1}^{3}|\bm{a}^\ell \cdot\bm{q}|^2}}  \left(-\{(\bm{\theta}+\eta_i\bm{a}^i)\cdot\bm{q}\}^{+}+\{(\bm{\theta}+\eta_i\bm{a}^i)\cdot\bm{q}\}^{-}\right)\dd y\\
		&\quad+\int_{\omega}\dfrac{\left(-\{(\bm{\theta}+\eta_j\bm{a}^j)\cdot\bm{q}\}^{-}\right)}{\sqrt{\sum_{\ell=1}^{3}|\bm{a}^\ell \cdot\bm{q}|^2}}  \left(-\{(\bm{\theta}+\xi_i\bm{a}^i)\cdot\bm{q}\}^{+}+\{(\bm{\theta}+\xi_i\bm{a}^i)\cdot\bm{q}\}^{-}\right)\dd y\\
		&\ge \int_{\omega}\dfrac{\left|-\{(\bm{\theta}+\xi_j\bm{a}^j)\cdot\bm{q}\}^{-}\right|^2}{\sqrt{\sum_{\ell=1}^{3}|\bm{a}^\ell \cdot\bm{q}|^2}} \dd y 
		+\int_{\omega}\dfrac{\left|-\{(\bm{\theta}+\eta_j\bm{a}^j)\cdot\bm{q}\}^{-}\right|^2}{\sqrt{\sum_{\ell=1}^{3}|\bm{a}^\ell \cdot\bm{q}|^2}} \dd y
		+\int_{\omega}\dfrac{\left(-\{(\bm{\theta}+\xi_j\bm{a}^j)\cdot\bm{q}\}^{-}\right)}{\sqrt{\sum_{\ell=1}^{3}|\bm{a}^\ell \cdot\bm{q}|^2}}  \left(\{(\bm{\theta}+\eta_i\bm{a}^i)\cdot\bm{q}\}^{-}\right)\dd y\\
		&\quad+\int_{\omega}\dfrac{\left(-\{(\bm{\theta}+\eta_j\bm{a}^j)\cdot\bm{q}\}^{-}\right)}{\sqrt{\sum_{\ell=1}^{3}|\bm{a}^\ell \cdot\bm{q}|^2}}  \left(\{(\bm{\theta}+\xi_i\bm{a}^i)\cdot\bm{q}\}^{-}\right)\dd y\\
		&=\int_{\omega}\dfrac{\left|\left(-\{(\bm{\theta}+\eta_j\bm{a}^j)\cdot\bm{q}\}^{-}\right) - \left(-\{(\bm{\theta}+\xi_j\bm{a}^j)\cdot\bm{q}\}^{-}\right)\right|^2}{\sqrt{\sum_{\ell=1}^{3}|\bm{a}^\ell \cdot\bm{q}|^2}}\dd y\ge 0,
	\end{align*}
	proves the monotonicity of the operator $\bm{\beta}$.
	
	For showing the boundedness of the operator $\bm{\beta}$, we show that it maps bounded sets of $\bm{L}^2(\omega)$ into bounded sets of $\bm{L}^2(\omega)$.
	Let the set $\mathscr{F} \subset \bm{L}^2(\omega)$ be bounded. For each $\bm{\xi} \in \mathscr{F}$, we have that
	\begin{align*}
		&\|\bm{\beta}(\bm{\xi})\|_{\bm{L}^2(\omega)}=\left(\int_{\omega}\dfrac{|-\{(\bm{\theta}+\xi_j\bm{a}^j)\cdot\bm{q}\}^{-}|^2}{\sum_{\ell=1}^3|\bm{a}^\ell \cdot\bm{q}|^2}\sum_{i=1}^3|\bm{a}^i\cdot\bm{q}|^2 \dd y\right)^{1/2}\\
		&= \|-\{(\bm{\theta}+\xi_j\bm{a}^j)\cdot\bm{q}\}^{-}\|_{L^2(\omega)} \le \|\bm{\theta}\cdot\bm{q}\|_{L^2(\omega)}+\|\bm{\xi}\|_{\bm{L}^2(\omega)},
	\end{align*}
	and the sought boundedness is thus asserted, being $\bm{\theta} \in \mathcal{C}^3(\overline{\omega};\mathbb{E}^3)$ and $\mathscr{F}$ bounded in $\bm{L}^2(\omega)$.
	
	Finally, to establish the Lipschitz continuity, for all $\bm{\xi}$ and $\bm{\eta} \in \bm{L}^2(\omega)$, we evaluate $\|\bm{\beta}(\bm{\xi})-\bm{\beta}(\bm{\eta})\|_{\bm{L}^2(\omega)}$. We have that
	\begin{equation*}
		\begin{aligned}
			&\|\bm{\beta}(\bm{\xi})-\bm{\beta}(\bm{\eta})\|_{\bm{L}^2(\omega)}=\left(\int_{\omega}\dfrac{1}{\sum_{\ell=1}^{3}|\bm{a}^\ell \cdot\bm{q}|^2}\left\{\left|[-\{(\bm{\theta}+\xi_i\bm{a}^i)\cdot\bm{q}\}^{-}] - [-\{(\bm{\theta}+\eta_j\bm{a}^j)\cdot\bm{q}\}^{-}] \right|^2 \left(\sum_{\ell=1}^{3}|\bm{a}^\ell \cdot\bm{q}|^2\right) \right\}\dd y\right)^{1/2}\\
			&=\left(\int_{\omega}\left|\left[-\{(\bm{\theta}+\xi_j\bm{a}^j)\cdot\bm{q}\}^{-}\right] - \left[-\{(\bm{\theta}+\eta_j\bm{a}^j)\cdot\bm{q}\}^{-}\right]\right|^2 \dd y\right)^{1/2}\\
			&=\left(\int_{\omega}\left|\dfrac{|(\bm{\theta}+\xi_j\bm{a}^j)\cdot\bm{q}|-(\bm{\theta}+\xi_j\bm{a}^j)\cdot\bm{q}}{2} - \dfrac{|(\bm{\theta}+\eta_j\bm{a}^j)\cdot\bm{q}|-(\bm{\theta}+\eta_j\bm{a}^j)\cdot\bm{q}}{2}\right|^2 \dd y\right)^{1/2}\\
			&\le\left(\int_{\omega}\left(\left|\dfrac{(\bm{\theta}+\xi_j\bm{a}^j)\cdot\bm{q}-(\bm{\theta}+\eta_j\bm{a}^j)\cdot\bm{q}}{2}\right| + \left|\dfrac{(\bm{\theta}+\xi_j\bm{a}^j)\cdot\bm{q}-(\bm{\theta}+\eta_j\bm{a}^j)\cdot\bm{q}}{2}\right|\right)^2 \dd y\right)^{1/2}\\
			&= \left\|(\bm{\theta}+\xi_j\bm{a}^j)\cdot\bm{q}-(\bm{\theta}+\eta_j\bm{a}^j)\cdot\bm{q}\right\|_{L^2(\omega)}\le \|\bm{\xi}-\bm{\eta}\|_{\bm{L}^2(\omega)},
		\end{aligned}
	\end{equation*}
	and the sought Lipschitz continuity is thus established. Note in passing that the Lipschitz constant is $L=1$. This completes the proof.
\end{proof}

Let $\kappa>0$ denote a penalty parameter which is meant to approach zero. The penalised version of Problem~\ref{problemK} is formulated as follows.

\begin{customprob}{$\mathcal{P}_{K,\kappa}^\varepsilon(\omega)$}
	\label{problem2}
	Find $\bm{\zeta}^\varepsilon_\kappa=(\zeta^\varepsilon_{\kappa,i}) \in \bm{V}_K(\omega)$ satisfying the following variational equations:
	\begin{equation*}
	\varepsilon \int_\omega a^{\alpha\beta\sigma\tau} \gamma_{\sigma\tau}(\bm{\zeta}^\varepsilon_\kappa) \gamma_{\alpha\beta} (\bm{\eta}) \sqrt{a} \dd y 
	+\dfrac{\varepsilon^3}{3}\int_{\omega} a^{\alpha\beta\sigma\tau} \rho_{\sigma\tau}(\bm{\zeta}^\varepsilon_\kappa) \rho_{\alpha\beta}(\bm{\eta}) \sqrt{a} \dd y
	+\dfrac{\varepsilon}{\kappa}\int_{\omega} \bm{\beta}(\bm{\zeta}^\varepsilon_\kappa) \cdot \bm{\eta} \dd y
	= \int_\omega p^{i,\varepsilon} \eta_i \sqrt{a} \dd y,
	\end{equation*}
	for all $\bm{\eta} = (\eta_i) \in \bm{V}_K(\omega)$.
	\bqed
\end{customprob}

The existence and uniqueness of solutions of Problem~\ref{problem2} can be established by resorting to the Minty-Browder theorem (cf., e.g., Theorem~9.14-1 of~\cite{PGCLNFAA}). For the sake of completeness, we present the proof of this existence and uniqueness result.

\begin{theorem}
\label{ex-un-kappa}
Let $\bm{q} \in\mathbb{E}^3$ be a given unit-norm vector. Assume that $\bm{\theta} \in \mathcal{C}^3(\overline{\omega};\mathbb{E}^3)$ is such that $\min_{y \in \overline{\omega}}(\bm{\theta}(y)\cdot\bm{q})>~0$.

Then, for each $\kappa>0$ and $\varepsilon>0$, Problem~\ref{problem2} admits a unique solution. Moreover, the family of solutions $\{\bm{\zeta}^\varepsilon_\kappa\}_{\kappa>0}$ is bounded in $\bm{V}_K(\omega)$ independently of $\kappa$, and 
$$
\bm{\zeta}^\varepsilon_\kappa \to \bm{\zeta}^\varepsilon,\quad\textup{ in }\bm{V}_K(\omega) \textup{ as }\kappa \to 0^+,
$$
where $\bm{\zeta}^\varepsilon$ is the solution of Problem~\ref{problemK}.
\end{theorem}
\begin{proof}
Let us define the operator $\bm{A}^\varepsilon:\bm{V}_K(\omega) \to \bm{V}'_K(\omega)$ by
\begin{equation*}
\langle \bm{A}^\varepsilon \bm{\xi},\bm{\eta}\rangle_{\bm{V}'_M(\omega), \bm{V}_M(\omega)}:=\varepsilon\int_{\omega} a^{\alpha\beta\sigma\tau} \gamma_{\sigma\tau}(\bm{\xi}) \gamma_{\alpha\beta}(\bm{\eta}) \sqrt{a} \dd y+\dfrac{\varepsilon^3}{3}\int_{\omega} a^{\alpha\beta\sigma\tau} \rho_{\sigma\tau}(\bm{\xi}) \rho_{\alpha\beta}(\bm{\eta}) \sqrt{a} \dd y.
\end{equation*}

We observe that the operator $\bm{A}^\varepsilon$ is linear, continuous and, thanks to the uniform positive-definiteness of the fourth order two-dimensional tensor $(a^{\alpha\beta\sigma\tau})$ (viz., e.g., Theorem~3.3-2 of~\cite{Ciarlet2000}) and Korn's inequality (Theorem~\ref{kornsurface}). We have that:
\begin{equation}
\label{Aeps}
\langle \bm{A}^\varepsilon \bm{\xi} -\bm{A}^\varepsilon\bm{\eta},\bm{\xi}-\bm{\eta}\rangle_{\bm{V}'_K(\omega), \bm{V}_K(\omega)} \ge \dfrac{\varepsilon^3}{3c_e c_0^2} \|\bm{\xi}-\bm{\eta}\|_{\bm{V}_K(\omega)}^2, \quad\textup{ for all }\bm{\xi}, \bm{\eta} \in \bm{V}_K(\omega).
\end{equation}

Define the operator $\hat{\bm{\beta}}:\bm{V}_K(\omega) \to \bm{V}'_K(\omega)$ as the following composition
\begin{equation*}
\bm{V}_K(\omega) \hookrightarrow \bm{L}^2(\omega) \xrightarrow{\bm{\beta}} \bm{L}^2(\omega) \hookrightarrow \bm{V}'_K(\omega).
\end{equation*}

Thanks to the monotonicity of $\bm{\beta}$ established in Lemma~\ref{lem:beta}, we easily infer that $\hat{\bm{\beta}}$ is monotone.
Therefore, as a direct consequence of~\eqref{Aeps} and Lemma~\ref{lem:beta}, we can infer that the operator $(\bm{A}^\varepsilon+\hat{\bm{\beta}}):\bm{V}_K(\omega) \to \bm{V}'_K(\omega)$ is strongly monotone. To see this, observe that for all $\bm{\eta}$, $\bm{\xi} \in \bm{V}_K(\omega)$ with $\bm{\xi}\neq\bm{\eta}$, we have that
\begin{equation*}
\begin{aligned}
&\langle (\bm{A}^\varepsilon+\hat{\bm{\beta}}) \bm{\xi} -(\bm{A}^\varepsilon+\hat{\bm{\beta}})\bm{\eta},\bm{\xi}-\bm{\eta}\rangle_{\bm{V}'_K(\omega), \bm{V}_K(\omega)}\\
&= \langle \bm{A}^\varepsilon \bm{\xi} -\bm{A}^\varepsilon\bm{\eta},\bm{\xi}-\bm{\eta}\rangle_{\bm{V}'_K(\omega), \bm{V}_K(\omega)}\\
&\quad+\langle\hat{\bm{\beta}}(\bm{\xi})-\hat{\bm{\beta}}(\bm{\eta}),\bm{\xi}-\bm{\eta}\rangle_{\bm{V}'_K(\omega), \bm{V}_K(\omega)}\ge \dfrac{\varepsilon^3}{3c_e c_0^2} \|\bm{\xi}-\bm{\eta}\|_{\bm{V}_K(\omega)}^2>0.
\end{aligned}
\end{equation*}

Similarly, we can establish the coerciveness of the operator $(\bm{A}^\varepsilon+\hat{\bm{\beta}})$. Indeed, we have that
\begin{equation*}
\dfrac{\langle (\bm{A}^\varepsilon+\hat{\bm{\beta}}) \bm{\eta},\bm{\eta}\rangle_{\bm{V}'_K(\omega), \bm{V}_K(\omega)}}{\|\bm{\eta}\|_{\bm{V}_K(\omega)}} =\dfrac{\langle\bm{A}^\varepsilon\bm{\eta},\bm{\eta}\rangle_{\bm{V}'_K(\omega), \bm{V}_K(\omega)}}{\|\bm{\eta}\|_{\bm{V}_K(\omega)}} +\dfrac{\langle \hat{\bm{\beta}}(\bm{\eta}),\bm{\eta}\rangle_{\bm{V}'_K(\omega), \bm{V}_K(\omega)}}{\|\bm{\eta}\|_{\bm{V}_K(\omega)}} \ge \dfrac{\varepsilon^3}{3c_e c_0^2} \|\bm{\eta}\|_{\bm{V}_K(\omega)},
\end{equation*}
where the last inequality is obtained by combining~\eqref{Aeps}, Lemma~\ref{lem:beta} with the fact that $\bm{0} \in \bm{U}_K(\omega)$ or, equivalently, that $\bm{\beta}(\bm{0})=\bm{0}$ in $\bm{L}^2(\omega)$.

The continuity of the operator $\bm{A}^\varepsilon$ and the Lipschitz continuity of the operator $\bm{\beta}$ established in Lemma~\ref{lem:beta} in turn give that the operator $(\bm{A}^\varepsilon+\hat{\bm{\beta}})$ is hemicontinuous, and we are in position to apply the Minty-Browder theorem (cf., e.g., Theorem~9.14-1 of~\cite{PGCLNFAA}) to establish that there exists a unique solution $\bm{\zeta}^\varepsilon_\kappa \in \bm{V}_K(\omega)$ for Problem~\ref{problem2}.

Observe that the fact that $\min_{y\in\overline{\omega}} (\bm{\theta}(y)\cdot\bm{q})>0$ implies:
\begin{equation}
	\label{beta-2}
	\begin{aligned}
		&\int_{\omega} \bm{\beta}(\bm{\zeta}^\varepsilon_\kappa) \cdot\bm{\zeta}^\varepsilon_\kappa \dd y
		=\int_{\omega}\dfrac{1}{\sqrt{\sum_{\ell=1}^{3}|\bm{a}^\ell \cdot\bm{q}|^2}} \left(-\{(\bm{\theta}+\zeta^\varepsilon_{\kappa,j}\bm{a}^j)\cdot\bm{q}\}^{-}\right) (\zeta^\varepsilon_{\kappa,i}\bm{a}^i\cdot\bm{q}) \dd y\\
		&=\int_{\omega}\dfrac{1}{\sqrt{\sum_{\ell=1}^{3}|\bm{a}^\ell \cdot\bm{q}|^2}} \left(-\{(\bm{\theta}+\zeta^\varepsilon_{\kappa,j}\bm{a}^j)\cdot\bm{q}\}^{-}\right) ((\bm{\theta}+\zeta^\varepsilon_{\kappa,i}\bm{a}^i)\cdot\bm{q}) \dd y\\
		&\quad-\int_{\omega}\dfrac{1}{\sqrt{\sum_{\ell=1}^{3}|\bm{a}^\ell \cdot\bm{q}|^2}} \left(-\{(\bm{\theta}+\zeta^\varepsilon_{\kappa,j}\bm{a}^j)\cdot\bm{q}\}^{-}\right) (\bm{\theta}\cdot\bm{q}) \dd y
		\ge \int_{\omega}\dfrac{|-\{(\bm{\theta}+\zeta^\varepsilon_{\kappa,i}\bm{a}^i)\cdot \bm{q}\}^{-}|^2}{\sqrt{\sum_{\ell=1}^{3}|\bm{a}^\ell \cdot\bm{q}|^2}}  \dd y.
	\end{aligned}
\end{equation}

Furthermore, if we specialize $\bm{\eta}=\bm{\zeta}^\varepsilon_\kappa$ in the variational equations of Problem~\ref{problem2}, we have that an application of Korn's inequality (Theorem~\ref{korn}), the monotonicity of $\bm{\beta}$ (Lemma~\ref{lem:beta}), the strict positiveness and boundedness of $a$ (Theorems~ 3.1-1 of~\cite{Ciarlet2000}), the uniform positive definiteness of the fourth order two-dimensional elasticity tensor $(a^{\alpha\beta\sigma\tau})$ (Theorem~3.3-2 of~\cite{Ciarlet2000}), and the fact that $\bm{0} \in \bm{U}_M(\omega)$ or, equivalently, that $\bm{\beta}(\bm{0})=\bm{0}$ in $\bm{L}^2(\omega)$ give:
\begin{equation*}
\begin{aligned}
\dfrac{\varepsilon^3\sqrt{a_0}}{3c_0^2 c_e}\|\bm{\zeta}^\varepsilon_\kappa\|_{\bm{V}_K(\omega)}^2 
&\le \varepsilon\int_{\omega} a^{\alpha\beta\sigma\tau} \gamma_{\sigma\tau}(\bm{\zeta}^\varepsilon_\kappa) \gamma_{\alpha\beta}(\bm{\zeta}^\varepsilon_\kappa) \sqrt{a} \dd y 
+\dfrac{\varepsilon^3}{3}\int_{\omega} a^{\alpha\beta\sigma\tau} \rho_{\sigma\tau}(\bm{\zeta}^\varepsilon_\kappa) \rho_{\alpha\beta}(\bm{\zeta}^\varepsilon_\kappa) \sqrt{a} \dd y 
+\dfrac{\varepsilon}{\kappa} \int_{\omega} \bm{\beta}(\bm{\zeta}^\varepsilon_\kappa) \cdot\bm{\zeta}^\varepsilon_\kappa \dd y\\
&\le \|\bm{p}^\varepsilon\|_{\bm{L}^2(\omega)} \|\bm{\zeta}^\varepsilon_\kappa\|_{\bm{V}_K(\omega)} \sqrt{a_1}
=\varepsilon \sqrt{a_1}\|\bm{p}\|_{\bm{L}^2(\omega)} \|\bm{\zeta}^\varepsilon_\kappa\|_{\bm{V}_K(\omega)}.
\end{aligned}
\end{equation*}
Note that the last equality holds thanks to the definition of $\bm{p}=(p^i)$ and $\bm{p}^\varepsilon=(p^{i,\varepsilon})$ introduced, respectively, in Problem~\ref{problemK} and Problem~\ref{problem1}. In conclusion, we have that:
\begin{equation}
\label{upper-bound-solution}
\|\bm{\zeta}^\varepsilon_\kappa\|_{\bm{V}_K(\omega)}
\le \dfrac{3 c_0^2 c_e \sqrt{a_1}}{\varepsilon^2 \sqrt{a_0}}\|\bm{p}\|_{\bm{L}^2(\omega)},\quad\textup{ for all }\kappa>0.
\end{equation}

By virtue of the definition of $p^{i,\varepsilon}$ and the assumptions on the data stated at the beginning of section~\ref{sec3}, we get that $\|\bm{\zeta}^\varepsilon_\kappa\|_{\bm{V}_K(\omega)}$ is bounded independently of $\kappa$. Therefore, by the Banach-Eberlein-Smulian theorem (cf., e.g., Theorem~5.14-4 of~\cite{PGCLNFAA}), we can extract a subsequence, still denoted $\{\bm{\zeta}^\varepsilon_\kappa\}_{\kappa>0}$ such that:
\begin{equation}
\label{beta-1}
\bm{\zeta}^\varepsilon_\kappa \rightharpoonup \bm{\zeta}^\varepsilon, \quad\textup{ in } \bm{V}_K(\omega) \textup{ as } \kappa\to0^+. 
\end{equation}

Specializing $\bm{\eta}=\bm{\zeta}^\varepsilon_\kappa$ in the variational equations of Problem~\ref{problem2} and applying~\eqref{beta-2} and~\eqref{beta-1} give that:
\begin{equation}
\label{beta-2.5}
\begin{aligned}
&\dfrac{\left(3\max\{\|\bm{a}^j \cdot \bm{q}\|_{\mathcal{C}^0(\overline{\omega})}^2;1\le j\le 3\}\right)^{-1/2}}{\kappa}\|-\{(\bm{\theta}+\zeta^\varepsilon_{\kappa,j}\bm{a}^j)\bm
q\}^{-}\|_{\bm{L}^2(\omega)}^2\\
&\le\dfrac{1}{\kappa}\int_{\omega} \bm{\beta}(\bm{\zeta}^\varepsilon_\kappa) \cdot\bm{\zeta}^\varepsilon_\kappa \dd y\le  \dfrac{3 c_0^2 c_e a_1}{\varepsilon^2 \sqrt{a_0}}\|\bm{p}\|_{\bm{L}^2(\omega)}^2.
\end{aligned}
\end{equation}

Therefore, we have that an application of~\eqref{beta-2.5} gives that
\begin{equation}
\label{beta-3}
\bm{\beta}(\bm{\zeta}^\varepsilon_\kappa) \to \bm{0},\quad\textup{ in }\bm{L}^2(\omega) \textup{ as }\kappa\to0^+,
\end{equation}
and that
\begin{equation}
\label{beta-4}
\int_{\omega} \bm{\beta}(\bm{\zeta}^\varepsilon_\kappa) \cdot \bm{\zeta}^\varepsilon_\kappa \dd y=\langle\hat{\bm{\beta}}(\bm{\zeta}^\varepsilon_\kappa),\bm{\zeta}^\varepsilon_\kappa\rangle_{\bm{V}'_K(\omega), \bm{V}_K(\omega)} \to 0,\quad\textup{ as }\kappa\to0^+.
\end{equation}

Therefore, the monotonicity of $\hat{\bm{\beta}}$ (which is a direct consequence of Lemma~\ref{lem:beta}), and the the properties established in~\eqref{beta-1}, \eqref{beta-3} and~\eqref{beta-4} put us in a position to apply Theorem~9.13-2 of~\cite{PGCLNFAA}. We obtain that $\hat{\bm{\beta}}(\bm{\zeta}^\varepsilon)=\bm{0}$, so that $\bm{\zeta}^\varepsilon \in \bm{U}_K(\omega)$.

Observe that the monotonicity of $\bm{\beta}$ (viz. Lemma~\ref{lem:beta}), the properties of $\bm{\zeta}^\varepsilon_\kappa$, the continuity of the components $\gamma_{\alpha\beta}$ of the linearised change of metric tensor, the definition of $\bm{p}^\varepsilon$ (Problem~\ref{problemK}), and the weak convergence~\eqref{beta-1} give
\begin{align*}
\|\bm{\zeta}^\varepsilon_\kappa - \bm{\zeta}^\varepsilon\|_{\bm{V}_K(\omega)}^2 &\le
\dfrac{3c_0^2 c_e}{\varepsilon^2\sqrt{a_0}}\int_{\omega} a^{\alpha\beta\sigma\tau}\gamma_{\sigma\tau}(\bm{\zeta}^\varepsilon_\kappa - \bm{\zeta}^\varepsilon)\gamma_{\alpha\beta}(\bm{\zeta}^\varepsilon_\kappa - \bm{\zeta}^\varepsilon)\sqrt{a} \dd y\\
&\quad+\dfrac{c_0^2 c_e}{\sqrt{a_0}}\int_{\omega} a^{\alpha\beta\sigma\tau}\rho_{\sigma\tau}(\bm{\zeta}^\varepsilon_\kappa - \bm{\zeta}^\varepsilon)\rho_{\alpha\beta}(\bm{\zeta}^\varepsilon_\kappa - \bm{\zeta}^\varepsilon)\sqrt{a} \dd y\\
&=-\dfrac{3c_0^2 c_e}{\kappa\varepsilon^2 \sqrt{a_0}}\int_{\omega} \bm{\beta}(\bm{\zeta}^\varepsilon_\kappa) \cdot (\bm{\zeta}^\varepsilon_\kappa - \bm{\zeta}^\varepsilon) \dd y\\
&\quad+\dfrac{3c_0^2 c_e}{\varepsilon^3\sqrt{a_0}} \int_{\omega} p^{i, \varepsilon} (\zeta^\varepsilon_{\kappa,i}-\zeta^\varepsilon_i) \sqrt{a} \dd y
-\dfrac{3c_0^2 c_e}{\varepsilon^2\sqrt{a_0}}\int_{\omega} a^{\alpha\beta\sigma\tau}\gamma_{\sigma\tau}(\bm{\zeta}^\varepsilon)\gamma_{\alpha\beta}(\bm{\zeta}^\varepsilon_\kappa - \bm{\zeta}^\varepsilon)\sqrt{a} \dd y\\
&\quad-\dfrac{c_0^2 c_e}{\sqrt{a_0}}\int_{\omega} a^{\alpha\beta\sigma\tau}\rho_{\sigma\tau}(\bm{\zeta}^\varepsilon)\rho_{\alpha\beta}(\bm{\zeta}^\varepsilon_\kappa - \bm{\zeta}^\varepsilon)\sqrt{a} \dd y\\
&\le \dfrac{3c_0^2 c_e}{\varepsilon^3\sqrt{a_0}} \int_{\omega} p^{i, \varepsilon} (\zeta^\varepsilon_{\kappa,i}-\zeta^\varepsilon_i) \sqrt{a} \dd y
-\dfrac{3c_0^2 c_e}{\varepsilon^2 \sqrt{a_0}}\int_{\omega} a^{\alpha\beta\sigma\tau}\gamma_{\sigma\tau}(\bm{\zeta}^\varepsilon)\gamma_{\alpha\beta}(\bm{\zeta}^\varepsilon_\kappa - \bm{\zeta}^\varepsilon)\sqrt{a} \dd y\\
&\quad-\dfrac{c_0^2 c_e}{\sqrt{a_0}}\int_{\omega} a^{\alpha\beta\sigma\tau}\rho_{\sigma\tau}(\bm{\zeta}^\varepsilon)\rho_{\alpha\beta}(\bm{\zeta}^\varepsilon_\kappa - \bm{\zeta}^\varepsilon)\sqrt{a} \dd y\\
&=\dfrac{3 c_0^2 c_e}{\varepsilon^3 \sqrt{a_0}} \int_{\omega} p^i (\zeta^\varepsilon_{\kappa,i}-\zeta^\varepsilon_i) \sqrt{a} \dd y
-\dfrac{3 c_0^2 c_e}{\varepsilon^2 \sqrt{a_0}}\int_{\omega} a^{\alpha\beta\sigma\tau}\gamma_{\sigma\tau}(\bm{\zeta}^\varepsilon)\gamma_{\alpha\beta}(\bm{\zeta}^\varepsilon_\kappa - \bm{\zeta}^\varepsilon)\sqrt{a} \dd y\\
&\quad-\dfrac{c_0^2 c_e}{\sqrt{a_0}}\int_{\omega} a^{\alpha\beta\sigma\tau}\rho_{\sigma\tau}(\bm{\zeta}^\varepsilon)\rho_{\alpha\beta}(\bm{\zeta}^\varepsilon_\kappa - \bm{\zeta}^\varepsilon)\sqrt{a} \dd y\to 0,
\end{align*}
as $\kappa \to 0^+$. Observe that the latter term is bounded independently of $\kappa$ although, in general, it is not bounded independently of $\varepsilon$. In conclusion, we have been able to establish the strong convergence:
\begin{equation}
\label{beta-5}
\bm{\zeta}^\varepsilon_\kappa \to \bm{\zeta}^\varepsilon,\quad\textup{ in } \bm{V}_K(\omega) \textup{ as } \kappa \to0^+.
\end{equation}

Specialising $(\bm{\eta}-\bm{\zeta}^\varepsilon_\kappa)\in\bm{V}_K(\omega)$ in the variational equations of Problem~\ref{problem2}, with $\bm{\eta}\in\bm{U}_K(\omega)$, the monotonicity of $\bm{\beta}$, the convergence~\eqref{beta-3} and the convergence~\eqref{beta-5} immediately give that the limit $\bm{\zeta}^\varepsilon\in\bm{U}_K(\omega)$ satisfies the variational inequalities in Problem~\ref{problemK}. This completes the proof.
\end{proof}

In order to conduct a sound numerical analysis of the obstacle problem under consideration, we need to prove a preparatory result concerning the augmentation of regularity of the solution of Problem~\ref{problem2} by resorting to the finite difference quotients approach originally proposed by Agmon, Douglis \& Nirenberg~\cite{AgmDouNir1959,AgmDouNir1964}, as well as the approach proposed by Frehse~\cite{Frehse1971} for variational inequalities, that was later on generalised in~\cite{Pie-2022-interior,Pie2020-1}.

Recalling that $\bm{\zeta}^\varepsilon_\kappa$ denotes the solution of Problem~\ref{problem2}, in the same spirit as Theorem~7.1-3(b) of~\cite{Ciarlet2000} we define
$$
n^{\alpha\beta,\varepsilon}_\kappa:=\varepsilon a^{\alpha\beta\sigma\tau}\gamma_{\sigma\tau}(\bm{\zeta}^\varepsilon_\kappa),
\qquad m^{\alpha\beta,\varepsilon}_\kappa:=\dfrac{\varepsilon^3}{3} a^{\alpha\beta\sigma\tau} \rho_{\sigma\tau}(\bm{\zeta}^\varepsilon_\kappa),
$$
and we also define
\begin{align*}
n^{\alpha\beta,\varepsilon}_\kappa|_{\beta}&:=\partial_\beta n^{\alpha\beta,\varepsilon}_\kappa+\Gamma^\alpha_{\beta\tau}n^{\beta\tau,\varepsilon}_\kappa+\Gamma^\beta_{\beta\tau}n^{\alpha\tau,\varepsilon}_\kappa,\\
m^{\alpha\beta,\varepsilon}_\kappa|_{\alpha\beta}&:=\partial_\alpha(m^{\alpha\beta,\varepsilon}_\kappa|_\beta)+\Gamma^\tau_{\alpha\tau}(m^{\alpha\beta,\varepsilon}_\kappa|_\beta).
\end{align*}

If the solution $\bm{\zeta}^\varepsilon_\kappa$ of Problem~\ref{problem2} is \emph{smooth enough}, then it is immediate to see that, in the same spirit of Theorem~7.1-3 of~\cite{Ciarlet2000}, it satisfies the following boundary value problem:
\begin{equation}
\label{BVP}
\begin{cases}
	-(n^{\alpha\beta,\varepsilon}_\kappa+b_\sigma^\alpha m^{\sigma\beta,\varepsilon}_\kappa)|_{\beta}-b_\sigma^\alpha(m^{\sigma\beta,\varepsilon}_\kappa|_\beta)+\dfrac{\varepsilon}{\kappa\sqrt{a}}\beta_\alpha(\bm{\zeta}^\varepsilon_\kappa)&=p^{\alpha,\varepsilon},\textup{ in }\omega,\\
	\\
	m^{\alpha\beta,\varepsilon}_\kappa|_{\alpha\beta}-b_\alpha^\sigma b_{\sigma\beta} m^{\alpha\beta,\varepsilon}_\kappa -b_{\alpha\beta} n^{\alpha\beta,\varepsilon}_\kappa+\dfrac{\varepsilon}{\kappa\sqrt{a}}\beta_3(\bm{\zeta}^\varepsilon_\kappa)&=p^{3,\varepsilon},\textup{ in }\omega,\\
	\\
	\zeta^\varepsilon_{\kappa,i}=\partial_{\nu}\zeta^\varepsilon_{\kappa,3}=0,\textup{ on }\gamma.
\end{cases}
\end{equation}

\section{Augmentation of the regularity of the solution of Problem~\ref{problem2}}
\label{sec:aug-interior}

Let $\omega_0\subset \omega$ and $\omega_1 \subset \omega$ be such that
\begin{equation}
\label{sets}
\omega_1 \subset\subset \omega_0 \subset \subset \omega.
\end{equation}

Let $\varphi_1 \in \mathcal{D}(\omega)$ be such that 
$$
\text{supp }\varphi_1 \subset\subset \omega_1 \textup{ and } 0\le \varphi_1 \le 1.
$$

By the definition of the symbol $\subset\subset$ in~\eqref{sets}, we obtain that the quantity
\begin{equation}
\label{d}
d=d(\varphi_1):=\dfrac{1}{2}\min\{\textup{dist}(\partial\omega_1,\partial\omega_0),\textup{dist}(\partial\omega_0,\gamma),\textup{dist}(\textup{supp }\varphi_1,\partial\omega_1)\}
\end{equation}
is strictly greater than zero. 

Denote by $D_{\rho h}$ the first order (forward) finite difference quotient of either a function or a vector field in the canonical direction $\bm{e}_\rho$ of $\mathbb{R}^2$ and with increment size $0<h<d$ sufficiently small. We can regard the first order (forward) finite difference quotient of a function as a linear operator defined as follows:
$$
D_{\rho h}: L^2(\omega) \to L^2(\omega_0).
$$

The first order finite difference quotient of a function $\xi$ in the canonical direction $\bm{e}_\rho$ of $\mathbb{R}^2$ and with increment size $0<h<d$ is defined by:
$$
D_{\rho h}\xi(y):=\dfrac{\xi(y+h\bm{e}_\rho)-\xi(y)}{h},
$$
for all (or, possibly, a.a.) $y\in\omega$ such that $(y+h\bm{e}_\rho)\in\omega$.

The first order finite difference quotient of a vector field $\bm{\xi}=(\xi_i)$ in the canonical direction $\bm{e}_\rho$ of $\mathbb{R}^2$ and with increment size $0<h<d$ is defined by
$$
D_{\rho h}\bm{\xi}(y):=\dfrac{\bm{\xi}(y+h\bm{e}_\rho)-\bm{\xi}(y)}{h},
$$
or, equivalently,
$$
D_{\rho h}\bm{\xi}(y)=(D_{\rho h}\xi_i(y)).
$$

Similarly, we can show that the first order (forward) finite difference quotient of a vector field is a linear operator from $\bm{L}^2(\omega)$ to $\bm{L}^2(\omega_0)$. 

We define the second order finite difference quotient of a function $\xi$ in the canonical direction $\bm{e}_\rho$ of $\mathbb{R}^2$ and with increment size $0<h<d$ by
$$
\delta_{\rho h}\xi(y):=\dfrac{\xi(y+h \bm{e}_\rho)-2 \xi(y)+\xi(y-h \bm{e}_\rho)}{h^2},
$$
for all (or, possibly, a.a.) $y \in \omega$ such that $(y\pm h\bm{e}_\rho) \in \omega$.

The second order finite difference quotient of a vector field $\bm{\xi}=(\xi_i)$ in the canonical direction $\bm{e}_\rho$ of $\mathbb{R}^2$ and with increment size $0<h<d$ is defined by
$$
\delta_{\rho h}\bm{\xi}(y):=\left(\dfrac{\xi_i(y+h \bm{e}_\rho)-2 \xi_i(y)+\xi_i(y-h \bm{e}_\rho)}{h^2}\right),
$$
for all (or, possibly, a.a.) $y \in \omega$ such that $(y\pm h\bm{e}_\rho) \in \omega$.

Define, following page~293 of~\cite{Evans2010}, the mapping $D_{-\rho h}:L^2(\omega) \to L^2(\omega_0)$ by
$$
D_{-\rho h}\xi(y):=\dfrac{\xi(y)-\xi(y-h\bm{e}_\rho)}{h},
$$
as well as the mapping $D_{-\rho h}:\bm{L}^2(\omega) \to \bm{L}^2(\omega_0)$ by
$$
D_{-\rho h}\bm{\xi}(y):=\dfrac{\bm{\xi}(y)-\bm{\xi}(y-h\bm{e}_\rho)}{h}.
$$

Note in passing that the second order finite difference quotient of a function $\xi$ can be expressed in terms of the first order finite difference quotient via the following identity:
\begin{equation*}
\label{ide}
\delta_{\rho h} \xi=D_{-\rho h} D_{\rho h} \xi.
\end{equation*}

Similarly, the second order finite difference quotient of a vector field $\bm{\xi}=(\xi_i)$ can be expressed in terms of the first order finite difference quotient via the following identity:
\begin{equation*}
\label{ide2}
\delta_{\rho h} \bm{\xi}=D_{-\rho h} D_{\rho h} \bm{\xi}.
\end{equation*}

	Let us define the translation operator $E$ in the canonical direction $\bm{e}_\rho$ of $\mathbb{R}^2$ and with increment size $0<h<d$ for a smooth enough function $v:\omega_0 \to \mathbb{R}$ by
\begin{align*}
E_{\rho h} v(y)&:=v(y+h \bm{e}_\rho),\\
E_{-\rho h} v(y)&:=v(y-h \bm{e}_\rho).
\end{align*}

Moreover, the following identities can be easily checked out (cf.\, page~293 of~\cite{Evans2010}, \cite{Frehse1971} and~\cite{Pie2020-1}):
\begin{align}
D_{\rho h}(v w)&=(E_{\rho h} w) (D_{\rho h} v)+v D_{\rho h} w, \label{D+}\\
D_{-\rho h}(v w)&=(E_{-\rho h} w )(D_{-\rho h} v)+v D_{-\rho h} w, \label{D-}\\
\delta_{\rho h}(vw)&=w \delta_{\rho h} v+(D_{\rho h}w )(D_{\rho h} v) +(D_{-\rho h}w)( D_{-\rho h} v)+v\delta_{\rho h}w.\label{delta+}
\end{align}

%Observe that the following identity holds for a nonvanishing function $g$:
%\begin{equation}
%	\label{fdq-ratio}
%	D_{\rho h}\left(\dfrac{f}{g}\right)=\dfrac{g (D_{\rho h}f)-f (D_{\rho h}g)}{g (E_{\rho h}g)}.
%\end{equation}

%Observe that the following identities hold:
%\begin{align}
%D_{-\rho h}\xi(y)&=D_{\rho h}\xi(y-h \bm{e}_\rho)=D_{\rho h}(E_{-\rho h} \xi)(y), \label{P1}\\
%D_{-\rho h}\bm{\xi}(y)&=D_{\rho h}\bm{\xi}(y-h \bm{e}_\rho)=D_{\rho h}(E_{-\rho h} \bm{\xi})(y) \label{P2}.
%\end{align}

We observe that the following properties hold for finite difference quotients.

The proof of the first lemma can be found in Lemma~4 of~\cite{Pie-2022-interior} and for this reason it is omitted.
\begin{lemma}
	\label{lem:fdq-1}
	Let $\{v_k\}_{k\ge1}$ be a sequence in $\mathcal{C}^1(\overline{\omega})$ that converges to a certain element $v \in H^1(\omega)$ with respect to the norm $\|\cdot\|_{H^1(\omega)}$.
	Then, we have that for all $0<h<d$ and all $\rho\in\{1,2\}$, 
	\begin{equation*}
	\label{conv1}
	D_{\rho h} v\in H^1(\omega_0) \textup{ with }\partial_\alpha(D_{\rho h} v)=D_{\rho h} (\partial_\alpha v) \quad\textup{ and }\quad D_{\rho h} v_k\to D_{\rho h} v \textup{ in }H^1(\omega_0) \textup{ as } k\to\infty.
	\end{equation*}
	\qed
\end{lemma}

As a direct consequence of Lemma~\ref{lem:fdq-1} and the definition of second order finite difference quotient, if $\{v_k\}_{k\ge1}$ is a sequence in $\mathcal{C}^1(\overline{\omega})$ that converges to a certain element $v \in H^1(\omega)$ with respect to the norm $\|\cdot\|_{H^1(\omega)}$, then, we have that for all $0<h<d$ and all $\rho\in\{1,2\}$, 
\begin{equation*}
\label{conv1-delta}
\delta_{\rho h} v\in H^1(\omega_1) \textup{ with }\partial_\alpha(\delta_{\rho h} v)=\delta_{\rho h} (\partial_\alpha v) \quad\textup{ and }\quad \delta_{\rho h} v_k\to \delta_{\rho h} v \textup{ in }H^1(\omega_1) \textup{ as } k\to\infty.
\end{equation*}

We also state the following elementary lemma, which exploits the compactness of the support of the test function $\varphi_1$ defined beforehand.
\begin{lemma}
\label{fdq-neg-part}
Let $f \in\mathcal{D}(\omega)$ with $\textup{supp }f \subset\subset \omega_1$.
Let $0<h<d$, where $d>0$ has been defined in~\eqref{d} and let $\rho\in\{1,2\}$ be given. Then,
\begin{equation*}
\int_{\omega} D_{\rho h}(-f^{-}) D_{\rho h}(f^{+}) \dd y \ge 0.
\end{equation*}
\end{lemma}
\begin{proof}
By the definition of $D_{\rho h}$ and the definition of the positive and negative part of a function, we have that
\begin{equation*}
\int_{\omega} D_{\rho h}(-f^{-}) D_{\rho h}(f^{+}) \dd y
=-\int_{\omega} \left(\dfrac{f^{-}(y+h\bm{e}_\rho)-f^{-}(y)}{h}\right) \left(\dfrac{f^{+}(y+h\bm{e}_\rho)-f^{+}(y)}{h}\right) \dd y.
\end{equation*}

Therefore, the integrand under examination is always greater or equal than zero, as it was to be proved.
\end{proof}

We are ready to state the first main result of this section, that amounts to establishing that, under reasonable sufficient conditions on the problem data, the solution $\bm{\zeta}^\varepsilon_\kappa$ of Problem~\ref{problem2} is more regular in the interior of the definition domain.

\begin{theorem}
\label{aug:int}
Let $\omega_0$ and $\omega_1$ be as in~\eqref{sets}. Assume that there exists a unit norm vector $\bm{q} \in \mathbb{E}^3$ such that
\begin{equation*}
\min_{y \in \overline{\omega}} (\bm{\theta} (y) \cdot \bm{q}) > 0
\textup{ and }
\min_{y \in \overline{\omega}} (\bm{a}_3 (y) \cdot \bm{q}) > 0.
\end{equation*}

Assume also that the vector field $\bm{f}^\varepsilon=(f^{i,\varepsilon})$ defining the applied body force density is of class $L^2(\Omega^\varepsilon)\times L^2(\Omega^\varepsilon)\times L^2(\Omega^\varepsilon)$.
Then, the solution $\bm{\zeta}^\varepsilon_\kappa=(\zeta^\varepsilon_{\kappa,i})$ of Problem~\ref{problem2} is of class $\bm{V}_K(\omega)\cap H^2_{\textup{loc}}(\omega) \times H^2_{\textup{loc}}(\omega) \times H^3_{\textup{loc}}(\omega)$.
\end{theorem}
\begin{proof}
Fix $\varphi\in\mathcal{D}(\omega)$ such that $\text{supp }\varphi \subset\subset \omega_1$ and $0\le \varphi \le 1$. Let $\bm{\zeta}^\varepsilon_\kappa \in \bm{V}_K(\omega)$ be the unique solution of Problem~\ref{problem2}.
Observe that each component $\bm{\zeta}^\varepsilon_\kappa$ can be extended outside of $\omega$ by zero.
By Proposition~9.18 of~\cite{Brez11} , the prolongation by zero outside of $\omega$ is the only admissible choice for which the extended vector field is of class $H^1_0(\mathbb{R}^2) \times H^1_0(\mathbb{R}^2) \times H^2_0(\mathbb{R}^2)$. Therefore, it makes sense to consider the vector field
\begin{equation*}
(-\varphi \delta_{\rho h}(\varphi\bm{\zeta}^\varepsilon_\kappa)) \in H^1(\mathbb{R}^2) \times H^1(\mathbb{R}^2) \times H^2(\mathbb{R}^2).
\end{equation*}

Since the support of this vector field is compactly contained in $\omega_1$, we obtain that, actually,
\begin{equation*}
(-\varphi \delta_{\rho h}(\varphi\bm{\zeta}^\varepsilon_\kappa)) \in \bm{V}_K(\omega),
\end{equation*}
and we can specialize $\bm{\eta}=-\varphi \delta_{\rho h}(\varphi\bm{\zeta}^\varepsilon_\kappa)$ in the variational equations of Problem~\ref{problem2}.

Let us now observe that, thanks to H\"older's inequality, the following estimate holds:
\begin{equation}
\label{add-lab-1}
\begin{aligned}
&\int_{\omega} p^{i,\varepsilon} (-\varphi \delta_{\rho h}(\varphi\zeta^\varepsilon_{\kappa,i})) \sqrt{a} \dd y
=-\int_{\omega_1}(\varphi p^{i,\varepsilon}) (\delta_{\rho h}(\varphi\zeta^\varepsilon_{\kappa,i})) \sqrt{a} \dd y\\
&\le\varepsilon \|\varphi\|_{\mathcal{C}^1(\overline{\omega})}\|\bm{p}\|_{L^2(\omega)\times L^2(\omega) \times L^2(\omega)}\sqrt{a_1}
\|D_{\rho h}(\varphi \bm{\zeta}^\varepsilon_\kappa)\|_{H^1(\omega_1)\times H^1(\omega_1)\times H^2(\omega_1)}.
\end{aligned}
\end{equation}

Note in passing that $\|\bm{p}\|_{L^2(\omega)\times L^2(\omega) \times L^2(\omega)}$ is independent of $\varepsilon$ thanks to the assumptions on the data.

Thanks to these inequalities, we have that
\begin{equation}
\label{int-1}
\begin{aligned}
&\varepsilon\int_{\omega_1} a^{\alpha\beta\sigma\tau} \gamma_{\sigma\tau}(\bm{\zeta}^\varepsilon_\kappa) \gamma_{\alpha\beta}(-\varphi \delta_{\rho h}(\varphi\bm{\zeta}^\varepsilon_\kappa)) \sqrt{a} \dd y\\
&\quad+\dfrac{\varepsilon^3}{3}\int_{\omega_1} a^{\alpha\beta\sigma\tau} \rho_{\sigma\tau}(\bm{\zeta}^\varepsilon_\kappa) \rho_{\alpha\beta}(-\varphi \delta_{\rho h}(\varphi\bm{\zeta}^\varepsilon_\kappa)) \sqrt{a} \dd y\\
&\quad+\dfrac{\varepsilon}{\kappa}\int_{\omega_1}\bm{\beta}(\bm{\zeta}^\varepsilon_\kappa) \cdot (-\varphi \delta_{\rho h}(\varphi\bm{\zeta}^\varepsilon_\kappa)) \dd y\\
&\le \varepsilon \|\varphi\|_{\mathcal{C}^1(\overline{\omega})}\|\bm{p}\|_{L^2(\omega)\times L^2(\omega) \times L^2(\omega)}\sqrt{a_1}
\|D_{\rho h}(\varphi \bm{\zeta}^\varepsilon_\kappa)\|_{H^1(\omega_1)\times H^1(\omega_1)\times H^2(\omega_1)}.
\end{aligned}
\end{equation}

Proceeding as in~\cite{Pie-2022-interior}, one can show that~\eqref{upper-bound-solution}, \eqref{add-lab-1} and~\eqref{int-1} lead to:
\begin{equation}
\label{key-relation-2}
\begin{aligned}
&-\varepsilon \int_{\omega_1}a^{\alpha\beta\sigma\tau}\gamma_{\sigma\tau}(\varphi\bm{\zeta}^\varepsilon_\kappa)\gamma_{\alpha\beta}(\delta_{\rho h}(\varphi \bm{\zeta}^\varepsilon_\kappa))\sqrt{a} \dd y\\
&\le -\varepsilon \int_{\omega_1} a^{\alpha\beta\sigma\tau}\gamma_{\sigma\tau}(\bm{\zeta}^\varepsilon_\kappa)\gamma_{\alpha\beta}(\varphi \delta_{\rho h}(\varphi \bm{\zeta}^\varepsilon_\kappa))\sqrt{a} \dd y+\dfrac{C}{\varepsilon^2}(1+\|D_{\rho h}(\varphi \bm{\zeta}^\varepsilon_\kappa)\|_{H^1(\omega_1)\times H^1(\omega_1)\times H^2(\omega_1)}),
\end{aligned}
\end{equation}
for some $C>0$ independent of $\varepsilon$, $\kappa$ and $h$.

The second step of our analysis consists in showing the following estimate:
\begin{equation}
\label{key-relation-3}
\begin{aligned}
&-\dfrac{\varepsilon^3}{3} \int_{\omega_1}a^{\alpha\beta\sigma\tau}\rho_{\sigma\tau}(\varphi\bm{\zeta}^\varepsilon_\kappa)\rho_{\alpha\beta}(\delta_{\rho h}(\varphi \bm{\zeta}^\varepsilon_\kappa))\sqrt{a} \dd y\\
&\le -\dfrac{\varepsilon^3}{3} \int_{\omega_1} a^{\alpha\beta\sigma\tau}\rho_{\sigma\tau}(\bm{\zeta}^\varepsilon_\kappa)\rho_{\alpha\beta}(\varphi \delta_{\rho h}(\varphi \bm{\zeta}^\varepsilon_\kappa))\sqrt{a} \dd y+C\varepsilon(1+\|D_{\rho h}(\varphi \bm{\zeta}^\varepsilon_\kappa)\|_{H^1(\omega_1)\times H^1(\omega_1)\times H^2(\omega_1)}),
\end{aligned}
\end{equation}
for some $C>0$ independent of $\varepsilon$, $\kappa$ and $h$.
In order to show the validity of~\eqref{key-relation-3}, it suffices to observe that:
\begin{equation*}
\begin{aligned}
&\int_{\omega_1} a^{\alpha\beta\sigma\tau} \partial_{\sigma\tau} \zeta^\varepsilon_{\kappa,3} \partial_{\alpha\beta}(\varphi \delta_{\rho h}(\varphi \zeta^\varepsilon_{\kappa,3})) \sqrt{a} \dd y\\
&=\int_{\omega_1} a^{\alpha\beta\sigma\tau} \partial_{\sigma\tau}\zeta^\varepsilon_{\kappa,3} \partial_\alpha((\partial_\beta\varphi) \delta_{\rho h}(\varphi\zeta^\varepsilon_{\kappa,3})+\varphi \partial_\beta(\delta_{\rho h}(\varphi\zeta^\varepsilon_{\kappa,3})))\sqrt{a} \dd y\\
&=\int_{\omega_1} a^{\alpha\beta\sigma\tau} \partial_{\sigma\tau}\zeta^\varepsilon_{\kappa,3} [(\partial_{\alpha\beta}\varphi)\delta_{\rho h}(\varphi\zeta^\varepsilon_{\kappa,3})+(\partial_\beta\varphi)\partial_\alpha(\delta_{\rho h}(\varphi\zeta^\varepsilon_{\kappa,3}))+(\partial_\alpha\varphi) \partial_\beta(\delta_{\rho h}(\varphi\zeta^\varepsilon_{\kappa,3}))+\varphi\partial_{\alpha\beta}(\delta_{\rho h}(\varphi\zeta^\varepsilon_{\kappa,3}))] \sqrt{a} \dd y\\
&\le 3 \left(\max_{\alpha,\beta,\sigma,\tau \in \{1,2\}}\|a^{\alpha\beta\sigma\tau}\|_{\mathcal{C}^0(\overline{\omega})}\right) \sqrt{a_1} \|\zeta^\varepsilon_{\kappa,3}\|_{H^2(\omega)} \|\varphi\|_{\mathcal{C}^2(\overline{\omega})} \|D_{\rho h}(\varphi \zeta^\varepsilon_{\kappa,3})\|_{H^2(\omega_1)}\\
&\quad+\int_{\omega_1}a^{\alpha\beta\sigma\tau} (\partial_{\sigma\tau}\zeta^\varepsilon_{\kappa,3}) [\varphi \partial_{\alpha\beta}(\delta_{\rho h}(\varphi \zeta^\varepsilon_{\kappa,3}))]\sqrt{a} \dd y\\
&=3 \left(\max_{\alpha,\beta,\sigma,\tau \in \{1,2\}}\|a^{\alpha\beta\sigma\tau}\|_{\mathcal{C}^0(\overline{\omega})}\right) \sqrt{a_1} \|\zeta^\varepsilon_{\kappa,3}\|_{H^2(\omega)} \|\varphi\|_{\mathcal{C}^2(\overline{\omega})} \|D_{\rho h}(\varphi \zeta^\varepsilon_{\kappa,3})\|_{H^2(\omega_1)}\\
&\quad+\int_{\omega_1} a^{\alpha\beta\sigma\tau} [\partial_\sigma(\varphi \partial_\tau \zeta^\varepsilon_{\kappa,3})-(\partial_\sigma \varphi) (\partial_\tau \zeta^\varepsilon_{\kappa,3})] \partial_{\alpha\beta}(\delta_{\rho h}(\varphi \zeta^\varepsilon_{\kappa,3})) \sqrt{a} \dd y\\
&=3 \left(\max_{\alpha,\beta,\sigma,\tau \in \{1,2\}}\|a^{\alpha\beta\sigma\tau}\|_{\mathcal{C}^0(\overline{\omega})}\right) \sqrt{a_1} \|\zeta^\varepsilon_{\kappa,3}\|_{H^2(\omega)} \|\varphi\|_{\mathcal{C}^2(\overline{\omega})} \|D_{\rho h}(\varphi \zeta^\varepsilon_{\kappa,3})\|_{H^2(\omega_1)}\\
&\quad+\int_{\omega_1} a^{\alpha\beta\sigma\tau} \partial_{\sigma\tau}(\varphi\zeta^\varepsilon_{\kappa,3}) \partial_{\alpha\beta}(\delta_{\rho h}(\varphi\zeta^\varepsilon_{\kappa,3})) \sqrt{a} \dd y\\
&\quad-\int_{\omega_1}a^{\alpha\beta\sigma\tau} [(\partial_{\sigma\tau}\varphi)\zeta^\varepsilon_{\kappa,3}+(\partial_\tau \varphi)(\partial_\sigma \zeta^\varepsilon_{\kappa,3})+(\partial_\sigma \varphi)(\partial_\tau\zeta^\varepsilon_{\kappa,3})] \partial_{\alpha\beta}(\delta_{\rho h}(\varphi \zeta^\varepsilon_{\kappa,3})) \sqrt{a} \dd y\\
&=3 \left(\max_{\alpha,\beta,\sigma,\tau \in \{1,2\}}\|a^{\alpha\beta\sigma\tau}\|_{\mathcal{C}^0(\overline{\omega})}\right) \sqrt{a_1} \|\zeta^\varepsilon_{\kappa,3}\|_{H^2(\omega)} \|\varphi\|_{\mathcal{C}^2(\overline{\omega})} \|D_{\rho h}(\varphi \zeta^\varepsilon_{\kappa,3})\|_{H^2(\omega_1)}\\
&\quad+\int_{\omega_1} a^{\alpha\beta\sigma\tau} \partial_{\sigma\tau}(\varphi\zeta^\varepsilon_{\kappa,3}) \partial_{\alpha\beta}(\delta_{\rho h}(\varphi\zeta^\varepsilon_{\kappa,3})) \sqrt{a} \dd y\\
&\quad+\int_{\omega_1}\partial_\alpha(a^{\alpha\beta\sigma\tau} [(\partial_{\sigma\tau}\varphi)\zeta^\varepsilon_{\kappa,3}+(\partial_\tau \varphi)(\partial_\sigma \zeta^\varepsilon_{\kappa,3})+(\partial_\sigma \varphi)(\partial_\tau\zeta^\varepsilon_{\kappa,3})] \sqrt{a})\partial_\beta(\delta_{\rho h}(\varphi \zeta^\varepsilon_{\kappa,3})) \dd y\\
& \le 6 \left(\max_{\alpha,\beta,\sigma,\tau \in \{1,2\}}\|a^{\alpha\beta\sigma\tau}\|_{\mathcal{C}^1(\overline{\omega})}\right) \|\sqrt{a}\|_{\mathcal{C}^1(\overline{\omega})} \|\zeta^\varepsilon_{\kappa,3}\|_{H^2(\omega)} \|\varphi\|_{\mathcal{C}^3(\overline{\omega})} \|D_{\rho h}(\varphi \zeta^\varepsilon_{\kappa,3})\|_{H^2(\omega_1)}\\
&\quad+\int_{\omega_1} a^{\alpha\beta\sigma\tau} \partial_{\sigma\tau}(\varphi\zeta^\varepsilon_{\kappa,3}) \partial_{\alpha\beta}(\delta_{\rho h}(\varphi\zeta^\varepsilon_{\kappa,3})) \sqrt{a} \dd y.
\end{aligned}
\end{equation*}

The other multiplications in the product of two arbitrary covariant components of the change of curvature tensors are estimated in a similar manner, and the inequality~\eqref{key-relation-3} is thus proved.

The fact that the second addend in the right-hand side of~\eqref{key-relation-3} is of order $\mathcal{O}(\varepsilon)$ descends from~\eqref{upper-bound-solution}.

We then have that the fact that $\varphi$ has compact support in $\omega_1$, Korn's inequality (Theorem~\ref{korn}), estimates~\eqref{add-lab-1}, \eqref{int-1} and~\eqref{key-relation-2}, the integration-by-parts formula for finite difference quotients (cf., e.g., formula~(10) on page~293 of~\cite{Evans2010}), and the definition of $d$ (viz. \eqref{d}) give
\begin{equation*}
\begin{aligned}
&\dfrac{\varepsilon^3\sqrt{a_0}}{3 c_0^2 c_e}\|D_{\rho h}(\varphi \bm{\zeta}^\varepsilon_\kappa)\|_{H^1(\omega_1)\times H^1(\omega_1)\times H^2(\omega_1)}^2
+\dfrac{\varepsilon}{\kappa}\int_{\omega_1}\bm{\beta}(\bm{\zeta}^\varepsilon_\kappa) \cdot (-\varphi \delta_{\rho h}(\varphi\bm{\zeta}^\varepsilon_\kappa)) \dd y\\
&\le\varepsilon \int_{\omega_1} a^{\alpha\beta\sigma\tau} \gamma_{\sigma\tau}(D_{\rho h}(\varphi\bm{\zeta}^\varepsilon_\kappa)) \gamma_{\alpha\beta}(D_{\rho h}(\varphi\bm{\zeta}^\varepsilon_\kappa)) \sqrt{a} \dd y\\
&\quad+\dfrac{\varepsilon^3}{3} \int_{\omega_1} a^{\alpha\beta\sigma\tau} \rho_{\sigma\tau}(D_{\rho h}(\varphi\bm{\zeta}^\varepsilon_\kappa)) \rho_{\alpha\beta}(D_{\rho h}(\varphi\bm{\zeta}^\varepsilon_\kappa)) \sqrt{a} \dd y
+\dfrac{\varepsilon}{\kappa}\int_{\omega_1}\bm{\beta}(\bm{\zeta}^\varepsilon_\kappa) \cdot (-\varphi \delta_{\rho h}(\varphi\bm{\zeta}^\varepsilon_\kappa)) \dd y\\
&\le \dfrac{C}{\varepsilon^2}(1+\|D_{\rho h}(\varphi\bm{\zeta}^\varepsilon_\kappa)\|_{H^1(\omega_1)\times H^1(\omega_1)\times H^2(\omega_1)})
-\varepsilon \int_{\omega_1} D_{\rho h}(a^{\alpha\beta\sigma\tau}\sqrt{a}) E_{\rho h}\left(\gamma_{\sigma\tau}(\varphi\bm{\zeta}^\varepsilon_\kappa)\right) \gamma_{\alpha\beta}(D_{\rho h}(\varphi\bm{\zeta}^\varepsilon_\kappa)) \dd y\\
&\quad-\dfrac{\varepsilon^3}{3} \int_{\omega_1} D_{\rho h}(a^{\alpha\beta\sigma\tau}\sqrt{a}) E_{\rho h}\left(\rho_{\sigma\tau}(\varphi\bm{\zeta}^\varepsilon_\kappa)\right) \rho_{\alpha\beta}(D_{\rho h}(\varphi\bm{\zeta}^\varepsilon_\kappa)) \dd y\\
&\le \dfrac{C}{\varepsilon^2}(1+\|D_{\rho h}(\varphi\bm{\zeta}^\varepsilon_\kappa)\|_{H^1(\omega_1)\times H^1(\omega_1)\times H^2(\omega_1)})\\
&\quad+\varepsilon \left(\max_{\alpha,\beta,\sigma,\tau \in \{1,2\}}\{\|a^{\alpha\beta\sigma\tau}\sqrt{a}\|_{\mathcal{C}^1(\overline{\omega})}\}\right) \left(\max_{\alpha,\beta\in\{1,2\}}\|\gamma_{\alpha\beta}(D_{\rho h}(\varphi\bm{\zeta}^\varepsilon_\kappa)\|_{L^2(\omega_1)}\right) \left(\max_{\sigma,\tau\in\{1,2\}}\|E_{\rho h}(\gamma_{\sigma\tau}(\varphi \bm{\zeta}^\varepsilon_\kappa))\|_{L^2(\omega_1)}\right)\\
&\quad+\dfrac{\varepsilon^3}{3} \left(\max_{\alpha,\beta,\sigma,\tau \in \{1,2\}}\{\|a^{\alpha\beta\sigma\tau}\sqrt{a}\|_{\mathcal{C}^1(\overline{\omega})}\}\right) \left(\max_{\alpha,\beta\in\{1,2\}}\|\rho_{\alpha\beta}(D_{\rho h}(\varphi\bm{\zeta}^\varepsilon_\kappa)\|_{L^2(\omega_1)}\right) \left(\max_{\sigma,\tau\in\{1,2\}}\|E_{\rho h}(\rho_{\sigma\tau}(\varphi \bm{\zeta}^\varepsilon_\kappa))\|_{L^2(\omega_1)}\right)\\
&= \dfrac{C}{\varepsilon^2}(1+\|D_{\rho h}(\varphi\bm{\zeta}^\varepsilon_\kappa)\|_{H^1(\omega_1)\times H^1(\omega_1)\times H^2(\omega_1)})\\
&\quad+\varepsilon \left(\max_{\alpha,\beta,\sigma,\tau \in \{1,2\}}\{\|a^{\alpha\beta\sigma\tau}\sqrt{a}\|_{\mathcal{C}^1(\overline{\omega})}\}\right) \left(\max_{\alpha,\beta\in\{1,2\}}\|\gamma_{\alpha\beta}(D_{\rho h}(\varphi\bm{\zeta}^\varepsilon_\kappa)\|_{L^2(\omega_1)}\right) \left(\max_{\sigma,\tau\in\{1,2\}}\|\gamma_{\sigma\tau}(\varphi \bm{\zeta}^\varepsilon_\kappa)\|_{L^2(\omega_0)}\right)\\
&\quad+\dfrac{\varepsilon^3}{3} \left(\max_{\alpha,\beta,\sigma,\tau \in \{1,2\}}\{\|a^{\alpha\beta\sigma\tau}\sqrt{a}\|_{\mathcal{C}^1(\overline{\omega})}\}\right) \left(\max_{\alpha,\beta\in\{1,2\}}\|\rho_{\alpha\beta}(D_{\rho h}(\varphi\bm{\zeta}^\varepsilon_\kappa)\|_{L^2(\omega_1)}\right) \left(\max_{\sigma,\tau\in\{1,2\}}\|\rho_{\sigma\tau}(\varphi \bm{\zeta}^\varepsilon_\kappa)\|_{L^2(\omega_0)}\right)\\
&= \dfrac{C}{\varepsilon^2}(1+\|D_{\rho h}(\varphi\bm{\zeta}^\varepsilon_\kappa)\|_{H^1(\omega_1)\times H^1(\omega_1)\times H^2(\omega_1)})\\
&\quad+2C\varepsilon^3 \left(\max_{\alpha,\beta,\sigma,\tau \in \{1,2\}}\{\|a^{\alpha\beta\sigma\tau}\sqrt{a}\|_{\mathcal{C}^1(\overline{\omega})}\}\right) 
\|D_{\rho h}(\varphi\bm{\zeta}^\varepsilon_\kappa)\|_{H^1(\omega_1)\times H^1(\omega_1)\times H^2(\omega_1)}
\|\bm{\zeta}^\varepsilon_\kappa\|_{H^1(\omega)\times H^1(\omega)\times H^2(\omega)}\\
&\le \dfrac{C}{\varepsilon^2}(1+\|D_{\rho h}(\varphi\bm{\zeta}^\varepsilon_\kappa)\|_{H^1(\omega_1)\times H^1(\omega_1)\times H^2(\omega_1)}),
\end{aligned}
\end{equation*}
where, once again, the constant $C>0$ is independent of $\varepsilon$, $\kappa$ and $h$. The latter computations summarize in the following result
\begin{equation}
\label{checkpoint-1}
\begin{aligned}
&\dfrac{\varepsilon^3\sqrt{a_0}}{3 c_0^2 c_e}\|D_{\rho h}(\varphi \bm{\zeta}^\varepsilon_\kappa)\|_{H^1(\omega_1)\times H^1(\omega_1)\times H^2(\omega_1)}^2
+\dfrac{\varepsilon}{\kappa}\int_{\omega_1}\bm{\beta}(\bm{\zeta}^\varepsilon_\kappa) \cdot (-\varphi \delta_{\rho h}(\varphi\bm{\zeta}^\varepsilon_\kappa)) \dd y\\
&\le \dfrac{C}{\varepsilon^2}(1+\|D_{\rho h}(\varphi\bm{\zeta}^\varepsilon_\kappa)\|_{H^1(\omega_1)\times H^1(\omega_1)\times H^2(\omega_1)}),
\end{aligned}
\end{equation}
for some constant $C>0$ is independent of $\varepsilon$, $\kappa$ and $h$.

Let us now estimate the penalty term.  By~\eqref{beta-2.5}, we obtain that:
\begin{equation}
\label{bdd-3}
\dfrac{1}{\sqrt{\kappa}}\|-\{(\bm{\theta}+\zeta^\varepsilon_{\kappa,i}\bm{a}^i)\cdot\bm{q}\}^{-}\|_{\bm{L}^2(\omega)} \le \dfrac{M_1}{\varepsilon},
\end{equation}
for some $M_1>0$ independent of $\varepsilon$ and $\kappa$.

Let us now evaluate the penalty term in the governing equations of Problem~\ref{problem2}. An application of formulas~\eqref{D+}, \eqref{D-}, \eqref{delta+} and Lemma~\ref{fdq-neg-part} gives:
\begin{align*}
&\dfrac{1}{\kappa}\int_{\omega_1}\bm{\beta}(\bm{\zeta}^\varepsilon_\kappa) \cdot (-\varphi \delta_{\rho h}(\varphi \bm{\zeta}^\varepsilon_\kappa))\dd y =-\dfrac{1}{\kappa}\int_{\omega_1}\left[-\varphi\{(\bm{\theta}+\zeta^\varepsilon_{\kappa,j}\bm{a}^j)\cdot\bm{q}\}^{-}\left(\dfrac{\bm{a}^i\cdot\bm{q}}{\sqrt{\sum_{\ell=1}^3|\bm{a}^\ell\cdot\bm{q}|^2}}\right)\right]\delta_{\rho h}(\varphi\zeta^\varepsilon_{\kappa,i})\dd y\\
&=\dfrac{1}{\kappa}\int_{\omega_1}D_{\rho h}\left(-\varphi\{(\bm{\theta}+\zeta^\varepsilon_{\kappa,j}\bm{a}^j)\cdot\bm{q}\}^{-} \left(\dfrac{\bm{a}^i\cdot\bm{q}}{\sqrt{\sum_{\ell=1}^3|\bm{a}^\ell\cdot\bm{q}|^2}}\right)\right) D_{\rho h}(\varphi\zeta^\varepsilon_{\kappa,i})\dd y\\
&=\dfrac{1}{\kappa}\int_{\omega_1}\left[D_{\rho h}\left(-\{(\bm{\theta}+\zeta^\varepsilon_{\kappa,j}\bm{a}^j)\cdot\bm{q}\}^{-} \varphi\right) E_{\rho h}\left(\dfrac{\bm{a}^i\cdot\bm{q}}{\sqrt{\sum_{\ell=1}^3|\bm{a}^\ell\cdot\bm{q}|^2}}\right)\right] D_{\rho h}(\varphi\zeta^\varepsilon_{\kappa,i})\dd y\\
&\quad+\dfrac{1}{\kappa}\int_{\omega_1} \left[\left(-\{(\bm{\theta}+\zeta^\varepsilon_{\kappa,j}\bm{a}^j)\cdot\bm{q}\}^{-} \varphi\right) D_{\rho h}\left(\dfrac{\bm{a}^i\cdot\bm{q}}{\sqrt{\sum_{\ell=1}^3|\bm{a}^\ell\cdot\bm{q}|^2}}\right)\right] D_{\rho h}(\varphi\zeta^\varepsilon_{\kappa,i})\dd y\\
&=\dfrac{1}{\kappa}\int_{\omega_1}E_{\rho h}\left(\dfrac{1}{\sqrt{\sum_{\ell=1}^3|\bm{a}^\ell\cdot\bm{q}|^2}}\right)\left[D_{\rho h}\left(-\{(\bm{\theta}+\zeta^\varepsilon_{\kappa,j}\bm{a}^j)\cdot\bm{q}\}^{-} \varphi\right)\right] D_{\rho h}\left(\varphi \zeta^\varepsilon_{\kappa,i}\bm{a}^i\cdot\bm{q}\right)\dd y\\
&\quad-\dfrac{1}{\kappa}\int_{\omega_1}E_{\rho h}\left(\dfrac{1}{\sqrt{\sum_{\ell=1}^3|\bm{a}^\ell\cdot\bm{q}|^2}}\right)\left[D_{\rho h}\left(-\{(\bm{\theta}+\zeta^\varepsilon_{\kappa,j}\bm{a}^j)\cdot\bm{q}\}^{-} \varphi\right)\right] (\varphi \zeta^\varepsilon_{\kappa,i}) D_{\rho h}\left(\bm{a}^i\cdot\bm{q}\right)\dd y\\
&\quad+\dfrac{1}{\kappa}\int_{\omega_1} \left(-\{(\bm{\theta}+\zeta^\varepsilon_{\kappa,j}\bm{a}^j)\cdot\bm{q}\}^{-} \varphi\right)\left(D_{\rho h}\left(\dfrac{\bm{a}^i\cdot\bm{q}}{\sqrt{\sum_{\ell=1}^3|\bm{a}^\ell\cdot\bm{q}|^2}}\right) D_{\rho h}(\varphi\zeta^\varepsilon_{\kappa,i})\right)\dd y\\
&=\dfrac{1}{\kappa}\int_{\omega_1}E_{\rho h}\left(\dfrac{1}{\sqrt{\sum_{\ell=1}^3|\bm{a}^\ell\cdot\bm{q}|^2}}\right)\left[D_{\rho h}\left(-\{(\bm{\theta}+\zeta^\varepsilon_{\kappa,j}\bm{a}^j)\cdot\bm{q}\}^{-} \varphi\right)\right] D_{\rho h}\left(\varphi (\bm{\theta}+\zeta^\varepsilon_{\kappa,i}\bm{a}^i)\cdot\bm{q}\right)\dd y\\
&\quad-\dfrac{1}{\kappa}\int_{\omega_1}E_{\rho h}\left(\dfrac{1}{\sqrt{\sum_{\ell=1}^3|\bm{a}^\ell\cdot\bm{q}|^2}}\right)\left[D_{\rho h}\left(-\{(\bm{\theta}+\zeta^\varepsilon_{\kappa,j}\bm{a}^j)\cdot\bm{q}\}^{-} \varphi\right)\right] D_{\rho h}\left(\varphi \bm{\theta}\cdot\bm{q}\right)\dd y\\
&\quad+\dfrac{1}{\kappa}\int_{\omega_1}\left(-\{(\bm{\theta}+\zeta^\varepsilon_{\kappa,j}\bm{a}^j)\cdot\bm{q}\}^{-} \varphi\right) D_{-\rho h}\left(E_{\rho h}\left(\dfrac{1}{\sqrt{\sum_{\ell=1}^3|\bm{a}^\ell\cdot\bm{q}|^2}}\right)(\varphi \zeta^\varepsilon_{\kappa,i}) D_{\rho h}\left(\bm{a}^i\cdot\bm{q}\right)\right)\dd y\\
&\quad+\dfrac{1}{\kappa}\int_{\omega_1} \left(-\{(\bm{\theta}+\zeta^\varepsilon_{\kappa,j}\bm{a}^j)\cdot\bm{q}\}^{-} \varphi\right)\left(D_{\rho h}\left(\dfrac{\bm{a}^i\cdot\bm{q}}{\sqrt{\sum_{\ell=1}^3|\bm{a}^\ell\cdot\bm{q}|^2}}\right) D_{\rho h}(\varphi\zeta^\varepsilon_{\kappa,i})\right)\dd y\\
&=\dfrac{1}{\kappa}\int_{\omega_1} E_{\rho h}\left(\dfrac{1}{\sqrt{\sum_{\ell=1}^3|\bm{a}^\ell\cdot\bm{q}|^2}}\right)\left[D_{\rho h}\left(-\{(\bm{\theta}+\zeta^\varepsilon_{\kappa,j}\bm{a}^j)\cdot\bm{q}\}^{-} \varphi\right)\right] D_{\rho h}\left(\varphi \{(\bm{\theta}+\zeta^\varepsilon_{\kappa,i}\bm{a}^i)\cdot\bm{q}\}^{+}\right)\dd y\\
&\quad+\dfrac{1}{\kappa}\int_{\omega_1} E_{\rho h}\left(\dfrac{1}{\sqrt{\sum_{\ell=1}^3|\bm{a}^\ell\cdot\bm{q}|^2}}\right)\left|D_{\rho h}\left(-\{(\bm{\theta}+\zeta^\varepsilon_{\kappa,j}\bm{a}^j)\cdot\bm{q}\}^{-} \varphi\right)\right|^2\dd y\\
&\quad+\dfrac{1}{\kappa}\int_{\omega_1}(-\varphi \{(\bm{\theta}+\zeta^\varepsilon_{\kappa,j}\bm{a}^j)\cdot\bm{q}\}^{-}) D_{-\rho h}\left[E_{\rho h}\left(\dfrac{1}{\sqrt{\sum_{\ell=1}^3|\bm{a}^\ell\cdot\bm{q}|^2}}\right) D_{\rho h}(\varphi \bm{\theta}\cdot\bm{q})\right] \dd y\\
&\quad+\dfrac{1}{\kappa}\int_{\omega_1}(-\varphi \{(\bm{\theta}+\zeta^\varepsilon_{\kappa,j}\bm{a}^j)\cdot\bm{q}\}^{-}) D_{-\rho h}\left[E_{\rho h}\left(\dfrac{1}{\sqrt{\sum_{\ell=1}^3|\bm{a}^\ell\cdot\bm{q}|^2}}\right) \left(D_{\rho h}(\bm{a}^i\cdot\bm{q})\right) (\varphi\zeta^\varepsilon_{\kappa,i})\right] \dd y\\
&\quad+\dfrac{1}{\kappa}\int_{\omega_1} \left(-\{(\bm{\theta}+\zeta^\varepsilon_{\kappa,j}\bm{a}^j)\cdot\bm{q}\}^{-} \varphi\right)\left(D_{\rho h}\left(\dfrac{\bm{a}^i\cdot\bm{q}}{\sqrt{\sum_{\ell=1}^3|\bm{a}^\ell\cdot\bm{q}|^2}}\right) D_{\rho h}(\varphi\zeta^\varepsilon_{\kappa,i})\right)\dd y.
\end{align*}

Applying the latter computations, the fact that $\bm{\theta}\in\mathcal{C}^3(\overline{\omega};\mathbb{E}^3)$, Lemma~\ref{fdq-neg-part}, the assumption according to which $\min_{y \in \overline{\omega}}(\bm{a}^3\cdot\bm{q})>0$ and the fact that $\textup{supp }\varphi \subset\subset \omega_1$ to~\eqref{checkpoint-1} gives:
\begin{align*}
	&\dfrac{\varepsilon^3\sqrt{a_0}}{3 c_0^2 c_e}\|D_{\rho h}(\varphi \bm{\zeta}^\varepsilon_\kappa)\|_{H^1(\omega_1)\times H^1(\omega_1)\times H^2(\omega_1)}^2\\
	&\quad+\varepsilon\dfrac{\left(3\max\{\|\bm{a}^\ell\cdot\bm{q}\|_{\mathcal{C}^0(\overline{\omega})}^2;1\le \ell \le 3\}\right)^{-1/2}}{\kappa}\int_{\omega_1}\left|D_{\rho h}\left(-\{(\bm{\theta}+\zeta^\varepsilon_{\kappa,j}\bm{a}^j)\cdot\bm{q}\}^{-} \varphi\right)\right|^2\dd y\\
	&\le \dfrac{C}{\varepsilon^2}(1+\|D_{\rho h}(\varphi\bm{\zeta}^\varepsilon_\kappa)\|_{H^1(\omega_1)\times H^1(\omega_1)\times H^2(\omega_1)})\\
	&\quad-\dfrac{\varepsilon}{\kappa}\int_{\omega_1}(-\varphi \{(\bm{\theta}+\zeta^\varepsilon_{\kappa,j}\bm{a}^j)\cdot\bm{q}\}^{-}) D_{-\rho h}\left[E_{\rho h}\left(\dfrac{1}{\sqrt{\sum_{\ell=1}^3|\bm{a}^\ell\cdot\bm{q}|^2}}\right) D_{\rho h}(\varphi \bm{\theta}\cdot\bm{q})\right] \dd y\\
	&\quad-\dfrac{\varepsilon}{\kappa}\int_{\omega_1}(-\varphi \{(\bm{\theta}+\zeta^\varepsilon_{\kappa,j}\bm{a}^j)\cdot\bm{q}\}^{-}) D_{-\rho h}\left[E_{\rho h}\left(\dfrac{1}{\sqrt{\sum_{\ell=1}^3|\bm{a}^\ell\cdot\bm{q}|^2}}\right) \left(D_{\rho h}(\bm{a}^i\cdot\bm{q})\right) (\varphi\zeta^\varepsilon_{\kappa,i})\right] \dd y\\
	&\quad-\dfrac{\varepsilon}{\kappa}\int_{\omega_1} \left(-\{(\bm{\theta}+\zeta^\varepsilon_{\kappa,j}\bm{a}^j)\cdot\bm{q}\}^{-} \varphi\right)\left(D_{\rho h}\left(\dfrac{\bm{a}^i\cdot\bm{q}}{\sqrt{\sum_{\ell=1}^3|\bm{a}^\ell\cdot\bm{q}|^2}}\right) D_{\rho h}(\varphi\zeta^\varepsilon_{\kappa,i})\right)\dd y,
\end{align*}
for some constant $C>0$ independent of $\varepsilon$, $\kappa$ and $h$. 

An application of~\eqref{bdd-3}, H\"older's inequality, and the assumption $\min_{y\in\overline{\omega}}(\bm{a}^3(y)\cdot\bm{q})>0$ gives the following estimate:
\begin{equation}
\label{checkpoint-2}
\begin{aligned}
&\dfrac{\varepsilon^3\sqrt{a_0}}{3 c_0^2 c_e}\|D_{\rho h}(\varphi \bm{\zeta}^\varepsilon_\kappa)\|_{H^1(\omega_1)\times H^1(\omega_1)\times H^2(\omega_1)}^2\\
&\quad+\varepsilon\dfrac{\left(3\max\{\|\bm{a}^\ell\cdot\bm{q}\|_{\mathcal{C}^0(\overline{\omega})}^2;1\le \ell \le 3\}\right)^{-1/2}}{\kappa}\int_{\omega_1}\left|D_{\rho h}\left(-\{(\bm{\theta}+\zeta^\varepsilon_{\kappa,j}\bm{a}^j)\cdot\bm{q}\}^{-} \varphi\right)\right|^2\dd y\\
&\le \dfrac{C}{\varepsilon^2\sqrt{\kappa}}\left(1+\|D_{\rho h}(\varphi\bm{\zeta}^\varepsilon_\kappa)\|_{H^1(\omega_1)\times H^1(\omega_1)\times H^2(\omega_1)}\right).
\end{aligned}
\end{equation}

A direct consequence of~\eqref{checkpoint-2} is that:
\begin{equation}
\label{conclusion-1}
\dfrac{\sqrt{a_0}}{3 c_0^2 c_e}\|D_{\rho h}(\varphi \bm{\zeta}^\varepsilon_\kappa)\|_{H^1(\omega_1)\times H^1(\omega_1)\times H^2(\omega_1)}^2
-\dfrac{C}{\varepsilon^{5} \sqrt{\kappa}}\|D_{\rho h}(\varphi\bm{\zeta}^\varepsilon_\kappa)\|_{H^1(\omega_1)\times H^1(\omega_1)\times L^2(\omega_1)}-\dfrac{C}{\varepsilon^{5} \sqrt{\kappa}}\le 0.
\end{equation}

Regarding $\|D_{\rho h}(\varphi \bm{\zeta}^\varepsilon_\kappa)\|_{H^1(\omega_1)\times H^1(\omega_1)\times H^2(\omega_1)}$ as the variable of the corresponding second-degree polynomial $\frac{\sqrt{a_0}}{3 c_0^2 c_e} x^2 -\frac{C}{\varepsilon^{5} \sqrt{\kappa}} x -\frac{C}{\varepsilon^{5} \sqrt{\kappa}}$, we have that its discriminant is positive. Therefore, we have that the inequality~\eqref{conclusion-1} is satisfied for
\begin{equation}
	\label{conclusion-2}
0\le \|D_{\rho h}(\varphi \bm{\zeta}^\varepsilon_\kappa)\|_{H^1(\omega_1)\times H^1(\omega_1)\times H^2(\omega_1)} \le \dfrac{\dfrac{C}{\varepsilon^{5} \sqrt{\kappa}}+\sqrt{\dfrac{C^2}{\varepsilon^{10} \kappa}+4\frac{C\sqrt{a_0}}{3 c_0^2 c_e \varepsilon^{5} \sqrt{\kappa}}}}{\frac{2 \sqrt{a_0}}{3 c_0^2 c_e}},
\end{equation}
where the upper bound in~\eqref{conclusion-2} is independent of $h$. 
Applying~\eqref{conclusion-2} to~\eqref{checkpoint-2} gives that
\begin{equation*}
\label{conclusion-3}
\left\|D_{\rho h}\left(-\{(\bm{\theta}+\zeta^\varepsilon_{\kappa,j}\bm{a}^j)\cdot\bm{q}\}^{-} \varphi\right)\right\|_{L^2(\omega_1)}^2 \le \dfrac{C}{\varepsilon^8},
\end{equation*}
for some $C>0$ independent of $\varepsilon$, $\kappa$ and $h$, thus showing that the penalisation term is more regular in the interior of the domain $\omega$ for a given $\kappa$ and $\varepsilon$.
This completes the proof.
\end{proof}

As a remark, we observe that the higher regularity of the negative part of the constraint has been established without resorting by any means to Stampacchia's theorem~\cite{Stampacchia1965}.

Let us now show that the solution $\bm{\zeta}^\varepsilon_\kappa$ of Problem~\ref{problem2} enjoys the higher regularity established in Theorem~\ref{aug:int} up to the boundary of the domain $\omega$.

\begin{theorem}
\label{aug:bdry}
Assume that the boundary $\gamma$ of the domain $\omega$ is of class $\mathcal{C}^4$ and that $\bm{\theta} \in \mathcal{C}^4(\overline{\omega};\mathbb{E}^3)$.
Assume that there exists a unit-norm vector $\bm{q} \in \mathbb{E}^3$ such that
\begin{equation*}
\min_{y \in \overline{\omega}} (\bm{\theta} (y) \cdot \bm{q}) > 0 \quad
\textup{ and } \quad
\min_{y \in \overline{\omega}} (\bm{a}_3 (y) \cdot \bm{q}) > 0.
\end{equation*}

Assume also that the vector field $\bm{f}^\varepsilon=(f^{i,\varepsilon})$ defining the applied body force density is such that $\bm{p}^\varepsilon=(p^{i,\varepsilon}) \in H^1(\omega) \times H^1(\omega) \times L^2(\omega)$. Define $\bm{H}(\omega):=H^2(\omega) \times H^2(\omega) \times H^3(\omega)$.

Then, the solution $\bm{\zeta}^\varepsilon_\kappa=(\zeta^\varepsilon_{\kappa,i})$ of Problem~\ref{problem2} is of class $\bm{V}_K(\omega)\cap \bm{H}(\omega)$.
\end{theorem}
\begin{proof}
Consider the boundary value problem~\eqref{BVP}. Since the solution $\bm{\zeta}^\varepsilon_\kappa$ is of class $\bm{V}_K(\omega)$, we have that an application of the Stampacchia's theorem~\cite{Stampacchia1965} (see also Theorem~4.4 on page 153 of~\cite{EvansGariepy2015}) gives that:
\begin{equation*}
\dfrac{\varepsilon}{\kappa}\bm{\beta}(\bm{\zeta}^\varepsilon_\kappa) \in \bm{H}^1_0(\omega).
\end{equation*}

Therefore, an application of Proposition~9.1 in~\cite{Lions1957}, Theorem~4.4-5 of~\cite{Ciarlet2005} and of the Rellich-Kondra\v{s}ov theorem (cf., e.g., Theorem~6.6-3 of~\cite{PGCLNFAA}) gives that:
\begin{equation*}
\bm{\zeta}^\varepsilon_\kappa \in H^2(\omega) \times H^2(\omega) \times H^3(\omega) \hookrightarrow\hookrightarrow \bm{\mathcal{C}}^0(\overline{\omega}).
\end{equation*}

Since each component of the solution $\bm{\zeta}^\varepsilon_\kappa$ of Problem~\ref{problem2} is continuous, and since $\bm{\theta} \cdot\bm{q}>0$ in $\overline{\omega}$, we have that there exists a tubular neighbourhood $V_{\varepsilon,\kappa}$ of the smooth boundary $\gamma$ such that:
\begin{equation*}
(\bm{\theta}+\zeta^\varepsilon_{\kappa,j}\bm{a}^j) \cdot\bm{q}>0,\quad\textup{ in } V_{\varepsilon,\kappa}.
\end{equation*}

In particular, we have that $\bm{\beta}(\bm{\zeta}^\varepsilon_\kappa) =\bm{0}$ in $V_{\varepsilon,\kappa}$ and the conclusion follows by the standard augmentation of regularity argument near the boundary (cf., e.g., Theorem~4 on page~334 of~\cite{Evans2010}).
\end{proof}

Finally, we recall that the augmentation of regularity up to the boundary holds for domains with Lipschitz continuous boundary provided that these domains are convex (viz. \cite{Eggleston1958} and~\cite{Grisvard2011}).

\section{Approximation of the solution of Problem~\ref{problemK} via the Penalty Method}
\label{approx:original}

The main objective of this section is sharpen the strong convergence~\eqref{beta-5}. To do this, we will resort to the theory of saddle points (cf., e.g., \cite{Ciarlet1989} and~\cite{Glowinski1981}). The results presented in this section are ancillary, in the sense that the convergence of the numerical scheme proposed in the forthcoming sections (cf., Theorem~\ref{t:convergence} and Theorem~\ref{th:conv}) does not hinge on what we are about to prove.

Define the set $\Lambda$ by:
\begin{equation*}
\Lambda:=\{v \in L^2(\omega);  v(y) \ge 0 \textup{ for a.a. } y\in\omega\},
\end{equation*}
and observe that $\Lambda$ is a non-empty, closed, and convex subset of $L^2(\omega)$.

Let us consider the following Lagrangian functional associated with Problem~\ref{problemK}. Let $\mathcal{L}^\varepsilon:\bm{V}_K(\omega) \times \Lambda \to \mathbb{R}$, which is defined by:
\begin{equation}
\label{L}
\begin{aligned}
	\displaystyle
\mathcal{L}^\varepsilon(\bm{\eta},\phi)&:= \dfrac{\varepsilon}{2} \int_{\omega} a^{\alpha\beta\sigma\tau} \gamma_{\sigma\tau}(\bm{\eta}) \gamma_{\alpha\beta}(\bm{\eta}) \sqrt{a} \dd y
+\dfrac{\varepsilon^3}{6}\int_\omega a^{\alpha\beta\sigma\tau} \rho_{\sigma\tau}(\bm{\eta}) \rho_{\alpha\beta}(\bm{\eta}) \sqrt{a} \dd y\\
&\quad-\int_{\omega} p^{i,\varepsilon} \eta_i \sqrt{a} \dd y
-\dfrac{\varepsilon^3}{3} \int_{\omega} \phi [(\bm{\theta}+\eta_i \bm{a}^i)\cdot\bm{q}] \dd y,
\end{aligned}
\end{equation}
for all $(\bm{\eta},\phi) \in \bm{V}_K(\omega) \times \Lambda$.

In the next theorem, we show that the Lagrangian functional $\mathcal{L}^\varepsilon$ associated with Problem~\ref{problemK} has at least one saddle point in the set $\bm{V}_K(\omega) \times \Lambda$. A saddle point for $\mathcal{L}^\varepsilon$ in $\bm{V}_K(\omega) \times \Lambda$ is an element $(\bm{\xi}^\varepsilon,\varphi^\varepsilon) \in \bm{V}_K(\omega) \times \Lambda$ satisfying
\begin{equation}
\label{saddle-1}
\mathcal{L}^\varepsilon(\bm{\xi}^\varepsilon,\phi) \le \mathcal{L}^\varepsilon(\bm{\xi}^\varepsilon,\varphi^\varepsilon) \le \mathcal{L}^\varepsilon(\bm{\eta},\varphi^\varepsilon), \quad\textup{ for all } (\bm{\eta},\phi) \in \bm{V}_K(\omega) \times \Lambda.
\end{equation}

Thanks to Proposition~1.6 in Chapter~VI of~\cite{EkelandTemam} (see also~\cite{Haslinger1981}), we have that finding a saddle point of~\eqref{saddle-1} is equivalent to finding a pair $(\bm{\xi}^\varepsilon,\varphi^\varepsilon) \in \bm{V}_K(\omega) \times \Lambda$ satisfying
\begin{equation}
\label{saddle-2}
\begin{cases}
&\displaystyle\varepsilon\int_{\omega} a^{\alpha\beta\sigma\tau} \gamma_{\sigma\tau}(\bm{\xi}^\varepsilon) \gamma_{\alpha\beta}(\bm{\eta}-\bm{\xi}^\varepsilon) \sqrt{a} \dd y
+\dfrac{\varepsilon^3}{3}\int_\omega a^{\alpha\beta\sigma\tau} \rho_{\sigma\tau}(\bm{\xi}^\varepsilon) \rho_{\alpha\beta}(\bm{\eta}-\bm{\xi}^\varepsilon) \sqrt{a} \dd y\\
&\quad\displaystyle-\dfrac{\varepsilon^3}{3}\int_{\omega}\varphi^\varepsilon(\eta_i-\xi^\varepsilon_i)\bm{a}^i \cdot\bm{q} \dd y \ge \int_{\omega}p^{i,\varepsilon} (\eta_i-\xi^\varepsilon_i) \sqrt{a} \dd y, \quad\textup{ for all }\bm{\eta}=(\eta_i) \in \bm{V}_K(\omega),\\
\\
&\displaystyle\dfrac{\varepsilon^3}{3}\int_{\omega}(\phi-\varphi^\varepsilon) (\bm{\theta}+\xi^\varepsilon_i \bm{a}^i) \cdot\bm{q} \dd y\ge 0,\quad\textup{ for all }\phi \in \Lambda.
\end{cases}
\end{equation}

In the next theorem, we outline the properties of the saddle points of the Lagrangian functional $\mathcal{L}^\varepsilon$.

\begin{theorem}
\label{ex:saddle}
If the Lagrangian functional $\mathcal{L}^\varepsilon:\bm{V}_K(\omega) \times L^2(\omega) \to \mathbb{R}$ defined in~\eqref{L} has at least one saddle point $(\bm{\xi}^\varepsilon,\varphi^\varepsilon) \in \bm{V}_K(\omega) \times \Lambda$, then the first component of one such saddle point is uniquely determined, and coincides with the unique solution of Problem~\ref{problemK}.
\end{theorem}
\begin{proof}
By Proposition~1.4 in Chapter VI of~\cite{EkelandTemam}, the set of saddle points for $\mathcal{L}^\varepsilon$ is of the form $\bm{\mathcal{A}}_0 \times \mathcal{B}_0$, where
\begin{equation*}
\bm{\mathcal{A}}_0 \subset \bm{V}_K(\omega) \quad\textup{ and }\quad \mathcal{B}_0\subset \Lambda. 
\end{equation*}

The bilinear form
\begin{equation*}
(\bm{\eta},\bm{\xi}) \in \bm{V}_K(\omega) \times \bm{V}_K(\omega) \mapsto \varepsilon\int_{\omega} a^{\alpha\beta\sigma\tau} \gamma_{\sigma\tau}(\bm{\eta}) \gamma_{\alpha\beta}(\bm{\xi}) \sqrt{a} \dd y
+\dfrac{\varepsilon^3}{3}\int_\omega a^{\alpha\beta\sigma\tau} \rho_{\sigma\tau}(\bm{\eta}) \rho_{\alpha\beta}(\bm{\xi}) \sqrt{a} \dd y,
\end{equation*}
is symmetric, continuous and $\bm{V}_K(\omega)$-elliptic (cf., e.g., Theorem~\ref{kornsurface}). Observe that, given any $\phi \in \Lambda$, the mapping
\begin{equation*}
\bm{\eta} \in \bm{V}_K(\omega) \mapsto \mathcal{L}^\varepsilon(\bm{\eta},\phi),
\end{equation*}
is strictly convex as a sum of quadratic functionals and linear functionals, and continuous.

The bilinear form
\begin{equation*}
(\bm{\eta},\phi) \in \bm{V}_K(\omega) \times \Lambda \mapsto 
-\dfrac{\varepsilon^3}{3} \int_{\omega} \phi (\eta_i \bm{a}^i\cdot\bm{q}) \dd y,
\end{equation*}
is continuous and such that, for all $\bm{\eta}=(\eta_i)\in \bm{V}_K(\omega)$,
\begin{equation*}
\phi \in L^2(\omega) \mapsto \mathcal{L}^\varepsilon(\bm{\eta},\phi),
\end{equation*}
is linear (and so, \emph{a fortiori}, concave) and continuous. Therefore, by Proposition~1.5 in Chapter VI of~\cite{EkelandTemam}, the set $\bm{\mathcal{A}}_0$ contains at most one point, which we denote by $\bm{\xi}^\varepsilon$.

Let us now check that $\bm{\xi}^\varepsilon \in \bm{U}_K(\omega)$.
By the second set of variational inequalities in~\eqref{saddle-2} and the definition of $\Lambda$, we obtain that specialising $\phi=\varphi^\varepsilon+\{(\bm{\theta}+\xi^\varepsilon_i\bm{a}^i)\cdot\bm{q}\}^{-}$ gives:
\begin{equation}
\label{id-saddle-2}
\int_\omega \left|\{(\bm{\theta}+\xi^\varepsilon_i \bm{a}^i)\cdot\bm{q}\}^{-}\right|^2 \dd y \le 0,
\end{equation}
which in turn implies that $\bm{\xi}^\varepsilon \in \bm{U}_K(\omega)$. 

By the second set of variational inequalities in~\eqref{saddle-2} and the definition of $\Lambda$, we obtain that specialising $\phi=0$ and $\phi=2\varphi^\varepsilon$ gives
\begin{equation}
\label{id-saddle}
\int_\omega \varphi^\varepsilon [(\bm{\theta}+\xi^\varepsilon_i \bm{a}^i)\cdot\bm{q}] \dd y =0.
\end{equation}

Specialising $\bm{\eta} \in \bm{U}_K(\omega)$ in the first set of variational inequalities of~\eqref{saddle-2} and exploiting~\eqref{id-saddle-2} and~\eqref{id-saddle} gives that $\bm{\xi}^\varepsilon$ solves Problem~\ref{problemK}.
This completes the proof.
\end{proof}

Thanks to Theorem~\ref{ex:saddle}, we are able to estimate the norm of the difference of the solution $\bm{\zeta}^\varepsilon$ of Problem~\ref{problemK} and the solution $\bm{\zeta}^\varepsilon_\kappa$ of Problem~\ref{problem2}.

\begin{theorem}
\label{difference-1}
Let $(\bm{\zeta}^\varepsilon,\psi^\varepsilon) \in \bm{V}_K(\omega) \times \Lambda$ be a saddle point of the Lagrangian functional $\mathcal{L}^\varepsilon$ associated with Problem~\ref{problemK}. Let $\bm{\zeta}^\varepsilon_\kappa \in \bm{V}_K(\omega)$ be the solution of Problem~\ref{problem2}. Then, the following estimate holds:
\begin{equation*}
\|\bm{\zeta}^\varepsilon_\kappa -\bm{\zeta}^\varepsilon\|_{\bm{V}_K(\omega)}
\le \dfrac{\sqrt{c_e} c_0 \sqrt[4]{\kappa}}{\sqrt[4]{a_0}} \|\psi^\varepsilon\|_{L^2(\omega)}^{1/2}.
\end{equation*}
\end{theorem}
\begin{proof}
Let us test Problem~\ref{problem2} at $(\bm{\zeta}^\varepsilon_\kappa-\bm{\zeta}^\varepsilon) \in \bm{V}_K(\omega)$. The monotonicity of the operator $\bm{\beta}$ gives:
\begin{equation}
\label{A}
\begin{aligned}
&\varepsilon\int_{\omega} a^{\alpha\beta\sigma\tau} \gamma_{\sigma\tau}(\bm{\zeta}^\varepsilon_\kappa) \gamma_{\alpha\beta}(\bm{\zeta}^\varepsilon_\kappa-\bm{\zeta}^\varepsilon) \sqrt{a} \dd y
+\dfrac{\varepsilon^3}{3}\int_{\omega} a^{\alpha\beta\sigma\tau} \rho_{\sigma\tau}(\bm{\zeta}^\varepsilon_\kappa) \rho_{\alpha\beta}(\bm{\zeta}^\varepsilon_\kappa-\bm{\zeta}^\varepsilon) \sqrt{a} \dd y\\
&\le\int_{\omega} p^{i,\varepsilon} (\zeta^\varepsilon_{\kappa,i}-\zeta^\varepsilon_i) \sqrt{a} \dd y.
\end{aligned}
\end{equation}

Let us now consider the saddle point $(\bm{\zeta}^\varepsilon,\psi^\varepsilon) \in \bm{V}_K(\omega) \times \Lambda$ for the Lagrangian functional $\mathcal{L}^\varepsilon$ associated with Problem~\ref{problemK} (viz. Theorem~\ref{ex:saddle}). Thanks to~\eqref{id-saddle}, we have that the specialisation $\bm{\eta}=\bm{\zeta}^\varepsilon_\kappa$ in the first set of variational inequalities in~\eqref{saddle-2} gives:
\begin{equation}
\label{B}
\begin{aligned}
&-\varepsilon\int_{\omega} a^{\alpha\beta\sigma\tau} \gamma_{\sigma\tau}(\bm{\zeta}^\varepsilon) \gamma_{\alpha\beta}(\bm{\zeta}^\varepsilon_\kappa-\bm{\zeta}^\varepsilon) \sqrt{a} \dd y
-\dfrac{\varepsilon^3}{3}\int_{\omega} a^{\alpha\beta\sigma\tau} \rho_{\sigma\tau}(\bm{\zeta}^\varepsilon) \rho_{\alpha\beta}(\bm{\zeta}^\varepsilon_\kappa-\bm{\zeta}^\varepsilon) \sqrt{a} \dd y\\
&\le-\int_{\omega} p^{i,\varepsilon} (\zeta^\varepsilon_{\kappa,i}-\zeta^\varepsilon_i) \sqrt{a} \dd y -\dfrac{\varepsilon^3}{3} \int_\omega \psi^\varepsilon [(\bm{\theta}+\zeta^\varepsilon_{\kappa,i}\bm{a}^i)\cdot\bm{q}] \dd y.
\end{aligned}
\end{equation}

Adding~\eqref{B} to~\eqref{A}, and applying Theorem~\ref{kornsurface}, and~\eqref{bdd-3} gives:
\begin{equation*}
\begin{aligned}
&\dfrac{\varepsilon^3 \sqrt{a_0}}{3c_e c_0^2}\|\bm{\zeta}^\varepsilon_\kappa -\bm{\zeta}^\varepsilon\|_{\bm{V}_K(\omega)}^2\le \varepsilon\int_{\omega} a^{\alpha\beta\sigma\tau} \gamma_{\sigma\tau}(\bm{\zeta}^\varepsilon_\kappa-\bm{\zeta}^\varepsilon) \gamma_{\alpha\beta}(\bm{\zeta}^\varepsilon_\kappa-\bm{\zeta}^\varepsilon) \sqrt{a} \dd y\\
&\quad+\dfrac{\varepsilon^3}{3}\int_{\omega} a^{\alpha\beta\sigma\tau} \rho_{\sigma\tau}(\bm{\zeta}^\varepsilon_\kappa-\bm{\zeta}^\varepsilon) \rho_{\alpha\beta}(\bm{\zeta}^\varepsilon_\kappa-\bm{\zeta}^\varepsilon) \sqrt{a} \dd y\\
&\le-\dfrac{\varepsilon^3}{3} \int_\omega \psi^\varepsilon [(\bm{\theta}+\zeta^\varepsilon_{\kappa,i}\bm{a}^i)\cdot\bm{q}] \dd y
\le \dfrac{\varepsilon^3}{3} \int_{\omega} \psi^\varepsilon \{(\bm{\theta}+\zeta^\varepsilon_{\kappa,i}\bm{a}^i)\cdot\bm{q}\}^{-} \dd y\\
&\le \dfrac{\varepsilon^3}{3}\|\psi^\varepsilon\|_{L^2(\omega)} \|\{(\bm{\theta}+\zeta^\varepsilon_{\kappa,i}\bm{a}^i)\cdot\bm{q}\}^{-}\|_{L^2(\omega)}
\le \dfrac{\varepsilon^3 \sqrt{\kappa}}{3}\|\psi^\varepsilon\|_{L^2(\omega)}.
\end{aligned}
\end{equation*}

In conclusion, we have the sought estimate
\begin{equation*}
\|\bm{\zeta}^\varepsilon_\kappa -\bm{\zeta}^\varepsilon\|_{\bm{V}_K(\omega)}^2
\le \dfrac{c_e c_0^2 \sqrt{\kappa}}{\sqrt{a_0}} \|\psi^\varepsilon\|_{L^2(\omega)},
\end{equation*}
and the proof is complete.
\end{proof}

Note in passing that it would be possible to establish the conclusion in Theorem~\ref{ex:saddle} by restricting the first argument of the Lagrangian functional associated with Problem~\ref{problemK} in the non-empty, closed and convex subset $\bm{U}_K(\omega)$. This restrictions in turn implies that the only saddle point for the functional $\mathcal{L}^\varepsilon:\bm{U}_K(\omega) \times \Lambda \to \mathbb{R}$ is $(\bm{\zeta}^\varepsilon,0)$, where $\bm{\zeta}^\varepsilon \in \bm{U}_K(\omega)$ is the unique solution of Problem~\ref{problemK}. A Lagrangian defined over the set $\bm{U}_K(\omega)$, however, would restrict the validity of the first set of variational inequalities in~\eqref{saddle-2} to the sole test vector fields in $\bm{U}_K(\omega)$, and this would prevent us from specialising $\bm{\eta}=\bm{\zeta}^\varepsilon_\kappa$ for obtaining the crucial estimate~\eqref{B}. That is why we must define the Lagrangian functional $\mathcal{L}^\varepsilon$ associated with Problem~\ref{problemK} over $\bm{V}_K(\omega) \times \Lambda$. In general, the point $(\bm{\zeta}^\varepsilon,0)$ is not a saddle point for the Lagrangian functional $\mathcal{L}^\varepsilon$ associated with Problem~\ref{problemK} when it is defined over $\bm{V}_K(\omega) \times \Lambda$.

Besides, note that the right-hand side in the main estimate in Theorem~\ref{difference-1} is not, in general, bounded independently of $\varepsilon$.
We also note in passing that establishing the existence of saddle points for $\mathcal{L}^\varepsilon$ in one such framework is a difficult problem (cf., e.g., Remark~9.2 in Chapter~3, Section~9, page~223 of~\cite{Glowinski1981}; see also~\cite{GMV84}).

In section~\ref{approx:penalty} we shall recover an estimate for the difference between the solution $\bm{\zeta}^\varepsilon_\kappa$ of Problem~\ref{problem2} and the solution $\bm{\zeta}^{\varepsilon,h}_\kappa$ of the corresponding discretised problem, which will be denoted by~$\mathcal{P}^{\varepsilon,h}_{K,\kappa}(\omega)$.

By so doing, we will be able to provide a sound approximation of the solution $\bm{\zeta}^\varepsilon$ of Problem~\ref{problemK} in terms of the solution $\bm{\zeta}^{\varepsilon,h}_\kappa$ of the Finite Element discretisation of Problem~\ref{problem2}. However, before arriving at the desired conclusion, we need additional preliminary results, that we present in the next section.

\section{The doubly-penalised mixed formulation of Problem~\ref{problemK}}
\label{mixed:penalty}

The direct approximation of the solution of Problem~\ref{problemK} by means of a conforming Finite Element Method is not viable, since the transverse component of the solution is of class $H^2_0(\omega)$. The latter regularity, would require us to show that the transverse component of the solution of Problem~\ref{problemK} is \emph{at least} of class $H^4(\omega)$, which is clearly in contradiction with the results proved by Caffarelli and his collaborators~\cite{Caffarelli1979,CafFriTor1982}, whom showed that the highest achievable regularity for the solution of a variational inequality governed by a fourth order operator is $H^3(\omega)$.

The direct approximation of the solution of Problem~\ref{problemK} via a non-conforming Finite Element Method based on the Enriching Operators technique, proposed by Brenner and her collaborators~\cite{Brenner2013}, in general does not work when the minimiser of the constrained energy functional under consideration is a vector field. A special exception was proved to work in the case where the solution is a Kirchhoff-Love field, as it happens in the case where the energy functional is associated with the displacement of a linearly elastic shallow shell subjected to an obstacle~\cite{PS}. Besides, the latter technique cannot be easily converted into code.
For the sake of completeness, we also cite the recent paper~\cite{KNM24}, where the convergence analysis for a Finite Element Method for an obstacle problem for Naghdi's shells whose constraints bears only on the transverse component of the displacement field.

Since the solution of Problem~\ref{problemK} is not a Kirchhoff-Love field, we need to attack the problem amounting to approximate the solution of Problem~\ref{problemK} from a different angle. The idea we propose is to exploit Theorem~\ref{difference-1}, and approximate the solution of Problem~\ref{problem2} instead.
The approximation of the solution of Problem~\ref{problem2} via a conforming or non-conforming Finite Element Method, although feasible from the theoretical point of view, is not easily implementable on a computer as most of the high-level finite element packages do not come with libraries for Finite Elements for fourth order problems like, for instance the Hsieh-Clough-Tocher triangle, the Argyris triangle, or the Morley triangle~\cite{PGCFEM}.

We thus replace Problem~\ref{problem2} by the equivalent intrinsic formulation proposed by Blouza \& Le Dret~\cite{BlouzaLeDret1999}. The intrinsic formulation proposed by Blouza \& Le Dret, that was proved to work for middle surface with little regularity, is based on admissible displacements of the form $\tilde{\bm{\eta}}=\eta_i \bm{a}^i$ rather than the vector fields $\bm{\eta}=(\eta_i)$ utilised to formulate the variational problems considered so far.

It can be shown (viz. Lemma~4 of~\cite{BlouzaLeDret1999}) that the space $\bm{V}_K(\omega)$ defined in section~\ref{sec2} is a subspace of the space
$$
\tilde{\bm{V}}(\omega):=\{\tilde{\bm{\eta}}\in\bm{H}^1_0(\omega); \partial_{\alpha\beta} \tilde{\bm{\eta}} \cdot\bm{a}_3 \in L^2(\omega)\}.
$$

More precisely, it can be shown (Lemma~12 of~\cite{BlouzaLeDret1999}) that the space $\bm{V}_K(\omega)$ is isomorphic to the space $\tilde{\bm{V}}_0(\omega)$ defined by:
$$
\tilde{\bm{V}}_0(\omega):=\{\tilde{\bm{\eta}} \in \tilde{\bm{V}}(\omega); \partial_\alpha \tilde{\bm{\eta}} \cdot \bm{a}_3 =0 \textup{ on }\gamma\}.
$$

After equipping the space $\tilde{\bm{V}}(\omega)$ with the norm
$$
\|\tilde{\bm{\eta}}\|_{\tilde{\bm{V}}(\omega)}:=\left\{\|\tilde{\bm{\eta}}\|_{\bm{H}^1_0(\omega)}^2+\sum_{\alpha, \beta}\|\partial_{\alpha\beta}\tilde{\bm{\eta}} \cdot \bm{a}_3\|_{L^2(\omega)}\right\}^{1/2},\quad\textup{ for all }\tilde{\bm{\eta}} \in \tilde{\bm{V}}(\omega),
$$
it can be shown that the norm $\|\cdot\|_{\tilde{\bm{V}}(\omega)}$ is equivalent to the norm
$$
\vertiii{\tilde{\bm{\eta}}}:=\left\{\sum_{\alpha, \beta}\|\gamma_{\alpha\beta}(\tilde{\bm{\eta}})\|_{L^2(\omega)}^2+ \sum_{\alpha, \beta}\|\rho_{\alpha\beta}(\tilde{\bm{\eta}})\|_{L^2(\omega)}^2\right\}^{1/2},
$$
along all the elements $\tilde{\bm{\eta}} \in \tilde{\bm{V}}(\omega)$ (viz. Lemma~11 of~\cite{BlouzaLeDret1999}).

The covariant components of the change of metric tensor can be expressed in terms of the displacement $\tilde{\bm{\eta}}=\eta_i \bm{a}^i$, and they take the following form:
\begin{equation*}
\tilde{\gamma}_{\alpha\beta}(\tilde{\bm{\eta}})=\dfrac{1}{2}\left(\partial_\alpha\tilde{\bm{\eta}} \cdot\bm{a}^\beta+\partial_\beta\tilde{\bm{\eta}}\cdot\bm{a}^\alpha\right).
\end{equation*}

The covariant components of the change of curvature tensor can be expressed in terms of the displacement $\tilde{\bm{\eta}}=\eta_i \bm{a}^i$, and they take the following form:
\begin{equation*}
\begin{aligned}
\tilde{\rho}_{\alpha\beta}(\tilde{\bm{\eta}})&=(\partial_{\alpha\beta}\tilde{\bm{\eta}} -\Gamma_{\alpha\beta}^\sigma \partial_\sigma\tilde{\bm{\eta}})\cdot\bm{a}^3\\
&=\dfrac{1}{2}\left(\partial_\alpha\tilde{\bm{\eta}} \cdot\partial_\beta\bm{a}^3+\partial_\beta\tilde{\bm{\eta}} \cdot\partial_\alpha\bm{a}^3+\partial_\alpha\tilde{\bm{\xi}} \cdot\bm{a}_\beta+\partial_\beta\tilde{\bm{\xi}} \cdot\bm{a}_\alpha\right),
\end{aligned}
\end{equation*}
where the last equality holds when $\tilde{\bm{\xi}}=-(\partial_\alpha\tilde{\bm{\eta}}\cdot\bm{a}_3)\bm{a}^\alpha$. We thus define the space:
\begin{equation*}
\bm{W}(\omega):=\{(\bm{\tilde{\bm{\eta}}},\tilde{\bm{\xi}}) \in \bm{H}^1_0(\omega) \times \bm{H}^1_0(\omega); \tilde{\bm{\xi}}+(\partial_\alpha\tilde{\bm{\eta}}\cdot\bm{a}_3)\bm{a}^\alpha=\bm{0}\},
\end{equation*}
and note that $\bm{W}(\omega)$ coincides with the space $\tilde{\bm{V}}_0(\omega)$ (cf., e.g., page~140 of~\cite{Blouza2016}).

Since the immersion $\bm{\theta}$ is, by assumption, of class $\mathcal{C}^3(\overline{\omega};\mathbb{E}^3)$ then we have that by Lemma~2 of~\cite{BlouzaLeDret1999} (see also~\eqref{cmt} and~\eqref{cct}):
\begin{equation}
\label{coincidence}
\begin{aligned}
\tilde{\gamma}_{\alpha\beta}(\tilde{\bm{\eta}})&=\gamma_{\alpha\beta}(\bm{\eta}), \quad\textup{ for all }\bm{\eta}=(\eta_i) \in H^1(\omega)\times H^1(\omega)\times L^2(\omega),\\
\tilde{\rho}_{\alpha\beta}(\tilde{\bm{\eta}})&=\rho_{\alpha\beta}(\bm{\eta}), \quad\textup{ for all }\bm{\eta}=(\eta_i) \in H^1(\omega)\times H^1(\omega)\times H^2(\omega).
\end{aligned}
\end{equation}

We can thus propose an \emph{intrinsic} variational formulation of Problem~\ref{problemK}. First, we give an alternative definition for the covariant components of the change of curvature tensor $\tilde{\rho}_{\alpha\beta}$; we let $\tilde{\rho}_{\alpha\beta}: \bm{H}^1(\omega) \times \bm{H}^1(\omega) \to \mathbb{R}$ be defined by
\begin{equation}
\label{rho-alt}
\tilde{\rho}_{\alpha\beta}(\tilde{\bm{\eta}},\tilde{\bm{\xi}}):=
\dfrac{1}{2}\left(\partial_\alpha\tilde{\bm{\eta}} \cdot\partial_\beta\bm{a}^3+\partial_\beta\tilde{\bm{\eta}} \cdot\partial_\alpha\bm{a}^3+\partial_\alpha\tilde{\bm{\xi}} \cdot\bm{a}_\beta+\partial_\beta\tilde{\bm{\xi}} \cdot\bm{a}_\alpha\right),
\end{equation}
for all $(\tilde{\bm{\eta}},\tilde{\bm{\xi}}) \in \bm{H}^1(\omega) \times \bm{H}^1(\omega)$.
We will see that this formulation of the change of curvature tensor will be more amenable in the context of the study of the doubly penalised mixed formulation corresponding to Problem~\ref{problemK} we will introduce next. Define the space:
\begin{equation*}
\mathbb{X}(\omega):=\bm{H}^1_0(\omega) \times \bm{H}^1_0(\omega).
\end{equation*}

The intrinsic formulation of Problem~\ref{problemK} thus takes the following form.

\begin{customprob}{$\tilde{\mathcal{P}}_K^\varepsilon(\omega)$}
	\label{problemK1}
	Find 
	$(\tilde{\bm{\zeta}}^\varepsilon,\tilde{\bm{\varphi}}^\varepsilon) \in \tilde{\mathbb{U}}_K(\omega):=\{(\tilde{\bm{\eta}},\tilde{\bm{\xi}})\in \mathbb{X}(\omega);\tilde{\bm{\xi}}+(\partial_\alpha\tilde{\bm{\eta}} \cdot\bm{a}_3)\bm{a}^\alpha=\bm{0} \textup{ for a.a. }y \in \omega \textup{ and }(\bm{\theta}(y)+\tilde{\bm{\eta}}(y)) \cdot \bm{q} \ge 0 \textup{ for a.a. }y \in \omega\}$ satisfying the following variational inequalities:
	\begin{equation*}
	\varepsilon\int_\omega a^{\alpha\beta\sigma\tau} \tilde{\gamma}_{\sigma\tau}(\tilde{\bm{\zeta}}^\varepsilon) \tilde{\gamma}_{\alpha\beta} (\tilde{\bm{\eta}} - \tilde{\bm{\zeta}}^\varepsilon) \sqrt{a} \dd y
	+\dfrac{\varepsilon^3}{3} \int_\omega a^{\alpha\beta\sigma\tau} \tilde{\rho}_{\sigma\tau}(\tilde{\bm{\zeta}}^\varepsilon,\tilde{\bm{\varphi}}^\varepsilon) \tilde{\rho}_{\alpha\beta} (\tilde{\bm{\eta}} - \tilde{\bm{\zeta}}^\varepsilon,\tilde{\bm{\xi}}-\tilde{\bm{\varphi}}^\varepsilon) \sqrt{a} \dd y \ge \int_\omega \tilde{\bm{p}}^{\varepsilon} \cdot (\tilde{\bm{\eta}} - \tilde{\bm{\zeta}}^\varepsilon) \sqrt{a} \dd y,
	\end{equation*}
	for all $(\tilde{\bm{\eta}},\tilde{\bm{\xi}}) \in \tilde{\mathbb{U}}_K(\omega)$, where $\tilde{\bm{p}}^{\varepsilon}:=\varepsilon\left(\int_{-1}^{1} f^i \dd x_3\right)\bm{a}_i$.
	\bqed
\end{customprob}

This variational problem, which is simply a re-writing of Problem~\ref{problemK}, admits one and only one solution. Indeed, the set $\tilde{\mathbb{U}}_K(\omega)$ is non-empty, closed and convex, and the bilinear form associated with Problem~\ref{problemK1} is continuous and elliptic as a direct application of Lemma~11 of~\cite{BlouzaLeDret1999} and Lemma~3.3 of~\cite{Blouza2016}.

Moreover, thanks to~\eqref{coincidence}, the fist component of the solution $(\tilde{\bm{\zeta}}^\varepsilon,\tilde{\bm{\varphi}}^\varepsilon) \in \tilde{\mathbb{U}}_K(\omega)$ of Problem~\ref{problemK1} corresponds to the solution of Problem~\ref{problemK} via the isomorphism $\zeta^\varepsilon_j=\tilde{\bm{\zeta}}^\varepsilon \cdot \bm{a}_j$, for all $1 \le j \le 3$.

In the same spirit of section~\ref{sec:penalty}, we can penalise Problem~\ref{problemK1} by replacing the constraints
\begin{align*}
\tilde{\bm{\xi}}+(\partial_\alpha\tilde{\bm{\eta}} \cdot\bm{a}_3)\bm{a}^\alpha=\bm{0},\quad \textup{ for a.a. }y \in \omega,\\
(\bm{\theta}(y)+\tilde{\bm{\eta}}(y)) \cdot \bm{q} \ge 0,\quad \textup{ for a.a. }y \in \omega,
\end{align*}
by two monotone terms appearing in the variational formulation. Clearly, in the same spirit of what has been done in section~\ref{sec:penalty}, the set of variational inequalities posed over the non-empty, closed, and convex set $\tilde{\mathbb{U}}_K(\omega)$ gets replaced by a family of variational equations indexed over the penalty parameter $\kappa>0$. The doubly penalised formulation of Problem~\ref{problemK1} takes the following form.

\begin{customprob}{$\tilde{\mathcal{P}}_{K,\kappa}^\varepsilon(\omega)$}
	\label{problem2-1}
	Find 
	$(\tilde{\bm{\zeta}}^\varepsilon_\kappa,\tilde{\bm{\varphi}}^\varepsilon_\kappa) \in \bm{H}^1_0(\omega) \times \bm{H}^1_0(\omega)$ satisfying the following variational equations:
	\begin{equation*}
	\begin{aligned}
	&\varepsilon\int_\omega a^{\alpha\beta\sigma\tau} \tilde{\gamma}_{\sigma\tau}(\tilde{\bm{\zeta}}^\varepsilon_\kappa) \tilde{\gamma}_{\alpha\beta} (\tilde{\bm{\eta}}) \sqrt{a} \dd y
	+\dfrac{\varepsilon^3}{3} \int_\omega a^{\alpha\beta\sigma\tau} \tilde{\rho}_{\sigma\tau}(\tilde{\bm{\zeta}}^\varepsilon_\kappa,\tilde{\bm{\varphi}}^\varepsilon_\kappa) \tilde{\rho}_{\alpha\beta} (\tilde{\bm{\eta}},\tilde{\bm{\xi}}) \sqrt{a} \dd y\\
	&\quad+\dfrac{\varepsilon}{\kappa}\int_{\omega} \left[\tilde{\bm{\varphi}}^\varepsilon_\kappa+(\partial_\alpha\tilde{\bm{\zeta}}^\varepsilon_\kappa\cdot\bm{a}_3)\bm{a}^\alpha\right] \cdot \left[\tilde{\bm{\xi}}+(\partial_\alpha \tilde{\bm{\eta}}\cdot\bm{a}_3)\bm{a}^\alpha\right] \dd y\\
	&\quad+\dfrac{\varepsilon}{\kappa}\int_{\omega}\left(-\{(\bm{\theta}+\tilde{\bm{\zeta}}^\varepsilon_\kappa)\cdot\bm{q}\}^{-} \bm{q}\right) \cdot \tilde{\bm{\eta}} \dd y= \int_\omega \tilde{\bm{p}}^{\varepsilon} \cdot \tilde{\bm{\eta}} \sqrt{a} \dd y,
	\end{aligned}
	\end{equation*}
	for all $(\tilde{\bm{\eta}},\tilde{\bm{\xi}}) \in \bm{H}^1_0(\omega) \times \bm{H}^1_0(\omega)$, where $\tilde{\bm{p}}^{\varepsilon}:=\varepsilon\left(\int_{-1}^{1} f^i \dd x_3\right)\bm{a}_i$.
	\bqed
\end{customprob}

The existence and uniqueness of solutions for Problem~\ref{problemK1} directly descends from an application of Lemma~3.3 of~\cite{Blouza2016}, Lemma~\ref{lem:beta}, and the Minty-Browder Theorem (cf., e.g., Theorem~9.14-1 of~\cite{PGCLNFAA}).
Moreover, the following properties hold.

\begin{theorem}
\label{t:convergence}
Let $(\tilde{\bm{\zeta}}^\varepsilon,\tilde{\bm{\varphi}}^\varepsilon) \in \tilde{\mathbb{U}}_K(\omega)$ be the solution of Problem~\ref{problemK1}, and let $(\tilde{\bm{\zeta}}^\varepsilon_\kappa,\tilde{\bm{\varphi}}^\varepsilon_\kappa) \in \mathbb{X}(\omega)$ be the solution of Problem~\ref{problem2-1}. Then, there exists $C>0$ independent of $\kappa$ such that
\begin{equation}
\label{conc-1}
\|\tilde{\bm{\varphi}}^\varepsilon_\kappa +(\partial_\alpha \tilde{\bm{\zeta}}^\varepsilon_\kappa \cdot\bm{a}_3)\bm{a}^\alpha\|_{\bm{L}^2(\omega)} \le C \sqrt{\kappa},
\end{equation}
and
\begin{equation}
\label{conc-2}
\|(\tilde{\bm{\zeta}}^\varepsilon,\tilde{\bm{\varphi}}^\varepsilon)-(\tilde{\bm{\zeta}}^\varepsilon_\kappa,\tilde{\bm{\varphi}}^\varepsilon_\kappa)\|_{\mathbb{X}(\omega)}:=\left\{\|\tilde{\bm{\zeta}}^\varepsilon-\tilde{\bm{\zeta}}^\varepsilon_\kappa\|_{\bm{H}^1_0(\omega)}^2 + \|\tilde{\bm{\varphi}}^\varepsilon-\tilde{\bm{\varphi}}^\varepsilon_\kappa\|_{\bm{H}^1_0(\omega)}^2\right\}^{1/2} \to 0,\quad\textup{ as }\kappa\to 0^+.
\end{equation}
\end{theorem}
\begin{proof}
Specialise $(\tilde{\bm{\zeta}}^\varepsilon_\kappa,\tilde{\bm{\varphi}}^\varepsilon_\kappa)$ in the equations of Problem~\ref{problem2-1}. We have that:
	\begin{equation}
	\label{wc--2}
	\begin{aligned}
	&\varepsilon\int_{\omega}a^{\alpha\beta\sigma\tau}  \tilde{\gamma}_{\sigma\tau}(\tilde{\bm{\zeta}}^\varepsilon_\kappa) \tilde{\gamma}_{\alpha\beta}(\tilde{\bm{\zeta}}^\varepsilon_\kappa) \sqrt{a} \dd y
	+\dfrac{\varepsilon^3}{3}\int_{\omega} a^{\alpha\beta\sigma\tau} \tilde{\rho}_{\sigma\tau}(\tilde{\bm{\zeta}}^\varepsilon_\kappa,\tilde{\bm{\varphi}}^\varepsilon_\kappa)
	\tilde{\rho}_{\alpha\beta}(\tilde{\bm{\zeta}}^\varepsilon_\kappa,\tilde{\bm{\varphi}}^\varepsilon_\kappa)\sqrt{a} \dd y\\
	&\quad+\dfrac{\varepsilon}{\kappa} \int_\omega \left(\tilde{\bm{\varphi}}^\varepsilon_\kappa+(\partial_\alpha\bm{\zeta}^\varepsilon_\kappa \cdot\bm{a}_3)\bm{a}^\alpha\right) \left(\tilde{\bm{\varphi}}^\varepsilon_\kappa+(\partial_\alpha\tilde{\bm{\zeta}}^\varepsilon_\kappa\cdot\bm{a}_3)\bm{a}^\alpha\right) \dd y\\
	&\quad+\dfrac{\varepsilon}{\kappa}\underbrace{\int_{\omega} \left(-\{(\bm{\theta}+\tilde{\bm{\zeta}}^\varepsilon_\kappa)\cdot\bm{q}\}^{-} \bm{q}\right) \cdot \tilde{\bm{\zeta}}^\varepsilon_\kappa \dd y}_{\ge 0 \textup{ by Lemma~\ref{lem:beta}}}
	=\int_{\omega}\tilde{\bm{p}}^\varepsilon \cdot \tilde{\bm{\zeta}}^\varepsilon_\kappa \sqrt{a} \dd y.
	\end{aligned}
	\end{equation}

An application of Lemma~\ref{lem:beta}, and the uniform positive-definiteness of the the fourth order two-dimensional elasticity tensor $(a^{\alpha\beta\sigma\tau})$ (cf., e.g., Theorem~3.3-2 of~\cite{Ciarlet2000}) to~\eqref{wc--2} gives that:
\begin{equation*}
\varepsilon\min\left\{1, \dfrac{\sqrt{a_0}}{c_e}\right\}\left\{\sum_{\alpha\beta}\left(\|\tilde{\gamma}_{\alpha\beta}(\tilde{\bm{\zeta}}^\varepsilon_\kappa)\|_{L^2(\omega)}^2 +\dfrac{\varepsilon^2}{3} \|\tilde{\rho}_{\alpha\beta}(\tilde{\bm{\zeta}}^\varepsilon_\kappa)\|_{L^2(\omega)}^2\right) +\dfrac{1}{\kappa}\|(\tilde{\bm{\varphi}}^\varepsilon_\kappa+\partial_\alpha\tilde{\bm{\zeta}}^\varepsilon_\kappa\bm{a}_3)\bm{a}^\alpha\|_{\bm{L}^2(\omega)}^2\right\} \le \|\tilde{\bm{p}}^\varepsilon\|_{\bm{L}^2(\omega)} \|\tilde{\bm{\zeta}}^\varepsilon_\kappa\|_{\bm{H}^1_0(\omega)}.
\end{equation*}

An application of Lemma~3.3 of~\cite{Blouza2016} and the fact that $\kappa^{-1}>1$ without loss of generality, gives that there exists a constant $\tilde{C}_0>0$ independent of $\varepsilon$ and $\kappa$ such that
\begin{equation}
\label{bound-0}
\tilde{C}_0^2\dfrac{\varepsilon^3}{3} \min\left\{1,\dfrac{\sqrt{a_0}}{c_e}\right\} \|(\tilde{\bm{\zeta}}^\varepsilon_\kappa,\tilde{\bm{\varphi}}^\varepsilon_\kappa)\|_{\mathbb{X}(\omega)}^2 \le \|\tilde{\bm{p}}^\varepsilon\|_{\bm{L}^2(\omega)} \|(\tilde{\bm{\zeta}}^\varepsilon_\kappa,\tilde{\bm{\varphi}}^\varepsilon_\kappa)\|_{\mathbb{X}(\omega)},
\end{equation}
which in turn implies that the sequence
$\{(\tilde{\bm{\zeta}}^\varepsilon_\kappa,\tilde{\bm{\varphi}}^\varepsilon_\kappa)\}_{\kappa>0}$ is bounded in $\mathbb{X}(\omega)$ independently of $\kappa$.

More precisely, the estimate~\eqref{bound-0} implies that:
\begin{equation*}
\label{bound-1}
\|(\tilde{\bm{\varphi}}^\varepsilon_\kappa+\partial_\alpha\tilde{\bm{\zeta}}^\varepsilon_\kappa\cdot\bm{a}_3)\bm{a}^\alpha\|_{\bm{L}^2(\omega)}^2
\le \dfrac{3\kappa}{\varepsilon^4\tilde{C}_0^2\min\{1,c_e^{-1}\sqrt{a_0}\}}\|\tilde{\bm{p}}^\varepsilon\|_{\bm{L}^2(\omega)}^2,
\end{equation*}
so that the estimate~\eqref{conc-1} is immediately verified. 

Moreover, 
\begin{equation}
\label{wc--1}
\left\|-\{(\bm{\theta}+\tilde{\bm{\zeta}}^\varepsilon_\kappa)\cdot\bm{q}\}^{-} \bm{q}\right\|_{\bm{L}^2(\omega)}^2 
\le \dfrac{3\kappa}{\varepsilon^4\tilde{C}_0^2\min\{1,c_e^{-1}\sqrt{a_0}\}} \|\tilde{\bm{p}}^\varepsilon\|_{\bm{L}^2(\omega)}^2.
\end{equation}

In view of~\eqref{bound-0}, we have that the following weak convergences hold up to passing to a suitable subsequence:
\begin{equation}
\label{wc-0}
\begin{aligned}
\tilde{\bm{\zeta}}^\varepsilon_\kappa &\rightharpoonup \tilde{\bm{\zeta}}^\varepsilon_1,\quad\textup{ in }\bm{H}^1_0(\omega),\\
\tilde{\bm{\varphi}}^\varepsilon_\kappa &\rightharpoonup \tilde{\bm{\varphi}}^\varepsilon_1,\quad\textup{ in }\bm{H}^1_0(\omega).
\end{aligned}
\end{equation}

An application of the Rellich-Kondra\v{s}ov theorem (cf., e.g., Theorem~6.6-3 of~\cite{PGCLNFAA}) gives that $\tilde{\bm{\zeta}}^\varepsilon_\kappa \to \tilde{\bm{\zeta}}^\varepsilon_1$ in $\bm{L}^2(\omega)$. A subsequent application of Theorem~9.13-2(a) of~\cite{PGCLNFAA} thus gives that 
\begin{equation*}
\{(\bm{\theta}+\tilde{\bm{\zeta}}^\varepsilon_1)\cdot\bm{q}\}^{-}=0,\quad\textup{ a.e. in }\omega.
\end{equation*}

Similarly, another application of the Rellich-Kondra\v{s}ov theorem (cf., e.g., Theorem~6.6-3 of~\cite{PGCLNFAA}) gives that $\tilde{\bm{\varphi}}^\varepsilon_\kappa \to \tilde{\bm{\varphi}}^\varepsilon_1$ in $\bm{L}^2(\omega)$. Combining~\eqref{conc-1} with the sequential weak lower semicontinuity of the norm $\|\cdot\|_{\bm{L}^2(\omega)}$ gives that $\tilde{\bm{\varphi}}^\varepsilon_1=-(\partial_\alpha\tilde{\bm{\zeta}}^\varepsilon_1\cdot\bm{a}_3)\bm{a}^\alpha$ a.e. in $\omega$, so that $(\tilde{\bm{\zeta}}^\varepsilon_1,\tilde{\bm{\varphi}}^\varepsilon_1) \in \mathbb{U}_K(\omega)$.

Let $(\tilde{\bm{\eta}},\tilde{\bm{\xi}}) \in \tilde{\mathbb{U}}_K(\omega)$, and specialise $(\tilde{\bm{\eta}}-\tilde{\bm{\zeta}}^\varepsilon_\kappa,\tilde{\bm{\xi}}-\tilde{\bm{\varphi}}^\varepsilon_\kappa)$ in the equations of Problem~\ref{problem2-1}. We have that:
\begin{equation}
\label{wc-1}
\begin{aligned}
&\varepsilon\int_{\omega}a^{\alpha\beta\sigma\tau}  \tilde{\gamma}_{\sigma\tau}(\tilde{\bm{\zeta}}^\varepsilon_\kappa) \tilde{\gamma}_{\alpha\beta}(\tilde{\bm{\eta}}-\tilde{\bm{\zeta}}^\varepsilon_\kappa) \sqrt{a} \dd y
+\dfrac{\varepsilon^3}{3}\int_{\omega} a^{\alpha\beta\sigma\tau} \tilde{\rho}_{\sigma\tau}(\tilde{\bm{\zeta}}^\varepsilon_\kappa,\tilde{\bm{\varphi}}^\varepsilon_\kappa)
\tilde{\rho}_{\alpha\beta}(\tilde{\bm{\eta}}-\tilde{\bm{\zeta}}^\varepsilon_\kappa,\tilde{\bm{\xi}}-\tilde{\bm{\varphi}}^\varepsilon_\kappa)\sqrt{a} \dd y\\
&\quad+\dfrac{\varepsilon}{\kappa} \int_\omega \left(\tilde{\bm{\varphi}}^\varepsilon_\kappa+(\partial_\alpha\bm{\zeta}^\varepsilon_\kappa \cdot\bm{a}_3)\bm{a}^\alpha\right) \left((\tilde{\bm{\xi}}-\tilde{\bm{\varphi}}^\varepsilon_\kappa)+(\partial_\alpha(\tilde{\bm{\eta}}-\tilde{\bm{\zeta}}^\varepsilon_\kappa)\cdot\bm{a}_3)\bm{a}^\alpha\right) \dd y\\
&\quad+\dfrac{\varepsilon}{\kappa}\underbrace{\int_{\omega} \left(-\{(\bm{\theta}+\tilde{\bm{\zeta}}^\varepsilon_\kappa)\cdot\bm{q}\}^{-} \bm{q}\right) \cdot (\tilde{\bm{\eta}} -\tilde{\bm{\zeta}}^\varepsilon_\kappa)\dd y}_{\le 0 \textup{ by Lemma~\ref{lem:beta}}}
=\int_{\omega}\tilde{\bm{p}}^\varepsilon \cdot (\tilde{\bm{\eta}}-\tilde{\bm{\zeta}}^\varepsilon_\kappa) \sqrt{a} \dd y.
\end{aligned}
\end{equation}

Observe that the definition~\eqref{rho-alt} gives at once:
\begin{equation*}
	\label{linearity}
	\begin{aligned}
		&\tilde{\rho}_{\alpha\beta}((\tilde{\bm{\eta}}_1,\tilde{\bm{\xi}}_1) \pm (\tilde{\bm{\eta}}_2,\tilde{\bm{\xi}}_2)) = \tilde{\rho}_{\alpha\beta}(\tilde{\bm{\eta}}_1\pm \tilde{\bm{\eta}}_2,\tilde{\bm{\xi}}_1\pm \tilde{\bm{\xi}}_2) =
		\tilde{\rho}_{\alpha\beta}(\tilde{\bm{\eta}}_1,\tilde{\bm{\xi}}_1) \pm \tilde{\rho}_{\alpha\beta}(\tilde{\bm{\eta}}_2,\tilde{\bm{\xi}}_2),\\
		&\tilde{\rho}_{\alpha\beta}(c(\tilde{\bm{\eta}},\tilde{\bm{\xi}}))=c\tilde{\rho}_{\alpha\beta}(\tilde{\bm{\eta}},\tilde{\bm{\xi}}),
	\end{aligned}
\end{equation*}
for all $(\tilde{\bm{\eta}},\tilde{\bm{\xi}}) \in \bm{H}^1(\omega) \times \bm{H}^1(\omega)$, and for all $c \in \mathbb{R}$.
Applying the fact that $(\tilde{\bm{\eta}},\tilde{\bm{\xi}}) \in \tilde{\mathbb{U}}_K(\omega)$ to~\eqref{wc-1} gives:
\begin{equation}
\label{wc-2}
\begin{aligned}
&\varepsilon\int_{\omega}a^{\alpha\beta\sigma\tau}  \tilde{\gamma}_{\sigma\tau}(\tilde{\bm{\zeta}}^\varepsilon_\kappa) \tilde{\gamma}_{\alpha\beta}(\tilde{\bm{\eta}}-\tilde{\bm{\zeta}}^\varepsilon_\kappa) \sqrt{a} \dd y
+\dfrac{\varepsilon^3}{3}\int_{\omega} a^{\alpha\beta\sigma\tau} \tilde{\rho}_{\sigma\tau}(\tilde{\bm{\zeta}}^\varepsilon_\kappa,\tilde{\bm{\varphi}}^\varepsilon_\kappa)
\tilde{\rho}_{\alpha\beta}(\tilde{\bm{\eta}},\tilde{\bm{\xi}})\sqrt{a} \dd y\\
&\quad-\dfrac{\varepsilon^3}{3}\int_{\omega} a^{\alpha\beta\sigma\tau} \tilde{\rho}_{\sigma\tau}(\tilde{\bm{\zeta}}^\varepsilon_\kappa,\tilde{\bm{\varphi}}^\varepsilon_\kappa)
\tilde{\rho}_{\alpha\beta}(\tilde{\bm{\zeta}}^\varepsilon_\kappa,\tilde{\bm{\varphi}}^\varepsilon_\kappa)\sqrt{a} \dd y\ge \int_{\omega}\tilde{\bm{p}}^\varepsilon \cdot (\tilde{\bm{\eta}}-\tilde{\bm{\zeta}}^\varepsilon_\kappa) \sqrt{a} \dd y.
\end{aligned}
\end{equation}

The continuity of linear form on the right-hand side of~\eqref{wc-2}, the continuity of the bilinear forms on the left-hand side of~\eqref{wc-2}, and~\eqref{wc-0} allow us to pass to the $\limsup_{\kappa\to 0^+}$ in~\eqref{wc-2}, getting
\begin{equation*}
\label{wc-3}
\begin{aligned}
&\varepsilon\int_{\omega}a^{\alpha\beta\sigma\tau}  \tilde{\gamma}_{\sigma\tau}(\tilde{\bm{\zeta}}^\varepsilon_1) \tilde{\gamma}_{\alpha\beta}(\tilde{\bm{\eta}}-\tilde{\bm{\zeta}}^\varepsilon_1) \sqrt{a} \dd y
+\dfrac{\varepsilon^3}{3}\int_{\omega} a^{\alpha\beta\sigma\tau} \tilde{\rho}_{\sigma\tau}(\tilde{\bm{\zeta}}^\varepsilon_1,\tilde{\bm{\varphi}}^\varepsilon_1)
\tilde{\rho}_{\alpha\beta}(\tilde{\bm{\eta}}-\tilde{\bm{\zeta}}^\varepsilon_1,\tilde{\bm{\xi}}-\tilde{\bm{\varphi}}^\varepsilon_1)\sqrt{a} \dd y\\
&\ge \int_{\omega}\tilde{\bm{p}}^\varepsilon \cdot (\tilde{\bm{\eta}}-\tilde{\bm{\zeta}}^\varepsilon_1) \sqrt{a} \dd y,
\end{aligned}
\end{equation*}
for all $(\tilde{\bm{\eta}},\tilde{\bm{\xi}}) \in \tilde{\mathbb{U}}_K(\omega)$.
Hence, we have that $(\tilde{\bm{\zeta}}^\varepsilon_1,\tilde{\bm{\varphi}}^\varepsilon_1) \in \tilde{\mathbb{U}}_K(\omega)$ and thus coincides with the unique solution of Problem~\ref{problemK1}.

Thanks to Lemma~3.3 in~\cite{Blouza2016}, Lemma~\ref{lem:beta}, the uniform positive-definiteness of the fourth order two-dimensional elasticity tensor $(a^{\alpha\beta\sigma\tau})$ (cf., e.g., Theorem~3.3-2 of~\cite{Ciarlet2000}), the fact that $(\tilde{\bm{\zeta}}^\varepsilon,\tilde{\bm{\varphi}}^\varepsilon) \in \tilde{\mathbb{U}}_K(\omega)$ and~\eqref{wc-0}, we have that:
\begin{equation*}
\begin{aligned}
&\tilde{C}_0^2 \dfrac{\varepsilon^3}{3 c_e} \|(\tilde{\bm{\zeta}}^\varepsilon_\kappa,\tilde{\bm{\varphi}}^\varepsilon_\kappa)-(\tilde{\bm{\zeta}}^\varepsilon,\tilde{\bm{\varphi}}^\varepsilon)\|_{\mathbb{X}(\omega)}^2
\le \varepsilon\int_{\omega}a^{\alpha\beta\sigma\tau}  \tilde{\gamma}_{\sigma\tau}(\tilde{\bm{\zeta}}^\varepsilon_\kappa-\tilde{\bm{\zeta}}^\varepsilon) \tilde{\gamma}_{\alpha\beta}(\tilde{\bm{\zeta}}^\varepsilon_\kappa-\tilde{\bm{\zeta}}^\varepsilon) \sqrt{a} \dd y\\
&\quad +\dfrac{\varepsilon^3}{3}\int_{\omega} a^{\alpha\beta\sigma\tau} \tilde{\rho}_{\sigma\tau}(\tilde{\bm{\zeta}}^\varepsilon_\kappa-\tilde{\bm{\zeta}}^\varepsilon,\tilde{\bm{\varphi}}^\varepsilon_\kappa-\tilde{\bm{\varphi}}^\varepsilon)
\tilde{\rho}_{\alpha\beta}(\tilde{\bm{\zeta}}^\varepsilon_\kappa-\tilde{\bm{\zeta}}^\varepsilon,\tilde{\bm{\varphi}}^\varepsilon_\kappa-\tilde{\bm{\varphi}}^\varepsilon)\sqrt{a} \dd y\\
&\quad+\dfrac{\varepsilon}{\kappa} \int_\omega \left|(\tilde{\bm{\varphi}}^\varepsilon_\kappa-\tilde{\bm{\varphi}}^\varepsilon)+\left(\partial_\alpha(\tilde{\bm{\zeta}}^\varepsilon_\kappa-\tilde{\bm{\zeta}}^\varepsilon)\cdot\bm{a}^3\right)\bm{a}^\alpha\right|^2 \dd y + \dfrac{\varepsilon}{\kappa}\int_{\omega} \left(-\{(\bm{\theta}+\tilde{\bm{\zeta}}^\varepsilon_\kappa)\cdot\bm{q}\}^{-} \bm{q}\right) \cdot (\tilde{\bm{\zeta}}^\varepsilon_\kappa -\tilde{\bm{\zeta}}^\varepsilon)\dd y\\
&=\varepsilon\int_{\omega}a^{\alpha\beta\sigma\tau}  \tilde{\gamma}_{\sigma\tau}(\tilde{\bm{\zeta}}^\varepsilon_\kappa) \tilde{\gamma}_{\alpha\beta}(\tilde{\bm{\zeta}}^\varepsilon_\kappa-\tilde{\bm{\zeta}}^\varepsilon) \sqrt{a} \dd y
-\varepsilon\int_{\omega}a^{\alpha\beta\sigma\tau}  \tilde{\gamma}_{\sigma\tau}(\tilde{\bm{\zeta}}^\varepsilon) \tilde{\gamma}_{\alpha\beta}(\tilde{\bm{\zeta}}^\varepsilon_\kappa) \sqrt{a} \dd y\\
&\quad+\varepsilon\int_{\omega}a^{\alpha\beta\sigma\tau}  \tilde{\gamma}_{\sigma\tau}(\tilde{\bm{\zeta}}^\varepsilon) \tilde{\gamma}_{\alpha\beta}(\tilde{\bm{\zeta}}^\varepsilon) \sqrt{a} \dd y
+\dfrac{\varepsilon^3}{3}\int_{\omega} a^{\alpha\beta\sigma\tau} \tilde{\rho}_{\sigma\tau}(\tilde{\bm{\zeta}}^\varepsilon_\kappa,\tilde{\bm{\varphi}}^\varepsilon_\kappa)
\tilde{\rho}_{\alpha\beta}(\tilde{\bm{\zeta}}^\varepsilon_\kappa-\tilde{\bm{\zeta}}^\varepsilon,\tilde{\bm{\varphi}}^\varepsilon_\kappa-\tilde{\bm{\varphi}}^\varepsilon)\sqrt{a} \dd y\\
&\quad-\dfrac{\varepsilon^3}{3}\int_{\omega} a^{\alpha\beta\sigma\tau} \tilde{\rho}_{\sigma\tau}(\tilde{\bm{\zeta}}^\varepsilon,\tilde{\bm{\varphi}}^\varepsilon)
\tilde{\rho}_{\alpha\beta}(\tilde{\bm{\zeta}}^\varepsilon_\kappa-\tilde{\bm{\zeta}}^\varepsilon,\tilde{\bm{\varphi}}^\varepsilon_\kappa-\tilde{\bm{\varphi}}^\varepsilon)\sqrt{a} \dd y
+\dfrac{\varepsilon}{\kappa}\int_{\omega}\left|\tilde{\bm{\varphi}}^\varepsilon_\kappa+(\partial_\alpha \tilde{\bm{\zeta}}^\varepsilon_\kappa \cdot\bm{a}_3)\bm{a}^\alpha\right|^2 \dd y\\
&\quad+ \dfrac{\varepsilon}{\kappa}\int_{\omega} \left(-\{(\bm{\theta}+\tilde{\bm{\zeta}}^\varepsilon_\kappa)\cdot\bm{q}\}^{-} \bm{q}\right) \cdot (\tilde{\bm{\zeta}}^\varepsilon_\kappa -\tilde{\bm{\zeta}}^\varepsilon)\dd y\\
&=\int_{\omega}\tilde{\bm{p}}^\varepsilon \cdot (\tilde{\bm{\zeta}}^\varepsilon_\kappa -\tilde{\bm{\zeta}}^\varepsilon) \sqrt{a} \dd y
-\varepsilon\int_{\omega}a^{\alpha\beta\sigma\tau}  \tilde{\gamma}_{\sigma\tau}(\tilde{\bm{\zeta}}^\varepsilon) \tilde{\gamma}_{\alpha\beta}(\tilde{\bm{\zeta}}^\varepsilon_\kappa) \sqrt{a} \dd y
+\varepsilon\int_{\omega}a^{\alpha\beta\sigma\tau}  \tilde{\gamma}_{\sigma\tau}(\tilde{\bm{\zeta}}^\varepsilon) \tilde{\gamma}_{\alpha\beta}(\tilde{\bm{\zeta}}^\varepsilon) \sqrt{a} \dd y\\
&\quad-\dfrac{\varepsilon^3}{3}\int_{\omega} a^{\alpha\beta\sigma\tau} \tilde{\rho}_{\sigma\tau}(\tilde{\bm{\zeta}}^\varepsilon,\tilde{\bm{\varphi}}^\varepsilon)
\tilde{\rho}_{\alpha\beta}(\tilde{\bm{\zeta}}^\varepsilon_\kappa-\tilde{\bm{\zeta}}^\varepsilon,\tilde{\bm{\varphi}}^\varepsilon_\kappa-\tilde{\bm{\varphi}}^\varepsilon)\sqrt{a} \dd y\to 0, \quad\textup{ as }\kappa\to 0^+,
\end{aligned}
\end{equation*}
thus proving the desired strong convergence~\eqref{conc-2}.
\end{proof}

Now, define
\begin{align*}
\tilde{n}^{\alpha\beta,\varepsilon}_\kappa&:=\varepsilon a^{\alpha\beta\sigma\tau}\tilde{\gamma}_{\sigma\tau}(\tilde{\bm{\zeta}}^\varepsilon_\kappa),\\
\tilde{m}^{\alpha\beta,\varepsilon}_\kappa&:=\dfrac{\varepsilon^3}{3} a^{\alpha\beta\sigma\tau} \tilde{\rho}_{\sigma\tau}(\tilde{\bm{\zeta}}^\varepsilon_\kappa,\tilde{\bm{\varphi}}^\varepsilon_\kappa),\\
\tilde{\bm{\beta}}(\tilde{\bm{\zeta}}^\varepsilon_\kappa)&:=-\{(\bm{\theta}+\tilde{\bm{\zeta}}^\varepsilon_\kappa)\cdot\bm{q}\}^{-} \bm{q},\\
\tilde{\bm{q}}(\tilde{\bm{\zeta}}^\varepsilon_\kappa,\tilde{\bm{\varphi}}^\varepsilon_\kappa)&:=\tilde{\bm{\varphi}}^\varepsilon_\kappa+(\partial_\alpha\tilde{\bm{\zeta}}^\varepsilon_\kappa\cdot\bm{a}_3)\bm{a}^\alpha.
\end{align*}

In the same spirit of~\cite{Blouza2016}, we observe that the boundary value problem associated with Problem~\ref{problem2-1} takes the following form:
\begin{equation}
\label{BVP-mixed}
\begin{cases}
-\partial_\alpha\left(\tilde{n}^{\alpha\beta,\varepsilon}_\kappa \bm{a}_\beta \sqrt{a}+\tilde{m}^{\alpha\beta,\varepsilon}_\kappa (\partial_\beta \bm{a}_3) \sqrt{a}+\dfrac{\varepsilon}{\kappa\sqrt{a}} \tilde{\bm{q}}(\tilde{\bm{\zeta}}^\varepsilon_\kappa,\tilde{\bm{\varphi}}^\varepsilon_\kappa)\cdot\bm{a}^\alpha \bm{a}_3\right)+\dfrac{\varepsilon}{\kappa \sqrt{a}}\tilde{\bm{\beta}}(\tilde{\bm{\zeta}}^\varepsilon_\kappa)=\tilde{\bm{p}}^\varepsilon \sqrt{a}, \textup{ in }\omega,\\
\\
-\partial_\alpha(\tilde{m}^{\alpha\beta,\varepsilon}_\kappa (\partial_\beta \bm{a}_3) \sqrt{a})+\dfrac{\varepsilon}{\kappa\sqrt{a}}\tilde{\bm{q}}(\tilde{\bm{\zeta}}^\varepsilon_\kappa,\tilde{\bm{\varphi}}^\varepsilon_\kappa)=\bm{0}, \textup{ in }\omega,\\
\\
\tilde{\bm{\zeta}}^\varepsilon_\kappa = \tilde{\bm{\varphi}}^\varepsilon_\kappa =\bm{0}, \textup{ on }\gamma.
\end{cases}
\end{equation}

Finally, thanks to~\eqref{BVP-mixed}, we show that the solution $(\tilde{\bm{\zeta}}^\varepsilon_\kappa,\tilde{\bm{\varphi}}^\varepsilon_\kappa)$ of Problem~\ref{problem2-1} is also of class $(\bm{H}^2(\omega) \cap \bm{H}^1_0(\omega)) \times (\bm{H}^2(\omega) \cap \bm{H}^1_0(\omega))$ up to the boundary. Since the proof follows the same pattern as the proofs of Theorem~\ref{aug:int} and Theorem~\ref{aug:bdry}, we just limit ourselves to sketch it.

\begin{theorem}
\label{regularity}
Assume that the boundary $\gamma$ of the domain $\omega$ is of class $\mathcal{C}^4$ and that $\bm{\theta} \in \mathcal{C}^4(\overline{\omega};\mathbb{E}^3)$.
Assume that there exists a unit-norm vector $\bm{q} \in \mathbb{E}^3$ such that
\begin{equation*}
\min_{y \in \overline{\omega}} (\bm{\theta} (y) \cdot \bm{q}) > 0 \quad
\textup{ and } \quad
\min_{y \in \overline{\omega}} (\bm{a}_3 (y) \cdot \bm{q}) > 0.
\end{equation*}

Assume also that the vector field $\bm{f}^\varepsilon=(f^{i,\varepsilon})$ defining the applied body force density is such that $\bm{p}^\varepsilon=(p^{i,\varepsilon}) \in H^1(\omega) \times H^1(\omega) \times L^2(\omega)$.

Then, the solution $(\tilde{\bm{\zeta}}^\varepsilon_\kappa,\tilde{\bm{\varphi}}^\varepsilon_\kappa)$ of Problem~\ref{problem2-1} is also of class $(\bm{H}^2(\omega) \cap \bm{H}^1_0(\omega)) \times (\bm{H}^2(\omega) \cap \bm{H}^1_0(\omega))$
\end{theorem}
\begin{proof}
The proof of the higher regularity in the interior follows the same pattern as the proof of Theorem~\ref{aug:int}.

For what concerns the regularity up to the boundary, note that the estimate~\eqref{wc--1} and the fact that the bilinear form is \emph{uniformly strongly elliptic} in the sense of page~185 of~\cite{Nec67} put us in the position to apply Lemma~3.2 on page~263 of~\cite{Nec67}, which in turn gives that the solution $(\tilde{\bm{\zeta}}^\varepsilon_\kappa,\tilde{\bm{\varphi}}^\varepsilon_\kappa)$ of Problem~\ref{problem2-1} is of class $(\bm{H}^2(\omega) \cap \bm{H}^1_0(\omega)) \times (\bm{H}^2(\omega) \cap \bm{H}^1_0(\omega))$.
\end{proof}

\section{Numerical approximation of the solution of Problem~\ref{problem2-1} via the Finite Element Method}
\label{approx:penalty}

In this section we present a suitable Finite Element Method to approximate the solution to Problem~\ref{problem2-1}. 
Following~\cite{Bartels2016} and~\cite{PGCFEM}, we recall some basic terminology and definitions. 
In what follows the letter $h$ denotes a quantity approaching zero. For brevity, the same notation $C$ (with or without subscripts) designates a positive constant independent of $\varepsilon$, $\kappa$ and $h$, which can take different values at different places.
We denote by $(\mathcal{T}_h)_{h>0}$ a \emph{family of triangulations of the polygonal domain} $\overline{\omega}$ made of triangles and we let $T$ denote any element of such a family.
Let us first recall the \emph{rigorous} definition of \emph{finite element} in $\mathbb{R}^n$, where $n \ge 1$ is an integer. A \emph{finite element} in $\mathbb{R}^n$ is a \emph{triple} $(T,P, \mathcal{N})$ where:

(i) $T$ is a closed subset of $\mathbb{R}^n$ with non-empty interior and Lipschitz-continuous boundary,

(ii) $P$ is a finite dimensional space of real-valued functions defined over $T$,

(iii) $\mathcal{N}$ is is a finite set of linearly independent linear forms $N_i$, $1 \le i \le \dim P$, defined over the space $P$.

By definition, it is assumed that the set $\mathcal{N}$ is \emph{$P$-unisolvent} in the following sense: given any real scalars $\alpha_i$, $1\le i \le \dim P$, there exists a unique function $g \in P$ which satisfies
$$
N_i(g)=\alpha_i, \quad 1 \le i \le \dim P.
$$

It is henceforth assumed that the \emph{degrees of freedom}, $N_i$ , lie in the dual space of a function space larger than $P$ like, for instance, a Sobolev space.
For brevity we shall conform our terminology to the one of~\cite{PGCFEM}, calling the sole set $T$ a finite element.
Define the \emph{diameter} of any finite element $T$ as follows:
$$
h_T=\text{diam }T:= \max_{x,y \in T} |x-y|.
$$

Let us also define
$$
\rho_T:=\sup\{\text{diam }B; B \textup{ is a ball contained in }T\}.
$$

A triangulation $\mathcal{T}_h$ is said to be \emph{regular} (cf., e.g., \cite{PGCFEM}) if:

(i) There exists a constant $\sigma>0$, independent of $h$, such that
$$
\textup{for all }T \in \mathcal{T}_h,\quad \dfrac{h_T}{\rho_T} \le \sigma.
$$

(ii) The quantity $h:=\max\{h_T>0; T \in \mathcal{T}_h\} $ approaches zero.

A triangulation $\mathcal{T}_h$ is said to satisfy \emph{an inverse assumption} (cf., e.g., \cite{PGCFEM}) if there exists a constant $\upsilon>0$ such that
$$
\textup{for all }T \in \mathcal{T}_h,\quad \dfrac{h}{h_T} \le \upsilon.
$$

We assume that the finite elements $(K, P_K, \Sigma_K)$, $K \in \bigcup_{h>0}\mathcal{T}_h$, are of class $\mathcal{C}^0$ and are affine (cf. Section~2.3 of~\cite{PGCFEM}), in the sense that they are affine equivalent to a single reference element $(\hat{K}, \hat{P}, \hat{\Sigma})$.

%There is of course an ambiguity in the meaning of $h$, which was first regarded as a parameter associated with the considered family of triangulations, and which next denotes a geometrical entity. Nevertheless, in this paper, we have conformed to this standard notation (see~\cite{PGCFEM}). In the rest of this section, the \emph{parameter} $h$ is assumed to be fixed and we also assume that the triangulation $\mathcal{T}_h$ under consideration is regular and satisfies the aforementioned inverse assumption.
%Let $\mathcal{V}_h$ be the set of \emph{all} of the \emph{nodal points} of $\mathcal{T}_h$, let $p$ denote any point of $\mathcal{V}_h$ and let $\mathcal{E}_h$ be the set of \emph{open edges} of $\mathcal{T}_h$, in the sense that
%$$
%\textup{any edge }e \in \mathcal{E}_h \textup{ is isomorphic to the open interval }(0,1). 
%$$

The forthcoming finite element analysis will be carried out using triangles of type $(1)$ (see Figure~2.2.1 of~\cite{PGCFEM}) to approximate the components of the solution of Problem~\ref{problem2-1}. In this case, the set $\mathcal{V}_h$ consists of all the vertices of the triangulation $\mathcal{T}_h$.

Let $\{V_{i,h}\}_{i=1}^3$ be three finite dimensional spaces such that $V_{i,h}\subset H^1_0(\omega)$.
Define
$$
\bm{V}_h:=V_{1,h} \times V_{2,h}\times V_{3,h},
$$
and observe that $\bm{V}_h \subset \bm{H}^1_0(\omega)$.

Let us now define the $\bm{V}_h$ interpolation operator $\bm{\Pi}_h:\bm{\mathcal{C}}^0(\overline{\omega})\to\bm{V}_h$ as follows
$$
\bm{\Pi}_h \bm{\xi}:=\left(\Pi_{1,h} \xi_1, \Pi_{2,h} \xi_2, \Pi_{3,h} \xi_3\right)\quad\textup{ for all }\bm{\xi}=(\xi_i)\in \bm{\mathcal{C}}^0(\overline{\omega}),
$$
where $\Pi_{i,h}$ is the standard $V_{i,h}$ interpolation operator (cf., e.g., \cite{PGCFEM}). 
It thus results that the interpolation operator $\bm{\Pi}_h$ satisfies the following properties
\begin{equation*}
(\Pi_{j,h} \xi_j)(p)=\xi_j(p)\quad\textup{ for all integers }1 \le j \le 3 \textup{ and all vertices }p \in \mathcal{V}_h.
\end{equation*}

An application of Theorem~3.2.1 of~\cite{PGCFEM} gives that there exists a constant $\hat{C}>0$ independent of $h$ such that
\begin{equation}
\label{Pih}
\|\bm{\xi}-\bm{\Pi}_h \bm{\xi}\|_{\bm{H}^1_0(\omega)} \le \hat{C} h |\bm{\xi}|_{\bm{H}^2(\omega)},
\end{equation}
for all $\bm{\xi}\in \bm{H}^2(\omega)\cap \bm{H}^1_0(\omega)$, where $|\cdot|_{\bm{H}^2(\omega)}$ denotes the semi-norm associated with the norm $\|\cdot\|_{\bm{H}^2(\omega)}$.

The discretised version of Problem~\ref{problem2} is formulated as follows.

\begin{customprob}{$\tilde{\mathcal{P}}_{K,\kappa}^{\varepsilon,h}(\omega)$}
	\label{problem3}
	Find $(\tilde{\bm{\zeta}}^{\varepsilon,h}_\kappa,\tilde{\bm{\varphi}}^{\varepsilon,h}_\kappa) \in \bm{V}_h \times \bm{V}_h$ satisfying the following variational equations:
	\begin{equation*}
	\begin{aligned}
		&\varepsilon\int_\omega a^{\alpha\beta\sigma\tau} \tilde{\gamma}_{\sigma\tau}(\tilde{\bm{\zeta}}^{\varepsilon,h}_\kappa) \tilde{\gamma}_{\alpha\beta} (\tilde{\bm{\eta}}) \sqrt{a} \dd y
		+\dfrac{\varepsilon^3}{3} \int_\omega a^{\alpha\beta\sigma\tau} \tilde{\rho}_{\sigma\tau}(\tilde{\bm{\zeta}}^{\varepsilon,h}_\kappa,\tilde{\bm{\varphi}}^{\varepsilon,h}_\kappa) \tilde{\rho}_{\alpha\beta} (\tilde{\bm{\eta}},\tilde{\bm{\xi}}) \sqrt{a} \dd y\\
		&\quad+\dfrac{\varepsilon}{\kappa}\int_{\omega} \left[\tilde{\bm{\varphi}}^{\varepsilon,h}_\kappa+(\partial_\alpha\tilde{\bm{\zeta}}^{\varepsilon,h}_\kappa\cdot\bm{a}_3)\bm{a}^\alpha\right] \cdot \left[\tilde{\bm{\xi}}+(\partial_\alpha \tilde{\bm{\eta}}\cdot\bm{a}_3)\bm{a}^\alpha\right] \dd y\\	&\quad+\dfrac{\varepsilon}{\kappa}\int_{\omega}\left(-\{(\bm{\theta}+\tilde{\bm{\zeta}}^{\varepsilon,h}_\kappa)\cdot\bm{q}\}^{-} \bm{q}\right) \cdot \tilde{\bm{\eta}} \dd y= \int_\omega \tilde{\bm{p}}^{\varepsilon} \cdot \tilde{\bm{\eta}} \sqrt{a} \dd y,
	\end{aligned}
	\end{equation*}
	for all $(\tilde{\bm{\eta}},\tilde{\bm{\xi}}) \in \bm{V}_h \times \bm{V}_h$.
	\bqed
\end{customprob}

It can be easily shown, thanks to the Minty-Browder theorem (cf.,e.g., Theorem~9.14-1 of~\cite{PGCLNFAA}), that Problem~\ref{problem3} admits a unique solution $(\tilde{\bm{\zeta}}^{\varepsilon,h}_\kappa,\tilde{\bm{\varphi}}^{\varepsilon,h}_\kappa) \in \bm{V}_h \times \bm{V}_h$. We also refer the reader to~\cite{BK23} for a discussion on convex minimisation problems.

We are thus able to derive the main result of this paper, which is the convergence of the solution $(\tilde{\bm{\zeta}}^{\varepsilon,h}_\kappa,\tilde{\bm{\varphi}}^{\varepsilon,h}_\kappa)$ of the discrete variational problem~\ref{problem3} to the solution $(\tilde{\bm{\zeta}}^\varepsilon_\kappa,\tilde{\bm{\varphi}}^\varepsilon_\kappa)$ of Problem~\ref{problem2-1}.

\begin{theorem}
\label{th:conv}
Let $(\tilde{\bm{\zeta}}^\varepsilon_\kappa,\tilde{\bm{\varphi}}^\varepsilon_\kappa) \in \mathbb{X}(\omega)$ be the solution of Problem~\ref{problem2-1}, and let $(\tilde{\bm{\zeta}}^{\varepsilon,h}_\kappa,\tilde{\bm{\varphi}}^{\varepsilon,h}_\kappa)$ be the solution of Problem~\ref{problem3}. Then, we have that
\begin{equation*}
\|(\tilde{\bm{\zeta}}^\varepsilon_\kappa,\tilde{\bm{\varphi}}^\varepsilon_\kappa)-(\tilde{\bm{\zeta}}^{\varepsilon,h}_\kappa,\tilde{\bm{\varphi}}^{\varepsilon,h}_\kappa)\|_{\mathbb{X}(\omega)} \to 0,\quad\textup{ as }h \to 0^+.
\end{equation*}
\end{theorem}
\begin{proof}
Define $m:=\min\{1,\frac{\sqrt{a_0}}{c_e}\}$, define $M:=\max_{\alpha,\beta,\sigma,\tau \in \{1,2\}}\|a^{\alpha\beta\sigma\tau}\|_{\mathcal{C}^0(\overline{\omega})}$, and let $\tilde{C}_0$ be the positive constant independent of $\varepsilon$, $\kappa$ and $h$ introduced in~\eqref{bound-0}. Observe that the following chain of estimates holds
\begin{equation*}
\begin{aligned}
&\dfrac{m \tilde{C}_0^2 \varepsilon^3}{3} \left\{\|\tilde{\bm{\zeta}}^\varepsilon_\kappa-\tilde{\bm{\zeta}}^{\varepsilon,h}_\kappa\|_{\bm{H}^1_0(\omega)}^2
+\|\tilde{\bm{\varphi}}^\varepsilon_\kappa-\tilde{\bm{\varphi}}^{\varepsilon,h}_\kappa\|_{\bm{H}^1_0(\omega)}^2\right\}
+\dfrac{\varepsilon}{\kappa}\left\|\left(-\{(\bm{\theta}+\tilde{\bm{\zeta}}^\varepsilon_\kappa)\cdot\bm{q}\}^{-} \bm{q}\right)-\left(-\{(\bm{\theta}+\tilde{\bm{\zeta}}^{\varepsilon,h}_\kappa)\cdot\bm{q}\}^{-} \bm{q}\right)\right\|_{\bm{L}^2(\omega)}^2\\
&\le\varepsilon\int_\omega a^{\alpha\beta\sigma\tau} \tilde{\gamma}_{\sigma\tau}(\tilde{\bm{\zeta}}^\varepsilon_\kappa-\tilde{\bm{\zeta}}^{\varepsilon,h}_\kappa) \tilde{\gamma}_{\alpha\beta}(\tilde{\bm{\zeta}}^\varepsilon_\kappa-\tilde{\bm{\zeta}}^{\varepsilon,h}_\kappa)\sqrt{a} \dd y\\
&\quad+\dfrac{\varepsilon^3}{3}\int_\omega a^{\alpha\beta\sigma\tau} \tilde{\rho}_{\sigma\tau}(\tilde{\bm{\zeta}}^\varepsilon_\kappa-\tilde{\bm{\zeta}}^{\varepsilon,h}_\kappa,\tilde{\bm{\varphi}}^\varepsilon_\kappa-\tilde{\bm{\varphi}}^{\varepsilon,h}_\kappa) \tilde{\rho}_{\alpha\beta}(\tilde{\bm{\zeta}}^\varepsilon_\kappa-\tilde{\bm{\zeta}}^{\varepsilon,h}_\kappa,\tilde{\bm{\varphi}}^\varepsilon_\kappa-\tilde{\bm{\varphi}}^{\varepsilon,h}_\kappa)\sqrt{a} \dd y\\
&\quad+\dfrac{\varepsilon}{\kappa}\int_{\omega} \left|(\tilde{\bm{\varphi}}^\varepsilon_\kappa-\tilde{\bm{\varphi}}^{\varepsilon,h}_\kappa)+(\partial_\beta(\tilde{\bm{\zeta}}^\varepsilon_\kappa-\tilde{\bm{\zeta}}^{\varepsilon,h}_\kappa)\cdot\bm{a}^3)\bm{a}_\beta\right|^2 \dd y\\
&\quad+\dfrac{\varepsilon}{\kappa}\int_{\omega}\left[\left(-\{(\bm{\theta}+\tilde{\bm{\zeta}}^\varepsilon_\kappa)\cdot\bm{q}\}^{-} \bm{q}\right)-\left(-\{(\bm{\theta}+\tilde{\bm{\zeta}}^{\varepsilon,h}_\kappa)\cdot\bm{q}\}^{-} \bm{q}\right)\right] \cdot(\tilde{\bm{\zeta}}^\varepsilon_\kappa-\tilde{\bm{\zeta}}^{\varepsilon,h}_\kappa) \dd y\\
&=\varepsilon\int_\omega a^{\alpha\beta\sigma\tau} \tilde{\gamma}_{\sigma\tau}(\tilde{\bm{\zeta}}^\varepsilon_\kappa-\tilde{\bm{\zeta}}^{\varepsilon,h}_\kappa) \tilde{\gamma}_{\alpha\beta}(\tilde{\bm{\zeta}}^\varepsilon_\kappa-\Pi_h\tilde{\bm{\zeta}}^\varepsilon_\kappa)\sqrt{a} \dd y\\
&\quad+\dfrac{\varepsilon^3}{3}\int_\omega a^{\alpha\beta\sigma\tau} \tilde{\rho}_{\sigma\tau}(\tilde{\bm{\zeta}}^\varepsilon_\kappa-\tilde{\bm{\zeta}}^{\varepsilon,h}_\kappa,\tilde{\bm{\varphi}}^\varepsilon_\kappa-\tilde{\bm{\varphi}}^{\varepsilon,h}_\kappa) \tilde{\rho}_{\alpha\beta}(\tilde{\bm{\zeta}}^\varepsilon_\kappa-\Pi_h \tilde{\bm{\zeta}}^\varepsilon_\kappa,\tilde{\bm{\varphi}}^\varepsilon_\kappa-\Pi_h \tilde{\bm{\varphi}}^\varepsilon_\kappa)\sqrt{a} \dd y\\
&\quad+\dfrac{\varepsilon}{\kappa}\int_{\omega} \left[(\tilde{\bm{\varphi}}^\varepsilon_\kappa-\tilde{\bm{\varphi}}^{\varepsilon,h}_\kappa)+(\partial_\beta(\tilde{\bm{\zeta}}^\varepsilon_\kappa-\tilde{\bm{\zeta}}^{\varepsilon,h}_\kappa)\cdot\bm{a}^3)\bm{a}_\beta\right]
\left[(\tilde{\bm{\varphi}}^\varepsilon_\kappa-\Pi_h \tilde{\bm{\varphi}}^\varepsilon_\kappa)+(\partial_\beta(\tilde{\bm{\zeta}}^\varepsilon_\kappa-\Pi_h \tilde{\bm{\zeta}}^\varepsilon_\kappa)\cdot\bm{a}^3)\bm{a}_\beta\right] \dd y\\
&\quad+\dfrac{\varepsilon}{\kappa}\int_{\omega}\left[\left(-\{(\bm{\theta}+\tilde{\bm{\zeta}}^\varepsilon_\kappa)\cdot\bm{q}\}^{-} \bm{q}\right)-\left(-\{(\bm{\theta}+\tilde{\bm{\zeta}}^{\varepsilon,h}_\kappa)\cdot\bm{q}\}^{-} \bm{q}\right)\right] \cdot(\tilde{\bm{\zeta}}^\varepsilon_\kappa-\Pi_h\tilde{\bm{\zeta}}^\varepsilon_\kappa) \dd y\\
&\le \varepsilon M \sqrt{a_1} \|\tilde{\bm{\zeta}}^\varepsilon_\kappa-\tilde{\bm{\zeta}}^{\varepsilon,h}_\kappa\|_{\bm{H}^1_0(\omega)} \|\tilde{\bm{\zeta}}^\varepsilon_\kappa-\Pi_h \tilde{\bm{\zeta}}^\varepsilon_\kappa\|_{\bm{H}^1_0(\omega)}\\
&\quad+\dfrac{\varepsilon^3}{3} M \sqrt{a_1} \left(\max_{1\le i \le 3}\|\bm{a}_i\|_{\bm{\mathcal{C}}^1(\overline{\omega})}^2\right) \left(\|\tilde{\bm{\zeta}}^\varepsilon_\kappa-\tilde{\bm{\zeta}}^{\varepsilon,h}_\kappa\|_{\bm{H}^1_0(\omega)}+\|\tilde{\bm{\varphi}}^\varepsilon_\kappa-\tilde{\bm{\varphi}}^{\varepsilon,h}_\kappa\|_{\bm{H}^1_0(\omega)}\right)\\
&\quad \times\left(\|\tilde{\bm{\zeta}}^\varepsilon_\kappa-\Pi_h\tilde{\bm{\zeta}}^\varepsilon_\kappa\|_{\bm{H}^1_0(\omega)}+\|\tilde{\bm{\varphi}}^\varepsilon_\kappa-\Pi_h \tilde{\bm{\varphi}}^\varepsilon_\kappa\|_{\bm{H}^1_0(\omega)}\right)\\
&\quad+\dfrac{\varepsilon}{\kappa} \left\|\left(-\{(\bm{\theta}+\tilde{\bm{\zeta}}^\varepsilon_\kappa)\cdot\bm{q}\}^{-} \bm{q}\right)-\left(-\{(\bm{\theta}+\tilde{\bm{\zeta}}^{\varepsilon,h}_\kappa)\cdot\bm{q}\}^{-} \bm{q}\right)\right\|_{\bm{L}^2(\omega)} \|\tilde{\bm{\zeta}}^\varepsilon_\kappa-\Pi_h \tilde{\bm{\zeta}}^\varepsilon_\kappa\|_{\bm{H}^1_0(\omega)}\\
&\quad+\frac{\varepsilon}{\kappa}\max_{1\le i \le 3}\|\bm{a}_i\|_{\bm{\mathcal{C}}^1(\overline{\omega})}\left(\|\tilde{\bm{\varphi}}^\varepsilon_\kappa-\tilde{\bm{\varphi}}^{\varepsilon,h}_\kappa\|_{\bm{H}^1_0(\omega)}+\|\tilde{\bm{\zeta}}^\varepsilon_\kappa-\tilde{\bm{\zeta}}^{\varepsilon,h}_\kappa\|_{\bm{H}^1_0(\omega)}\right)
\left(\|\tilde{\bm{\varphi}}^\varepsilon_\kappa-\Pi_h\tilde{\bm{\varphi}}^\varepsilon_\kappa\|_{\bm{H}^1_0(\omega)}+\|\tilde{\bm{\zeta}}^\varepsilon_\kappa-\Pi_h\tilde{\bm{\zeta}}^\varepsilon_\kappa\|_{\bm{H}^1_0(\omega)}\right)\\
&\le m \tilde{C}_0^2\dfrac{\varepsilon^3}{6} \|\tilde{\bm{\zeta}}^\varepsilon_\kappa-\tilde{\bm{\zeta}}^{\varepsilon,h}_\kappa\|_{\bm{H}^1_0(\omega)}^2
+m \tilde{C}_0^2\dfrac{\varepsilon^3}{8} \|\tilde{\bm{\varphi}}^\varepsilon_\kappa-\tilde{\bm{\varphi}}^{\varepsilon,h}_\kappa\|_{\bm{H}^1_0(\omega)}^2\\
&\quad+\dfrac{\varepsilon}{2\kappa} \left\|\left(-\{(\bm{\theta}+\tilde{\bm{\zeta}}^\varepsilon_\kappa)\cdot\bm{q}\}^{-} \bm{q}\right)-\left(-\{(\bm{\theta}+\tilde{\bm{\zeta}}^{\varepsilon,h}_\kappa)\cdot\bm{q}\}^{-} \bm{q}\right)\right\|_{\bm{L}^2(\omega)}^2\\
&\quad+\left(\dfrac{6 M^2 a_1}{m \varepsilon \tilde{C}_0^2}+\dfrac{8M^2 a_1 \varepsilon^3 \max_{1\le i \le 3}\|\bm{a}_i\|_{\bm{\mathcal{C}}^1(\overline{\omega})}^4}{3m \tilde{C}_0^2}+\dfrac{12 \max_{1\le i \le 3}\|\bm{a}_i\|_{\bm{\mathcal{C}}^1(\overline{\omega})}^4}{\varepsilon m \tilde{C}_0^2 \kappa^2}+\dfrac{\varepsilon}{2\kappa}\right) \|\tilde{\bm{\zeta}}^\varepsilon_\kappa-\Pi_h \tilde{\bm{\zeta}}^\varepsilon_\kappa\|_{\bm{H}^1_0(\omega)}^2\\
&\quad+\left(\dfrac{8M^2 a_1 \varepsilon^3 \max_{1\le i \le 3}\|\bm{a}_i\|_{\bm{\mathcal{C}}^1(\overline{\omega})}^4}{3m \tilde{C}_0^2}+\dfrac{12 \max_{1\le i \le 3}\|\bm{a}_i\|_{\bm{\mathcal{C}}^1(\overline{\omega})}^4}{\varepsilon m \tilde{C}_0^2 \kappa^2}\right) \|\tilde{\bm{\varphi}}^\varepsilon_\kappa-\Pi_h \tilde{\bm{\varphi}}^\varepsilon_\kappa\|_{\bm{H}^1_0(\omega)}^2,
\end{aligned}
\end{equation*}
where the first estimate is due to Lemma~3.3 of~\cite{Blouza2016}, the second estimate holds thanks to H\"older's inequality, the third estimate holds thanks to Young's inequality~\cite{Young1912} as well as the inequality~\eqref{Pih}.

Thanks to Theorem~\ref{aug:int} (cf.~\eqref{conclusion-2}) and Theorem~\ref{aug:bdry}, according to which the penalty vanishes near the boundary $\gamma$, we have that the following estimate holds
\begin{equation}
\label{final-estimate}
\begin{aligned}
&\dfrac{m \tilde{C}_0^2 \varepsilon^3}{6} \|\tilde{\bm{\zeta}}^\varepsilon_\kappa-\tilde{\bm{\zeta}}^{\varepsilon,h}_\kappa\|_{\bm{H}^1_0(\omega)}^2
+\dfrac{5 m \tilde{C}_0^2 \varepsilon^3}{24}\|\tilde{\bm{\varphi}}^\varepsilon_\kappa-\tilde{\bm{\varphi}}^{\varepsilon,h}_\kappa\|_{\bm{H}^1_0(\omega)}^2\\
&\quad+\dfrac{\varepsilon}{2\kappa}\left\|\left(-\{(\bm{\theta}+\tilde{\bm{\zeta}}^\varepsilon_\kappa)\cdot\bm{q}\}^{-} \bm{q}\right)-\left(-\{(\bm{\theta}+\tilde{\bm{\zeta}}^{\varepsilon,h}_\kappa)\cdot\bm{q}\}^{-} \bm{q}\right)\right\|_{\bm{L}^2(\omega)}^2\\
&\le\left(\dfrac{6 M^2 a_1}{m \varepsilon \tilde{C}_0^2}+\dfrac{8M^2 a_1 \varepsilon^3 \max_{1\le i \le 3}\|\bm{a}_i\|_{\bm{\mathcal{C}}^1(\overline{\omega})}^4}{3m \tilde{C}_0^2}+\dfrac{12 \max_{1\le i \le 3}\|\bm{a}_i\|_{\bm{\mathcal{C}}^1(\overline{\omega})}^4}{\varepsilon m \tilde{C}_0^2 \kappa^2}+\dfrac{\varepsilon}{2\kappa}\right) \dfrac{\hat{C}_0^2 h^2}{\kappa\varepsilon^{10}}\\
&\quad+\left(\dfrac{8M^2 a_1 \varepsilon^3 \max_{1\le i \le 3}\|\bm{a}_i\|_{\bm{\mathcal{C}}^1(\overline{\omega})}^4}{3m \tilde{C}_0^2}+\dfrac{12 \max_{1\le i \le 3}\|\bm{a}_i\|_{\bm{\mathcal{C}}^1(\overline{\omega})}^4}{\varepsilon m \tilde{C}_0^2 \kappa^2}\right) \dfrac{\hat{C}_0^2 h^2}{\kappa\varepsilon^{10}},
\end{aligned}
\end{equation}
where $\hat{C}_0>0$ is independent of $\varepsilon$, $\kappa$ and $h$. The denominator in terms of powers of $\varepsilon$ and $\kappa$ is due to~\eqref{Pih}, \eqref{conc-2} and formula~(3.7) in Lemma~3.2 on page~263 of~\cite{Nec67}.

Hence, if we keep $\varepsilon$ and $\kappa$ fixed, and we let $h \to 0^+$, then the desired conclusion follows and the proof is complete.
\end{proof}

As a remark, we observe that the convergence in Theorem~\ref{th:conv} is not, in general, independent of $\varepsilon$ and $\kappa$.

If we specialise $\kappa$ as follows,
\begin{equation*}
\kappa=\dfrac{h^{q/3}}{\varepsilon^{11}}, \quad\textup{ with } 0\le q <2,
\end{equation*}
we have that $\kappa=\kappa(h) \to 0^+$ as $h\to 0^+$, and that the right-hand side of~\eqref{final-estimate} becomes independent of $\varepsilon$ and $\kappa$.

We note in passing that, differently from~\cite{Scholz1987}, in the proof of Theorem~\ref{th:conv} we exploit the augmented regularity in a different way as the vectorial nature of the problem prevents form straightforwardly applying the estimates that led to the conclusion in~\cite{Scholz1987}.

\section{Numerical Simulations}
\label{numerics}

In this last section of the paper, we implement numerical simulations aiming to test the convergence of the algorithms presented in section~\ref{mixed:penalty} and in section~\ref{approx:penalty}.

Let $R>0$ be given. We consider as a domain a circle of radius $r_A:=\frac{R}{2}$, and we denote one such domain by $\omega$:
\begin{equation*}
\omega:=\left\{y=(y_\alpha)\in \mathbb{R}^2;\sqrt{y_1^2+y_2^2}<r_A\right\}.
\end{equation*}

The middle surface of the membrane shell under consideration is a non-hemispherical spherical cap which is not in contact with the plane $\{x_3=0\}$. The parametrization we choose is $\bm{\theta} \in \mathcal{C}^4(\overline{\omega};\mathbb{E}^3)$ defined by:
\begin{equation}
\label{middlesurf}
\bm{\theta}(y):=\left(y_1, y_2, \sqrt{R^2-y_1^2-y_2^2}-0.85\right),\quad\textup{ for all } y=(y_\alpha) \in \overline{\omega}.
\end{equation}

Throughout this section, the values of $\varepsilon$, $\lambda$, $\mu$ and $R$ are fixed once and for all as follows:
\begin{equation*}
	\begin{aligned}
		\varepsilon&=0.001,\\
		\lambda&=0.4,\\
		\mu&=0.012,\\
		R&=1.0.
	\end{aligned}
\end{equation*}

We define, once and for all, the unit-norm vector $\bm{q}$ orthogonal to the plane $\{x_3=0\}$ constituting the obstacle by $\bm{q}=(0,0,1)$.
The applied body force density $\bm{p}^\varepsilon=(p^{i,\varepsilon})$ entering the first two batches of experiments is given by $\bm{p}^\varepsilon=(0,0,g(y))$, where
$$
g(y):=
\begin{cases}
-\frac{\varepsilon}{10}(-5.0 y_1^2-5.0 y_2^2+0.295), &\textup{ if } |y|^2< 0.060,\\
0, &\textup{otherwise}.
\end{cases}
$$

The expressions of the geometrical parameters (i.e., the covariant and contravariant bases, the first fundamental form in covariant and contravariant components, the second fundamental form in covariant and mixed components, etc.) associated with the middle surface~\eqref{middlesurf} were computed by means of a specific MATLAB library, which was used for generating the numerical experiments presented in the recent paper~\cite{DPSY2023}.
The numerical simulations are performed by means of the software FEniCS~\cite{Fenics2016} and the visualization is performed by means of the software ParaView~\cite{Ahrens2005}.
The plots were created by means of the \verb*|matplotlib| libraries from a Python~3.9.8 installation.

The first batch of numerical experiments is meant to validate the claim of Theorem~\ref{t:convergence}.  After fixing the mesh size $0<h<<1$, we let $\kappa$ tend to zero in Problem~\ref{problem3}. Let $(\tilde{\bm{\zeta}}^{\varepsilon,h}_{\kappa_1},\tilde{\bm{\varphi}}^{\varepsilon,h}_{\kappa_1})$ and $(\tilde{\bm{\zeta}}^{\varepsilon,h}_{\kappa_2},\tilde{\bm{\varphi}}^{\varepsilon,h}_{\kappa_2})$ be two solutions of Problem~\ref{problem3}, where $\kappa_2:=2\kappa_1$ and $\kappa_1>0$.
Each component $(\tilde{\zeta}^{\varepsilon}_{\kappa,i},\tilde{\varphi}^{\varepsilon}_{\kappa,i})$ of the solution of Problem~\ref{problem2-1} is discretised by Lagrange triangles (cf., e.g., \cite{PGCFEM}) and homogeneous Dirichlet boundary conditions are imposed for all the components. 
At each iteration, Problem~\ref{problem3} is solved by Newton's method. The algorithm stops when the error is smaller than $5.0 \times 10^{-8}$.

\begin{figure}[H]
	\centering
	\subfloat[Error convergence.]{
		\begin{tabular}{lcr}
			\toprule
			$\kappa$ && Error  \\
			\midrule
			1.1e-09&&6.0e-04\\
			5.9e-10&&3.1e-04\\
			2.9e-10&&1.6e-04\\
			3.7e-11&&1.0e-05\\
			4.6e-12&&1.2e-06\\
			5.8e-13&&1.5e-07\\
			1.4e-13&& 3.9e-08\\
			\bottomrule
		\end{tabular}
	}
\hspace{0.25cm}
	\begin{subfigure}[t]{0.35\linewidth}
		\includegraphics[width=1.0\linewidth]{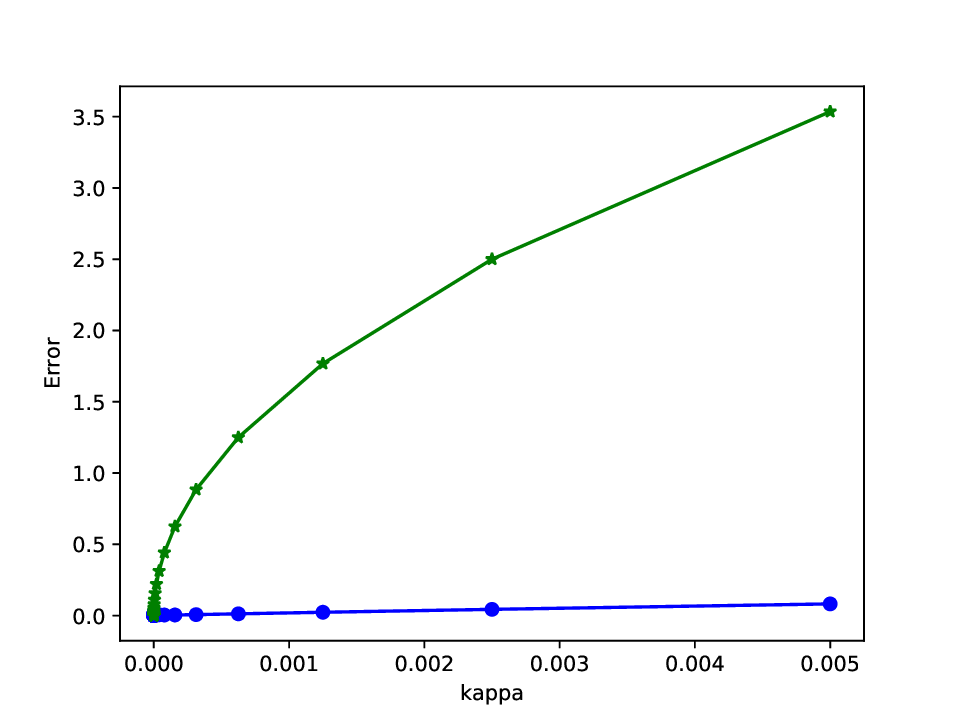}
		\subcaption{Error convergence as $\kappa \to 0^+$ for $h=0.01584901098161906$. Original figure.}
	\end{subfigure}%
	\hspace{0.25cm}
	\begin{subfigure}[t]{0.35\linewidth}
		\includegraphics[width=1.0\linewidth]{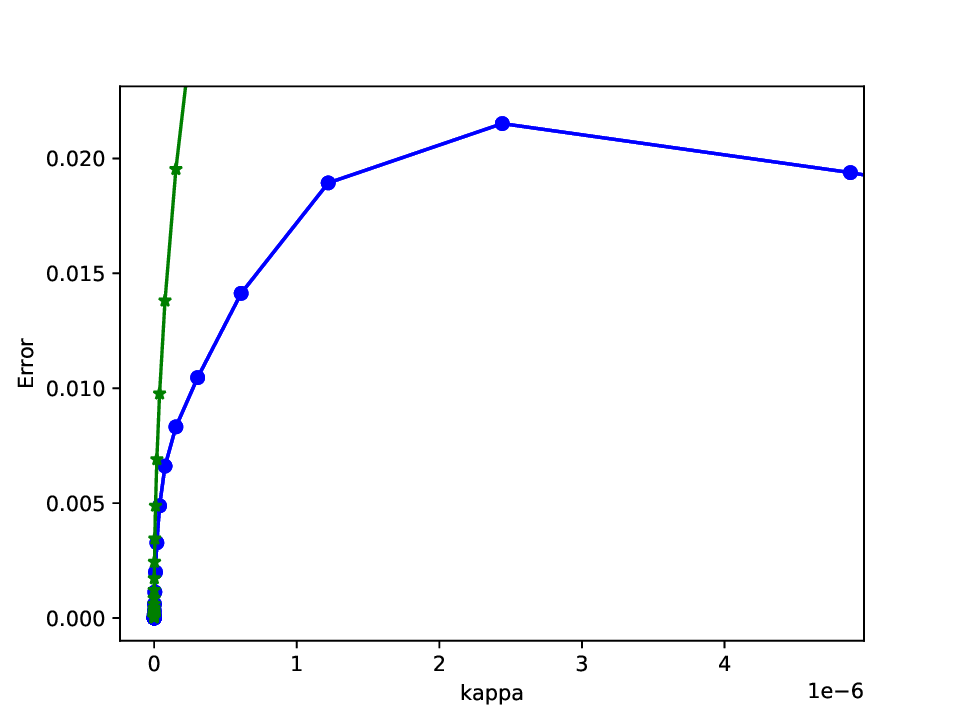}
		\subcaption{Error convergence as $\kappa \to 0^+$ for $h=0.01584901098161906$. The original figure was enlarged near the origin.}
	\end{subfigure}%
\end{figure}

\begin{figure}[H]
	\ContinuedFloat
	\centering
	\subfloat[Error convergence.]{
		\begin{tabular}{lcr}
			\toprule
			$\kappa$ && Error  \\
			\midrule
			1.1e-09&&1.5e-03\\
			1.4e-10&&2.5e-04\\
			7.4e-11&&1.3e-04\\
			9.3e-12&&1.6e-05\\
			5.8e-13&&1.0e-06\\
			7.2e-14&&1.3e-07\\
			1.8e-14&&3.2e-08\\
			\bottomrule
		\end{tabular}
	}
	\begin{subfigure}[t]{0.35\linewidth}
		\includegraphics[width=1.0\linewidth]{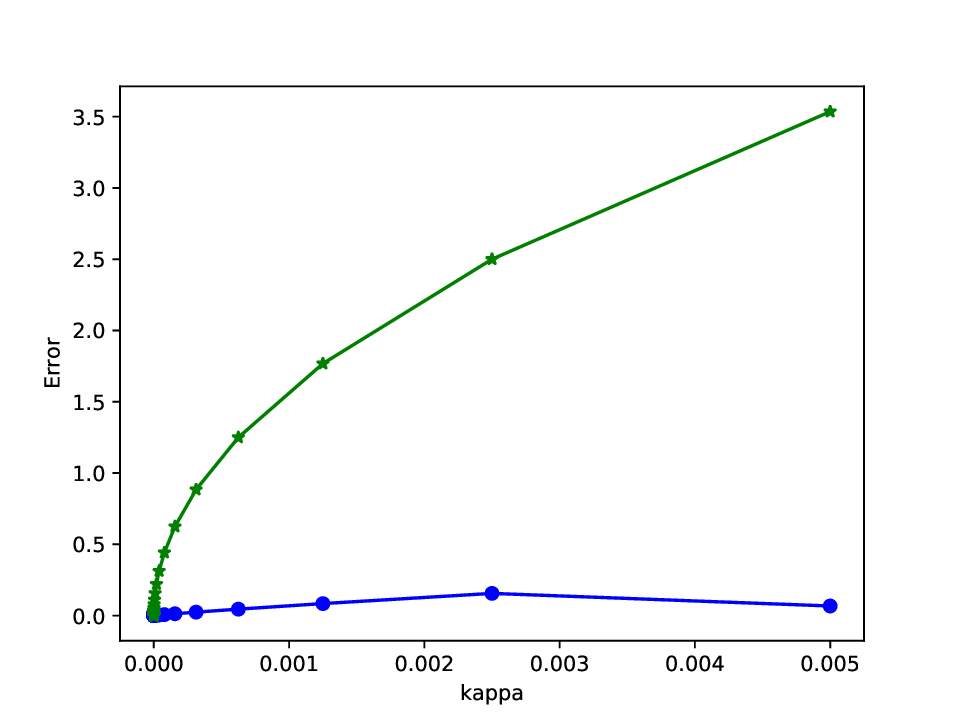}
		\subcaption{Error convergence as $\kappa \to 0^+$ for $h=0.007986374133406946$. Original figure.}
	\end{subfigure}%
	\hspace{0.5cm}
	\begin{subfigure}[t]{0.35\linewidth}
		\includegraphics[width=1.0\linewidth]{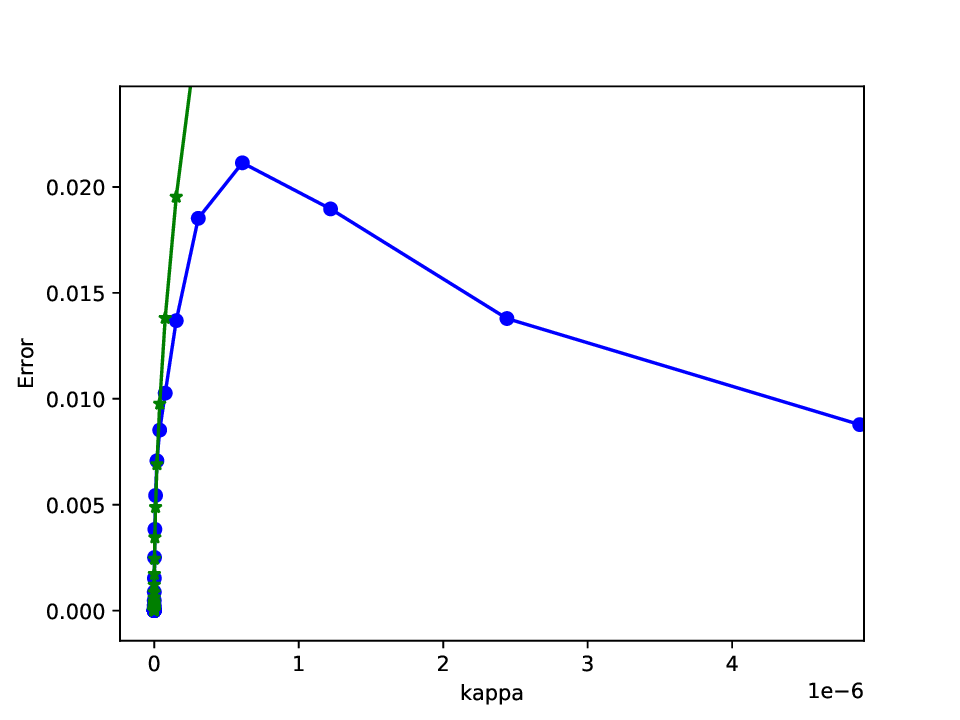}
		\subcaption{Error convergence as $\kappa \to 0^+$ for $h=0.007986374133406946$. This is an enlargement near the origin of the original figure.}
	\end{subfigure}%
\end{figure}

\begin{figure}[H]
	\ContinuedFloat
	\centering
	\subfloat[Error convergence.]{
		\begin{tabular}{lcr}
			\toprule
			$\kappa$ && Error  \\
			\midrule
			1.4e-10&&8.4e-04\\
			1.8e-11&&1.2e-04\\
			2.3e-12&&1.6e-05\\
			1.4e-13&&1.0e-06\\
			1.8e-14&&1.2e-07\\
			9.0e-15&&6.3e-08\\
			4.5e-15&&3.1e-08\\
			\bottomrule
		\end{tabular}
	}
	\begin{subfigure}[t]{0.35\linewidth}
		\includegraphics[width=1.0\linewidth]{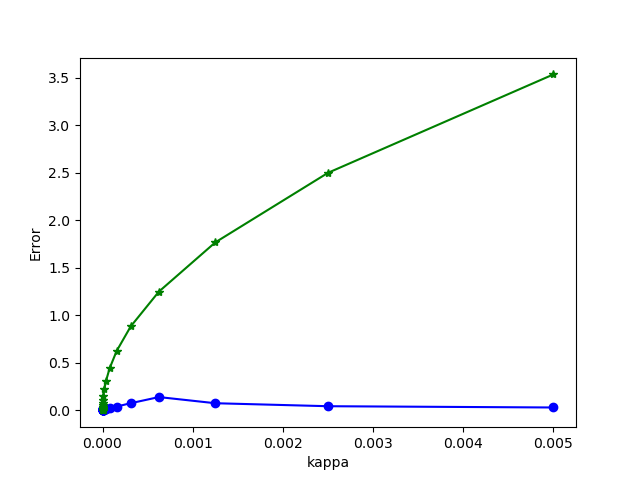}
		\subcaption{Error convergence as $\kappa \to 0^+$ for $h=0.004199300993662772$. Original figure.}
	\end{subfigure}%
	\hspace{0.5cm}
	\begin{subfigure}[t]{0.35\linewidth}
		\includegraphics[width=1.0\linewidth]{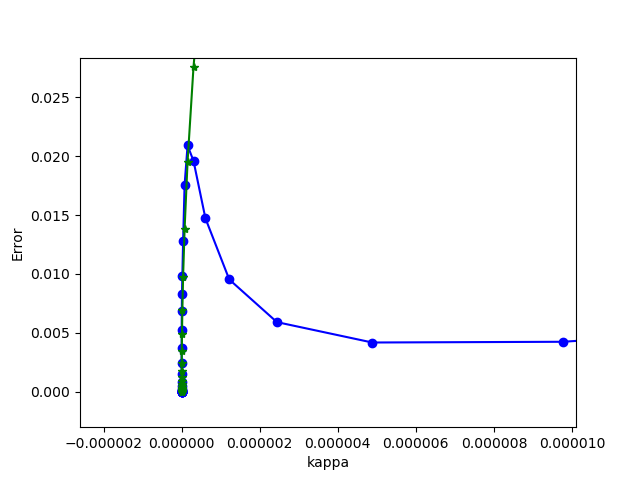}
		\subcaption{Error convergence as $\kappa \to 0^+$ for $h=0.004199300993662772$. This is an enlargement near the origin of the original figure.}
	\end{subfigure}%
\end{figure}

\begin{figure}[H]
	\ContinuedFloat
	\centering
	\subfloat[Error convergence.]{
		\begin{tabular}{lcr}
			\toprule
			$\kappa$ && Error  \\
			\midrule
			1.4e-10&&9.3e-04\\
			1.8e-11&&1.4e-04\\
			2.3e-12&&1.8e-05\\
			1.4e-13&&1.1e-06\\
			1.8e-14&&1.4e-07\\
			9.0e-15&&7.4e-08\\
			4.5e-15&&3.7e-08\\
			\bottomrule
		\end{tabular}
	}
	\begin{subfigure}[t]{0.35\linewidth}
		\includegraphics[width=1.0\linewidth]{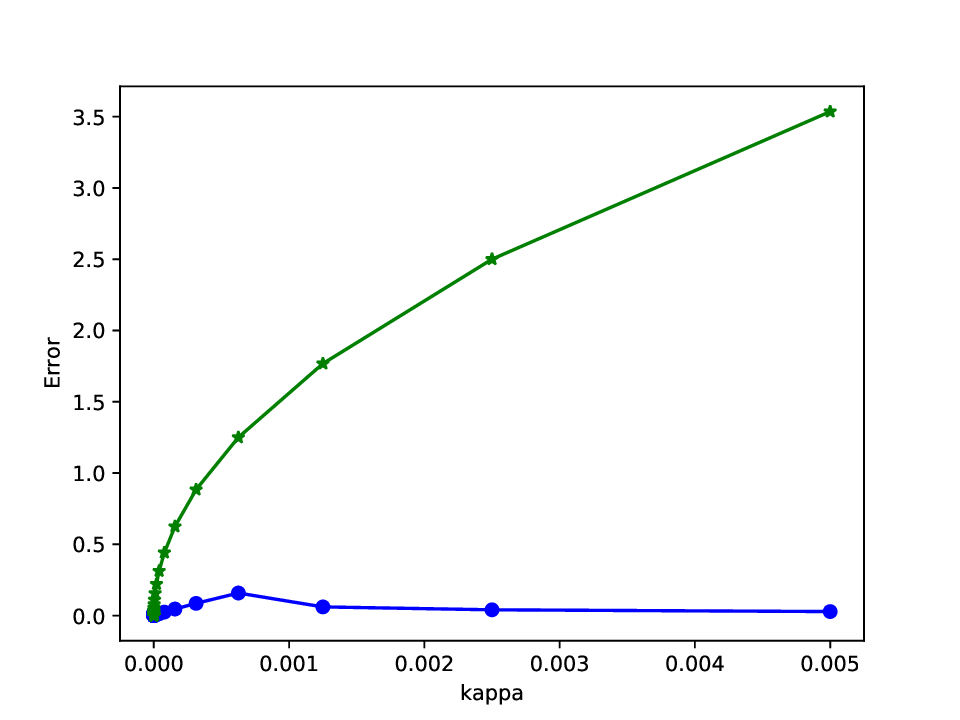}
		\subcaption{Error convergence as $\kappa \to 0^+$ for $h=0.003922444002155661$. Original figure.}
	\end{subfigure}%
	\hspace{0.5cm}
	\begin{subfigure}[t]{0.35\linewidth}
		\includegraphics[width=1.0\linewidth]{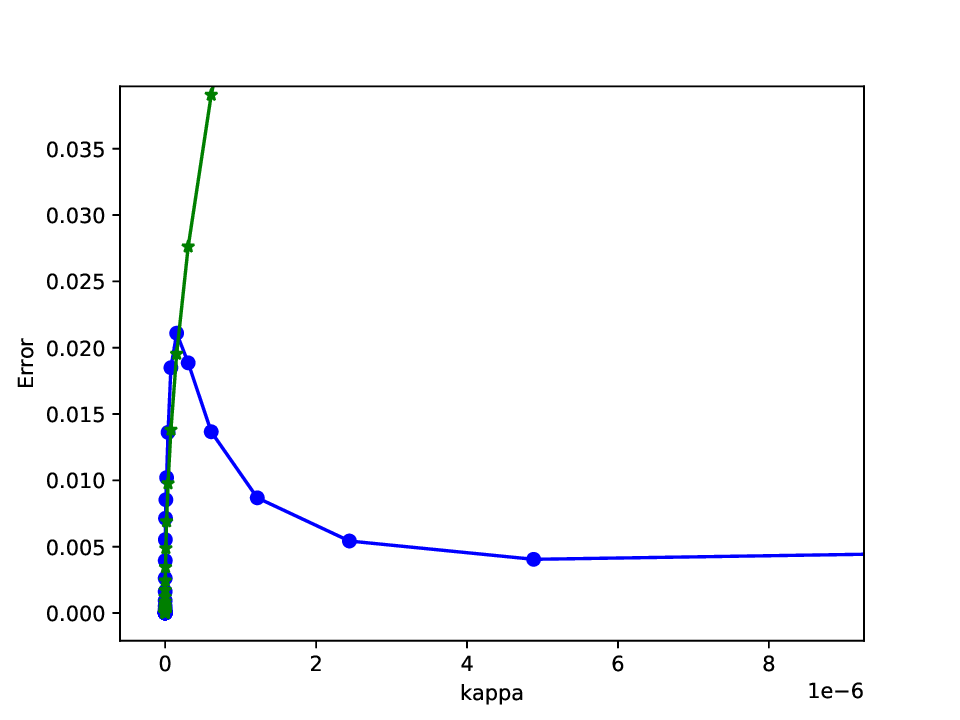}
		\subcaption{Error convergence as $\kappa \to 0^+$ for $h=0.003922444002155661$. This is an enlargement near the origin of the original figure.}
	\end{subfigure}%
\end{figure}

\begin{figure}[H]
	\ContinuedFloat
\centering
\subfloat[Error convergence.]{
	\begin{tabular}{lcr}
		\toprule
		$\kappa$ && Error  \\
		\midrule
		1.4e-10&&1.2e-03\\
		1.8e-11&&1.9e-04\\
		1.1e-12&&1.2e-05\\
		1.4e-13&&1.6e-06\\
		9.0e-15&&1.0e-07\\
		4.5e-15&&5.1e-08\\
		2.2e-15&&2.5e-08\\
		\bottomrule
	\end{tabular}
}
\begin{subfigure}[t]{0.35\linewidth}
	\includegraphics[width=1.0\linewidth]{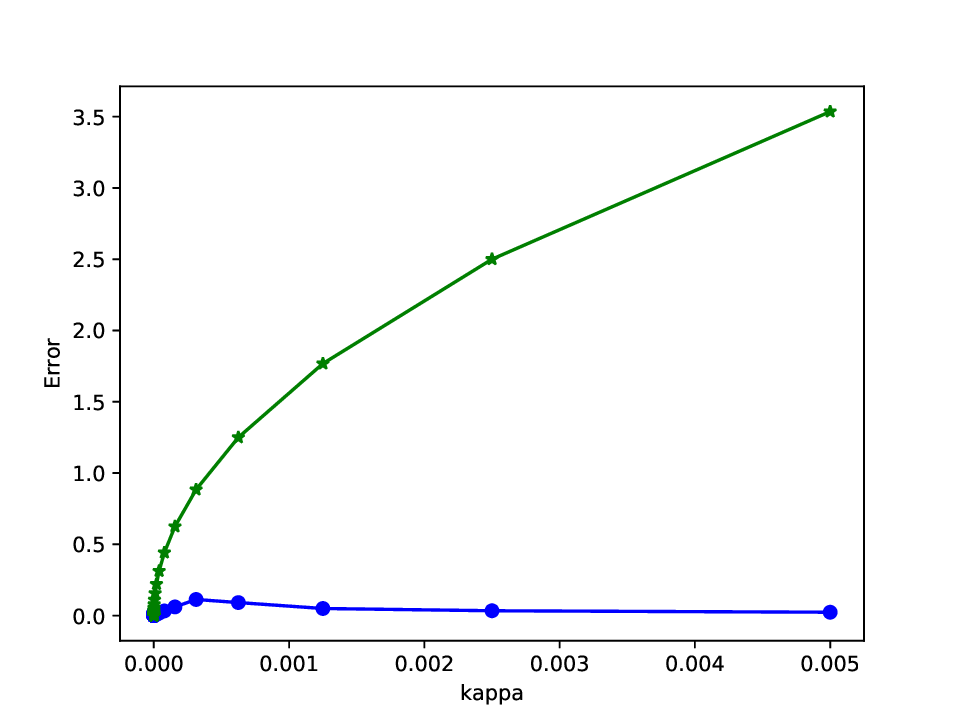}
	\subcaption{Error convergence as $\kappa \to 0^+$ for $h=0.0033478747001890077$. Original figure.}
\end{subfigure}%
\hspace{0.5cm}
\begin{subfigure}[t]{0.35\linewidth}
	\includegraphics[width=1.0\linewidth]{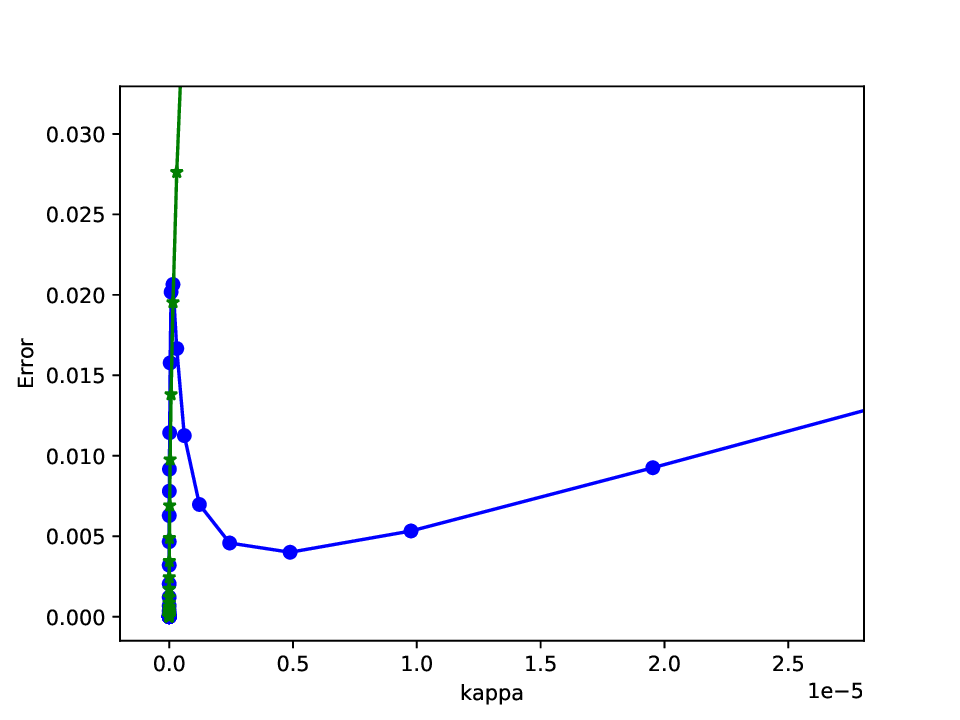}
	\subcaption{Error convergence as $\kappa \to 0^+$ for $h=0.0033478747001890077$. This is an enlargement near the origin of the original figure.}
\end{subfigure}%
\caption{Given $0<h<<1$, the first component of the solution $(\tilde{\bm{\zeta}}^{\varepsilon,h}_\kappa,\tilde{\bm{\varphi}}^{\varepsilon,h}_\kappa)$ of Problem~\ref{problem3} converges with respect to the standard norm of $\bm{H}^1_0(\omega)$ as $\kappa\to0^+$.}
\label{fig:1}
\end{figure}

From the data patterns in Figure~\ref{fig:1} we observe that, for a given mesh size $h$, the solution $(\tilde{\bm{\zeta}}^{\varepsilon,h}_\kappa,\tilde{\bm{\varphi}}^{\varepsilon,h}_\kappa)$ of Problem~\ref{problem3} converges as $\kappa \to 0^+$. This is coherent with the conclusion of Theorem~\ref{t:convergence}. Moreover, we observe that the convergence of the error residual curve depicted in blue follows the pattern of Theorem~\ref{difference-1}, as the curve of error residuals is always below the square root-like graph depicted in green.

The second batch of numerical experiments is meant to validate the claim of Theorem~\ref{th:conv}. 
We show that, for a fixed $0< q <2/3$ as in Theorem~\ref{th:conv}, the error residual $\|\tilde{\bm{\zeta}}^{\varepsilon,h_1}_{h_1^{q}}-\tilde{\bm{\zeta}}^{\varepsilon,h_2}_{h_2^{q}}\|_{\bm{H}^1_0(\omega)}$ tends to zero as $h_1,h_2 \to 0^+$.
The algorithm stops when the error residual of the Cauchy sequence is smaller than $6.0 \times 10^{-5}$.

Once again, each component of the solution of Problem~\ref{problem2-1} is discretised by Lagrange triangles (cf., e.g., \cite{PGCFEM}) and homogeneous Dirichlet boundary conditions are imposed for all the components and, at each iteration, Problem~\ref{problem3} is solved by Newton's method.

The results of these experiments are reported in Figure~\ref{fig:5} below.

\begin{figure}[H]
	\centering
	\begin{subfigure}[t]{0.35\linewidth}
		\includegraphics[width=1.0\linewidth]{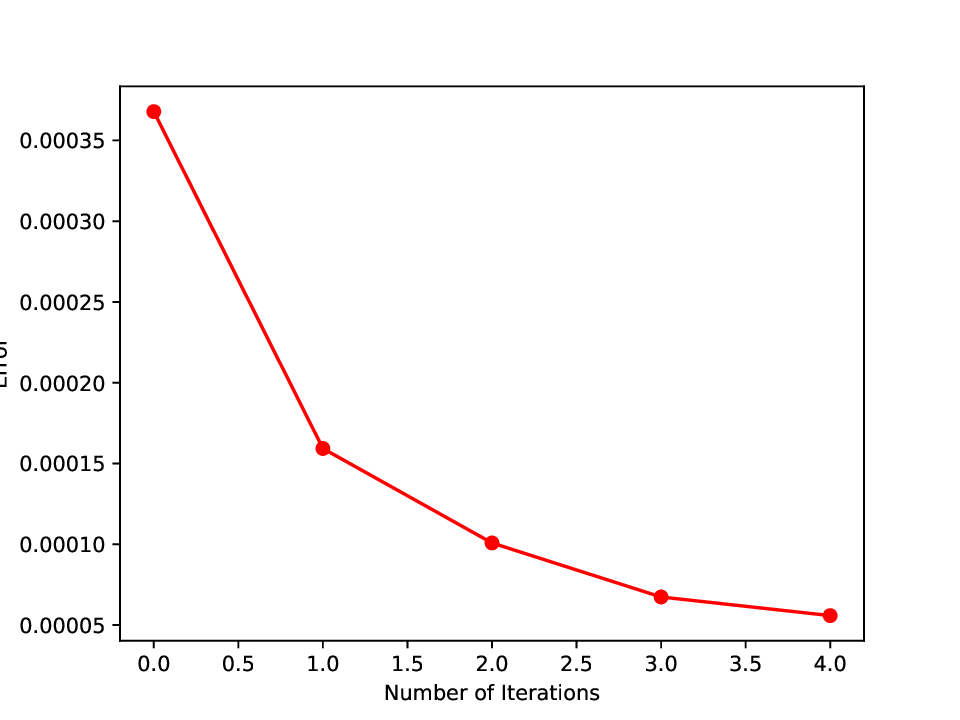}
		\subcaption{For $q=0.4$ the algorithm stops after four iterations}
	\end{subfigure}%
	\hspace{0.5cm}
	\begin{subfigure}[t]{0.35\linewidth}
		\includegraphics[width=1.0\linewidth]{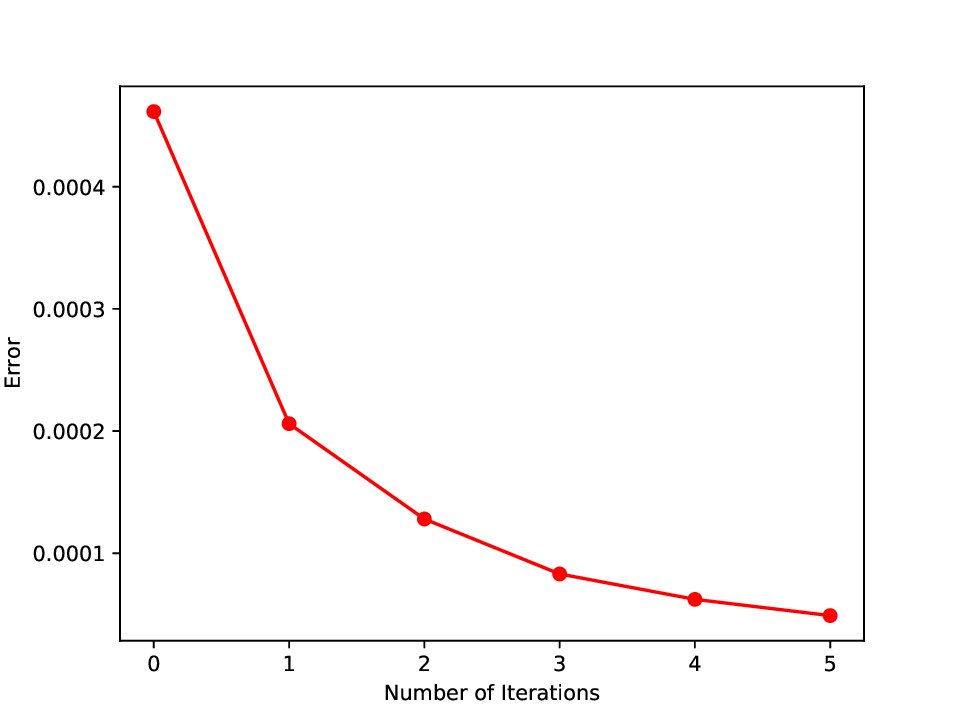}
		\subcaption{For $q=0.5$ the algorithm stops after five iterations}
	\end{subfigure}%
\end{figure}

\begin{figure}[H]
	\ContinuedFloat
	\centering
	\begin{subfigure}[t]{0.35\linewidth}
		\includegraphics[width=1.0\linewidth]{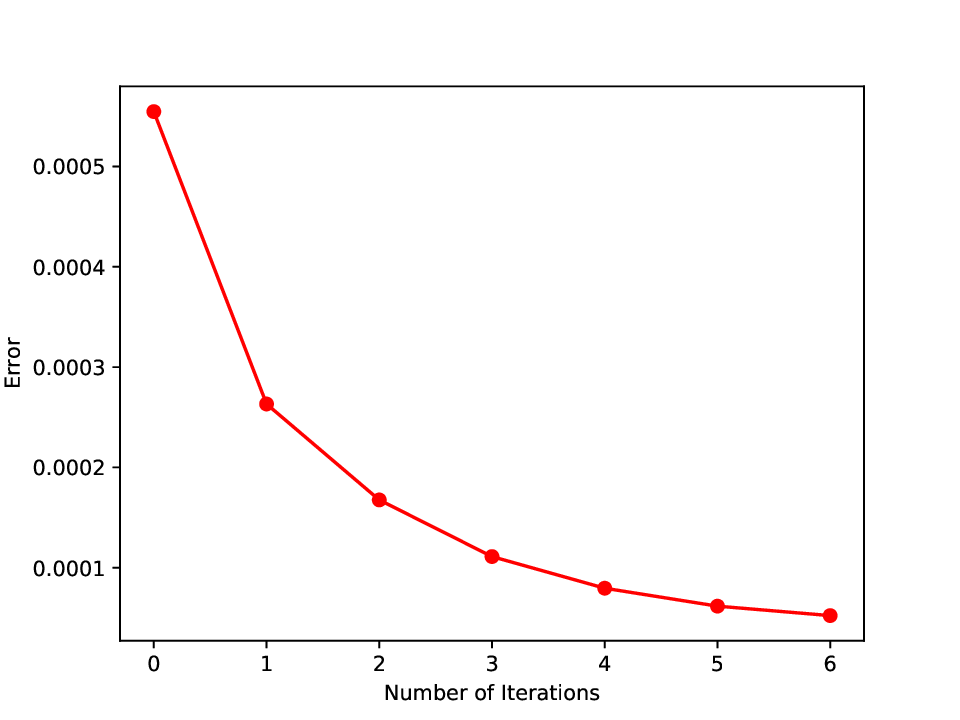}
		\subcaption{For $q=0.6$ the algorithm stops after six iterations}
	\end{subfigure}%
	\hspace{0.5cm}
	\begin{subfigure}[t]{0.35\linewidth}
	\includegraphics[width=1.0\linewidth]{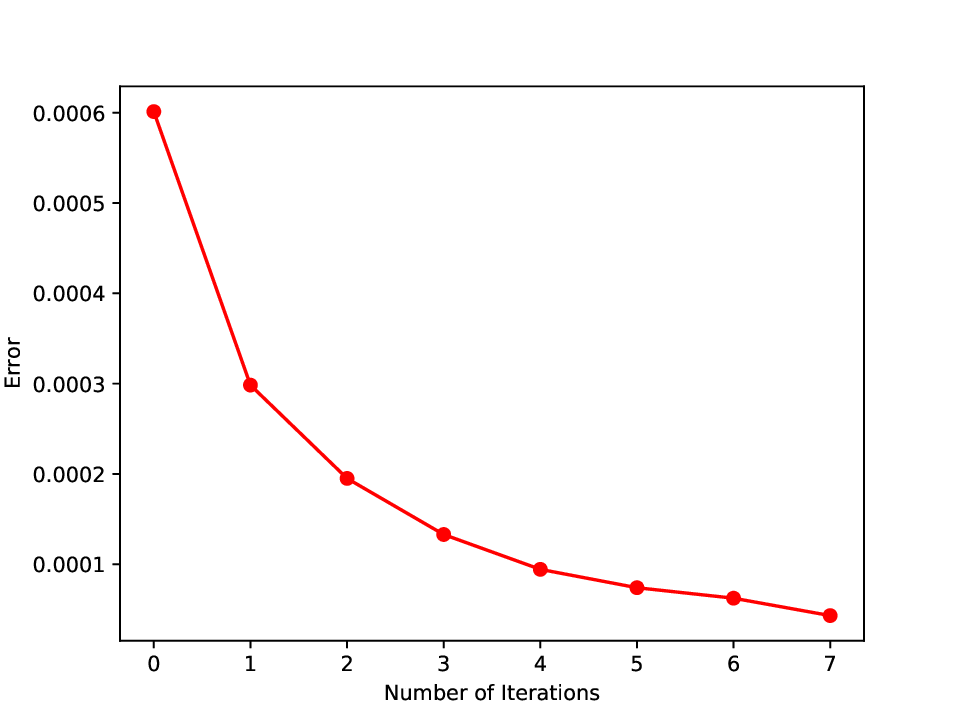}
		\subcaption{For $q=0.66$ the algorithm stops after seven iterations}
	\end{subfigure}%
	\caption{Given $0<q<2/3$ as in Theorem~\ref{th:conv}, the error $\|\tilde{\bm{\zeta}}^{\varepsilon,h_1}_{h_1^{q}}-\tilde{\bm{\zeta}}^{\varepsilon,h_2}_{h_2^{q}}\|_{\bm{H}^1_0(\omega)}$ converges to zero as $h_1, h_2\to0^+$. As $q$ increases, the number of iterations needed to meet the stopping criterion of the Cauchy sequence increases.}
	\label{fig:5}
\end{figure}

The third batch of numerical experiments validates the genuineness of the model from the qualitative point of view.
We observe that the presented data exhibits the pattern that, for a fixed $0<h<<1$ and a fixed $0< q <2/3$, the contact area increases as the applied body force intensity increases.

For the third batch of experiments, the applied body force density $\bm{p}^\varepsilon=(p^{i,\varepsilon})$ entering the model is given by $\bm{p}^\varepsilon=(0,0,\varepsilon g_\ell(y))$, where $\ell$ is a nonnegative integer and
$$
g_\ell(y):=
\begin{cases}
10^4(y_1^2+ y_2^2-0.0059 \ell), &\textup{ if } |y|^2< 0.0059 \ell,\\
0, &\textup{otherwise}.
\end{cases}
$$

The results of these experiments are reported in Figure~\ref{fig:6} below.

\begin{figure}[H]
	\centering
	\begin{subfigure}[t]{0.45\linewidth}
		\includegraphics[width=\linewidth]{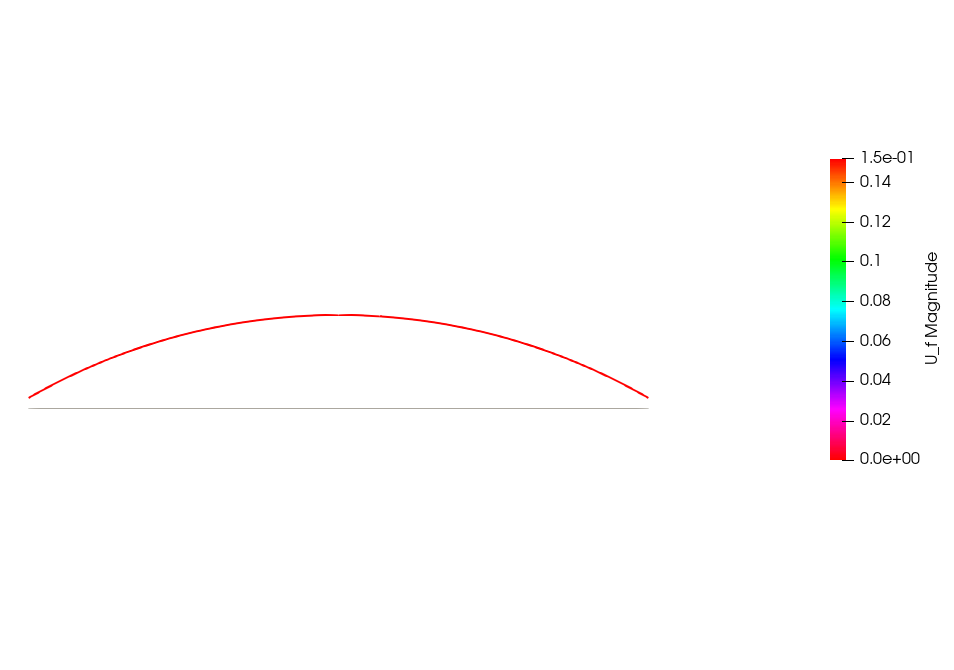}
		\subcaption{$\ell=0$}
	\end{subfigure}%
	\hspace{0.5cm}
	\begin{subfigure}[t]{0.45\linewidth}
		\includegraphics[width=1.0\linewidth]{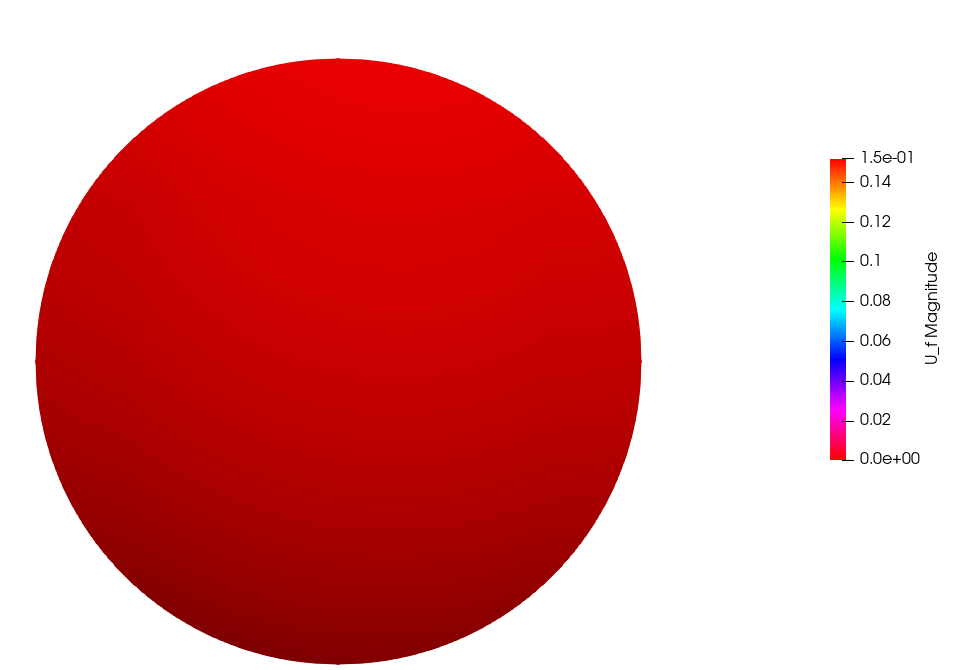}
		\subcaption{$\ell=0$}
	\end{subfigure}%
\end{figure}

\begin{figure}[H]
	\ContinuedFloat
	\centering
	\begin{subfigure}[t]{0.45\linewidth}
		\includegraphics[width=1.0\linewidth]{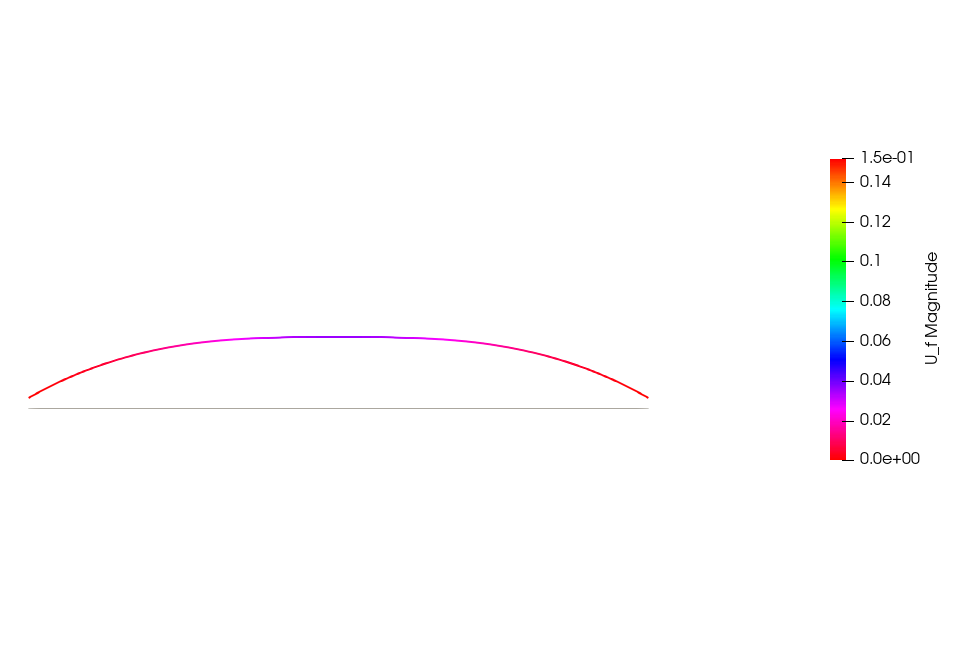}
		\subcaption{$\ell=4$}
	\end{subfigure}%
\hspace{0.5cm}
	\begin{subfigure}[t]{0.45\linewidth}
		\includegraphics[width=\linewidth]{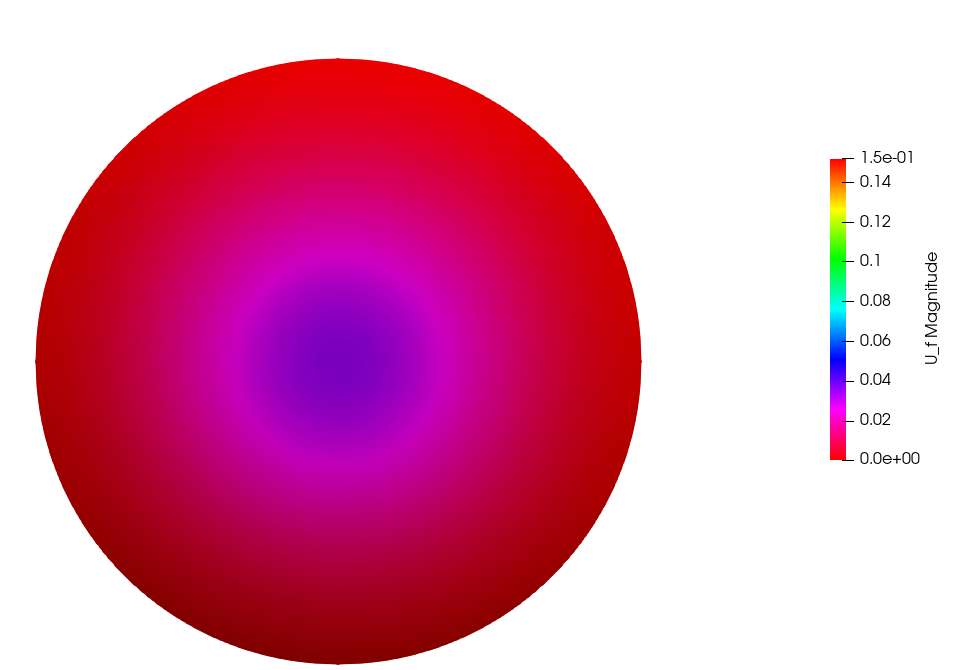}
		\subcaption{$\ell=4$}
	\end{subfigure}%
\end{figure}

\begin{figure}[H]
	\ContinuedFloat
	\centering
	\begin{subfigure}[t]{0.45\linewidth}
		\includegraphics[width=1.0\linewidth]{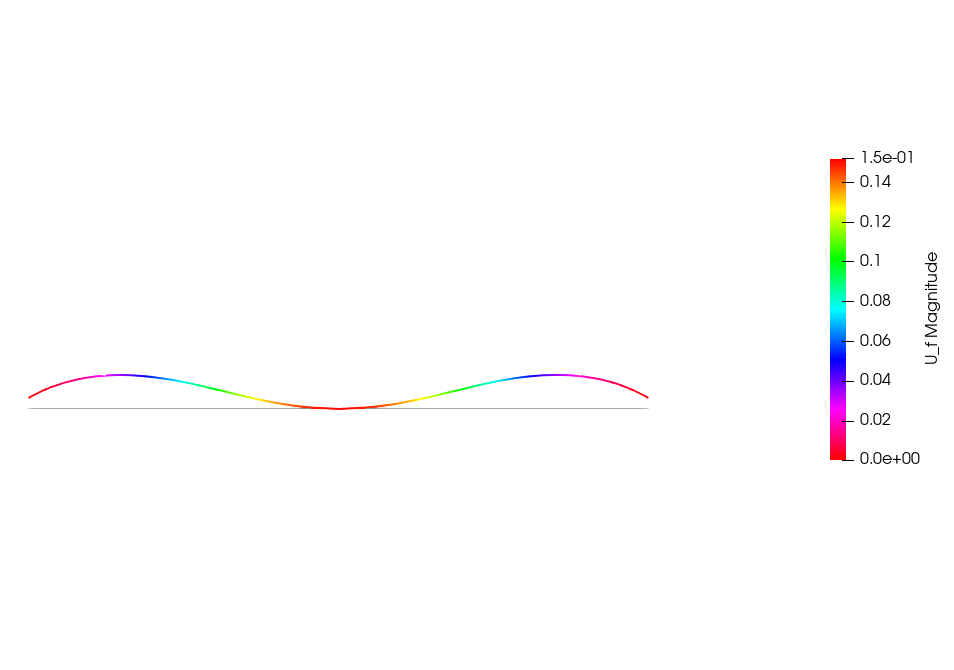}
		\subcaption{$\ell=9$}
	\end{subfigure}%
	\hspace{0.5cm}
	\begin{subfigure}[t]{0.45\linewidth}
		\includegraphics[width=\linewidth]{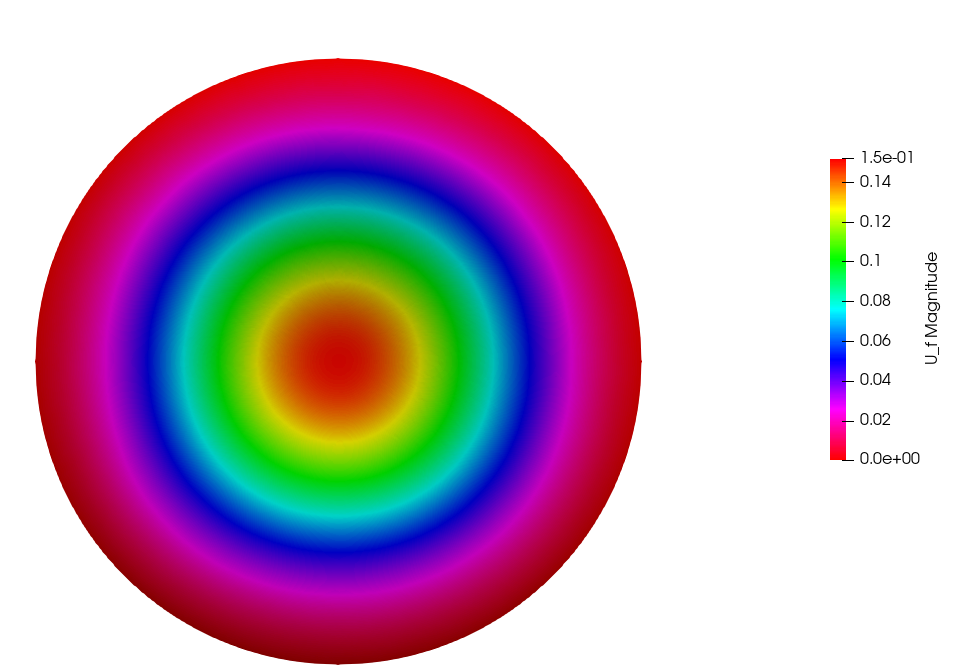}
		\subcaption{$\ell=9$}
	\end{subfigure}%
\end{figure}

\begin{figure}[H]
	\ContinuedFloat
	\centering
	\begin{subfigure}[t]{0.45\linewidth}
		\includegraphics[width=1.0\linewidth]{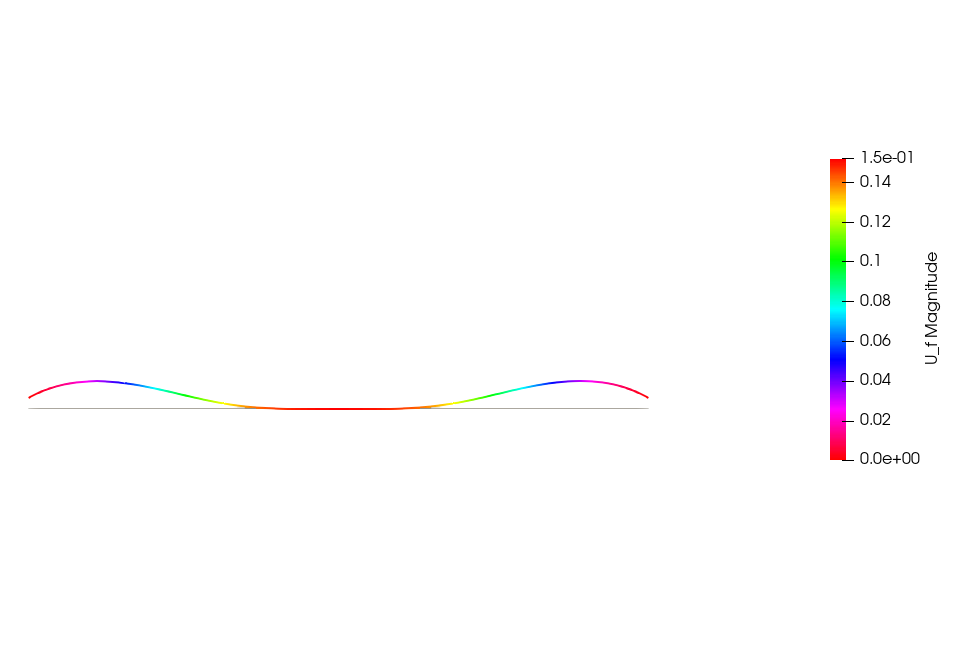}
		\subcaption{$\ell=15$}
	\end{subfigure}%
	\hspace{0.5cm}
	\begin{subfigure}[t]{0.45\linewidth}
		\includegraphics[width=\linewidth]{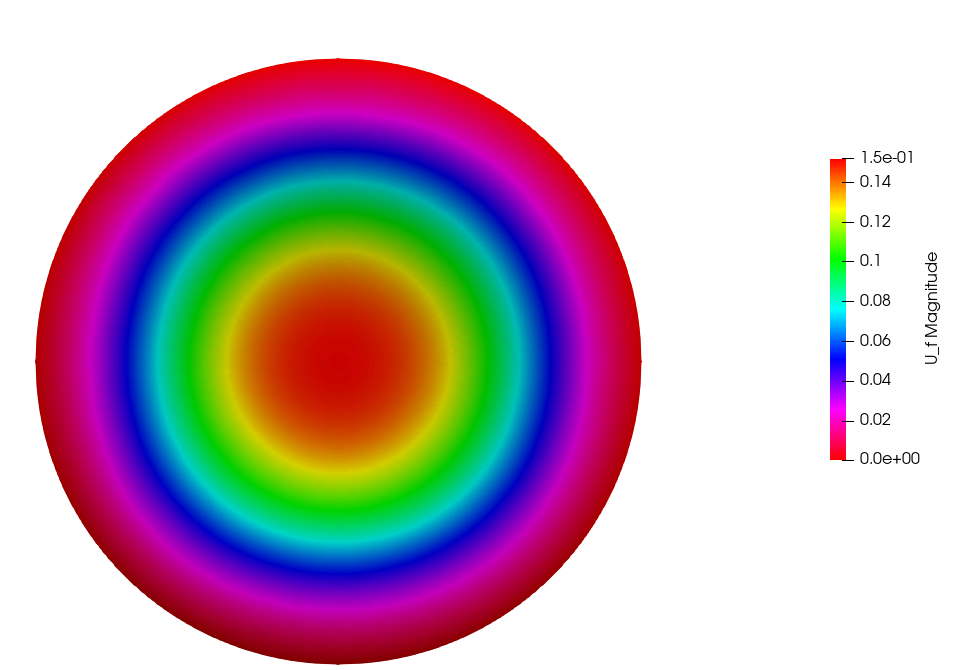}
		\subcaption{$\ell=15$}
	\end{subfigure}%
\end{figure}

%\begin{figure}[H]
%	\ContinuedFloat
%	\centering
%	\begin{subfigure}[t]{0.45\linewidth}
%		\includegraphics[width=1.0\linewidth]{./figures/ell23}
%		\subcaption{$\ell=23$}
%	\end{subfigure}%
%	\hspace{0.5cm}
%	\begin{subfigure}[t]{0.45\linewidth}
%		\includegraphics[width=\linewidth]{./figures/ell23up}
%		\subcaption{$\ell=23$}
%	\end{subfigure}%
%\end{figure}

\begin{figure}[H]
	\ContinuedFloat
	\centering
	\begin{subfigure}[t]{0.45\linewidth}
		\includegraphics[width=1.0\linewidth]{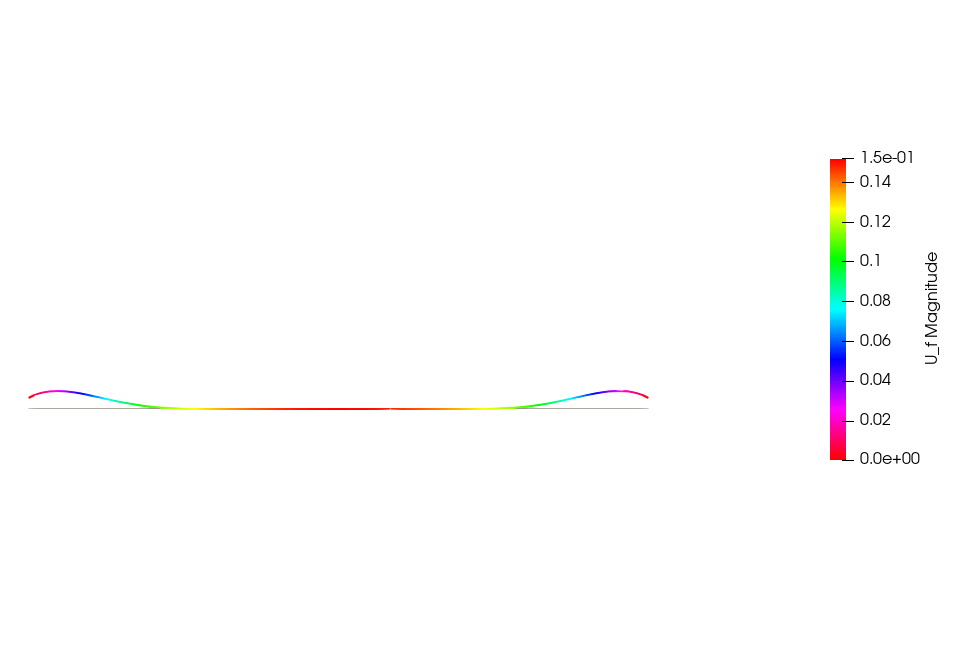}
		\subcaption{$\ell=30$}
	\end{subfigure}%
\hspace{0.5cm}
\begin{subfigure}[t]{0.45\linewidth}
	\includegraphics[width=\linewidth]{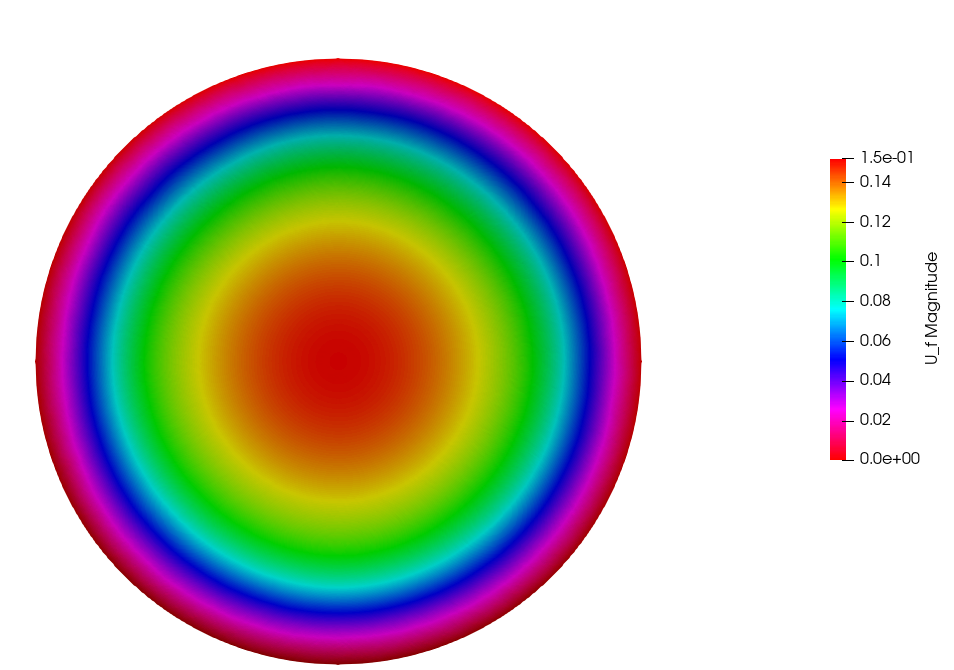}
	\subcaption{$\ell=30$}
\end{subfigure}%
	\caption{Cross sections of a deformed membrane shell subjected not to cross a given planar obstacle.
		Given $0<h<<1$ and $0<q<2/3$ we observe that as the applied body force magnitude increases the contact area increases.}
	\label{fig:6}
\end{figure}

\section*{Conclusions and Commentary}

In this paper we established the convergence of a numerical scheme based on the Finite Element Method for approximating the solution of a set of a Koiter's type variational inequality modelling the deformation of a linearly elastic elliptic membrane shell subject to remaining confined in a prescribed half space. 

After ruling out unfeasible options for carrying out the sought approximation via the Finite Element Method, we decided to opt for the intrinsic mixed formulation suggested by Blouza and his collaborators~\cite{Blouza2016}, which is based on the mathematical theory developed by Blouza \& Le Dret~\cite{BlouzaLeDret1999}.

Thanks to the intrinsic formulation proposed by Blouza \& Le Dret~\cite{BlouzaLeDret1999}, we were able to prove the convergence of the Finite Element scheme of the problem under consideration, and we were also able to implement this scheme on a computer. The numerical results obtained upon completion of three different batches of experiments corroborate the theoretical results previously derived.

\section*{Declarations}

\subsection*{Authors' Contribution}

All authors have contributed to the realisation of this manuscript in equal manner.

\subsection*{Acknowledgements}

Not applicable

\subsection*{Ethical Approval}

Not applicable.

\subsection*{Availability of Supporting Data}

Not applicable.

\subsection*{Competing Interests}

All authors certify that they have no affiliations with or involvement in any organization or entity with any competing interests in the subject matter or materials discussed in this manuscript.

\subsection*{Funding}

P.P. was partly supported by the Research Fund of Indiana University and by the Ky and Yu-Fen Fan Fund Travel Grant from the AMS (IU Award number 44-294-36 with the Simons Foundation).

X.S. was partly supported by the National Natural Science Foundation of China (NSFC.11971379), by the Distinguished Youth Foundation of Shaanxi Province (2022JC-01) and by the Ky and Yu-Fen Fan Fund Travel Grant from the AMS (IU Award number 44-294-36 with the Simons Foundation).

%\section*{Acknowledgements}
%The second author would like to express his gratitude to Professor Philippe G. Ciarlet for his encouragement and guidance.